\documentclass{article}

\usepackage{mystyle}
\usepackage{wasysym}
\begin{document}
\title{Even ordinals and the Kunen inconsistency\footnote{This research was supported by NSF Grant DMS 1902884. The author thanks Peter Koellner and Farmer Schlutzenberg for their comments on the paper.}}
\author{Gabriel Goldberg\\ Evans Hall\\ University Drive \\ Berkeley, CA 94720}
\maketitle
\begin{abstract}
This paper contributes to the theory of large cardinals beyond the Kunen
inconsistency, or choiceless large cardinal axioms, in the context where the
Axiom of Choice is not assumed. The first part of the paper investigates a
periodicity phenomenon: assuming choiceless large cardinal axioms, the
properties of the cumulative hierarchy turn out to alternate between even and
odd ranks. 
The second part of the paper explores the structure of ultrafilters under
choiceless large cardinal axioms, exploiting the fact that these axioms imply a
weak form of the author's Ultrapower Axiom \cite{UA}. 
The third and final part of the paper examines the consistency strength of
choiceless large cardinals, including a proof that assuming DC, the existence of
an elementary embedding \(j : V_{\lambda+3}\to V_{\lambda+3}\) implies the
consistency of ZFC + \(I_0\). 
embedding \(j : V_{\lambda+3}\to V_{\lambda+3}\) implies that every subset of
\(V_{\lambda+1}\) has a sharp. We show that the existence of an elementary
embedding from \(V_{\lambda+2}\) to \(V_{\lambda+2}\) is equiconsistent with the
existence of an elementary embedding from \(L(V_{\lambda+2})\) to
\(L(V_{\lambda+2})\) with critical point below \(\lambda\). We show that
assuming DC, the existence of an elementary embedding \(j : V_{\lambda+3}\to
V_{\lambda+3}\) implies the consistency of ZFC + \(I_0\). By a recent result of
Schlutzenberg \cite{SchlutzenbergI0}, an elementary embedding from
\(V_{\lambda+2}\) to \(V_{\lambda+2}\) does not suffice.
\end{abstract}

\section{Introduction}
Assuming the Axiom of Choice, the large cardinal hierarchy comes to an abrupt
halt in the vicinity of an \(\omega\)-huge cardinal. This is the content of
Kunen's Inconsistency Theorem. The anonymous referee of Kunen's 1968 paper
\cite{Kunen} raised the question of whether this theorem can be proved without
appealing to the Axiom of Choice. This question remains unanswered. If the
answer is no, then dropping the Axiom of Choice, a choiceless large cardinal
hierarchy extends unimpeded beyond the Kunen barrier. The consistency of these
large cardinals beyond choice would raise profound philosophical problems,
arguably undermining the status of ZFC as a foundation for all of mathematics.
(These problems will not be discussed further here.)

Of course, G\"odel's Incompleteness Theorem precludes a definitive positive
answer to the question of the consistency of any  large cardinal axiom,
choiceless or not. Instead, one can only hope to develop a large cardinal's
theory to the point that it would be unreasonable to doubt its consistency. This
has been achieved for some initial segment of the large cardinal hierarchy,
although which axioms are secured in this weak sense is not a matter of general
agreement. We can all agree, however, that there is scant evidence to date of
the consistency of any of the axioms beyond the Kunen inconsistency. 

In fact, a number of researchers have tried to {\it refute} the choiceless large
cardinals in ZF. Many partial results towards this appear in Woodin's {\it
Suitable Extender Models II}; for example, \cite[Section 7]{SEM} and
\cite[Section 5]{SEM2}. In the other direction, the theory of large cardinals
just below the Kunen inconsistency has been developed quite extensively: for
example, in \cite{SEM2} and \cite{Cramer}. The theory of choiceless large
cardinals far beyond the Kunen inconsistency, especially Berkeley cardinals, is
developed in \cite{KWB} and \cite{Cutolo}. Following Schlutzenberg
\cite{SchlutzenbergReinhardt}, we take up the theory of choiceless large
cardinals right at the level of the principle that Kunen refuted in ZFC. In
particular, we will be concerned with the structure of nontrivial elementary
embeddings from \(V_{\lambda+n}\) to \(V_{\lambda+n}\) where \(\lambda\) is a
limit ordinal and \(n\) is a natural number. One of the general themes of this
work is that, assuming choiceless large cardinal axioms, the structure of
\(V_{\lambda+2n}\) is very different from that of \(V_{\lambda+2n+1}\). 

The underlying phenomenon here involves the definability properties of
rank-into-rank embeddings, which is the subject of \cref{DefinabilitySection}.
An ordinal \(\alpha\) is said to be {\it even} if \(\alpha = \lambda+2n\) for
some limit ordinal \(\lambda\) and some natural number \(n\); otherwise,
\(\alpha\) is {\it odd.}
\begin{repthm}{FirstPeriodicity}
Suppose \(\epsilon\) is an even ordinal.
\begin{enumerate}[(1)]
\item No nontrivial elementary embedding from \(V_{\epsilon}\) to \(V_\epsilon\)
is definable over \(V_\epsilon\).
\item Every elementary embedding from \(V_{\epsilon+1}\) to \(V_{\epsilon+1}\)
is definable over \(V_{\epsilon+1}\).
\end{enumerate}
\end{repthm}
This theorem was the catalyst for most of this research. It was independently
discovered by Schlutzenberg, and is treated in greater detail in the joint paper
\cite{BergBerg}.

We will use the following notation:
\begin{defn}
Suppose \(M\) and \(N\) are transitive classes and \(j : M\to N\) is an
elementary embedding. Then the {\it critical point} of \(j\), denoted
\(\text{crit}(j)\), is the least ordinal moved by \(j\). The {\it critical
supremum of \(j\)}, denoted \(\kappa_\omega(j)\), is the least ordinal above
\(\text{crit}(j)\) that is fixed by \(j\).
\end{defn}
Most of the study of rank-to-rank embeddings has focused on embeddings \(j\)
either from \(V_{\kappa_\omega(j)}\) to \(V_{\kappa_\omega(j)}\) or from
\(V_{\kappa_\omega(j)+1}\) to \(V_{\kappa_\omega(j)+1}\). The reason, of course,
is that assuming the Axiom of Choice, these are the only rank-to-rank embeddings
there are. (This well-known fact follows from the proof of Kunen's theorem.)
Part of the purpose of this paper is to use \cref{FirstPeriodicity} to extend
this theory to embeddings of \(V_\epsilon\) and \(V_{\epsilon+1}\) where
\(\epsilon\) is an arbitrary even ordinal.

The mysterious analogy between the structure of the inner model \(L(\mathbb R)\)
assuming \(\text{AD}^{L(\mathbb R)}\) and that of \(L(V_{\lambda+1})\) under the
axiom \(I_0\) motivates much of the theory of \(L(V_{\lambda+1})\) developed in
\cite{SEM2}.\footnote{The axiom \(I_0\) states that there is an elementary
embedding from \(L(V_{\lambda+1})\) to \(L(V_{\lambda+1})\) with critical point
less than \(\lambda\).} In \cref{ThetaSection}, we attempt to develop a similar
analogy between the structure of arbitrary subsets of \(V_{\epsilon+1}\)
assuming that there is an elementary embedding from \(V_{\epsilon+2}\) to
\(V_{\epsilon+2}\), and the structure of subsets of \(\mathbb R\) assuming full
AD. 

Our main focus in \cref{ThetaSection} is the following sequence of cardinals:
\begin{defn}We denote by \(\theta_\alpha\) the supremum of all ordinals that are
the surjective image of \(V_\beta\) for some \(\beta < \alpha\).
\end{defn}
The problem of determining the structure of the cardinals \(\theta_\alpha\) is a
choiceless analog of the (generalized) Continuum Problem. Note that for any
limit ordinal \(\lambda\), \(\theta_{\lambda}\) is a strong limit
cardinal\footnote{In the context of ZF, a cardinal \(\theta\) is a {\it strong
limit cardinal} if \(\theta\) is not the surjective image of \(P(\beta)\) for
any ordinal \(\beta < \theta\).} and \(\theta_{\lambda+1} =
(\theta_{\lambda})^+\). We conjecture that this phenomenon generalizes
periodically:
\begin{repconj}{ThetaConj}
Suppose \(\epsilon\) is an even ordinal and there is an elementary embedding
from \(V_{\epsilon+1}\) to \(V_{\epsilon+1}\). Then \(\theta_{\epsilon}\) is a
strong limit cardinal and \(\theta_{\epsilon+1} = (\theta_{\epsilon})^+\).
\end{repconj}
Under the Axiom of Determinacy, \(\theta_\omega = \omega\) is a strong limit
cardinal, \(\theta_{\omega+1} = \omega_1\), \(\theta_{\omega+2} = \Theta\) is a
strong limit cardinal, and \(\theta_{\omega+3} = \Theta^+\).

In addition to this numerology, various partial results of \cref{ThetaSection}
suggest that \cref{ThetaConj} holds, or at least that \(\theta_{\epsilon}\) is
relatively large and \(\theta_{\epsilon+1}\) is relatively small. For example:
\begin{repthm}{OddBigThm}
Suppose \(\epsilon\) is an even ordinal. Suppose \(j : V_{\epsilon+2}\to
V_{\epsilon+2}\). Then there is no surjection from
\(P(({\theta_{\epsilon+1}})^{+\lambda})\) onto \(\theta_{\epsilon+2}\) where
\(\lambda = \kappa_\omega(j)\).
\end{repthm}
\begin{repthm}{EvenSmallThm}
Suppose \(\epsilon\) is an even ordinal. Suppose \(j : V_{\epsilon+3}\to
V_{\epsilon+3}\) is an elementary embedding with critical point \(\kappa\). Then
the interval \((\theta_{\epsilon+2},\theta_{\epsilon+3})\) contains fewer than
\(\kappa\) regular cardinals.
\end{repthm}

The attempt to prove \cref{ThetaConj} leads to the following principle: 
\begin{defn}
We say {\it \(V_{\alpha+1}\) satisfies the Collection Principle} if every
binary relation \(R\subseteq V_{\alpha}\times V_{\alpha+1}\) has a
subrelation \(S\) such that \(\text{dom}(S) = \text{dom}(R)\) and
\(\text{ran}(S)\) is the surjective image of \(V_\alpha\).
\end{defn}
From one perspective, the Collection Principle is a weak choice principle. It
follows from the Axiom of Choice, because one can take the subrelation \(S\) to
be a uniformization of \(R\). Another perspective is that the Collection
Principle states that \(\theta_{\alpha+1}\) is regular in a strong sense. In
particular, if \(V_{\alpha+1}\) satisfies the Collection Principle,
then \(\theta_{\alpha+1}\) is a regular cardinal. 
Under \(\text{AD}\), the converse holds at \(\omega+2\): if
\(\theta_{\omega+2}\) is regular, then \(V_{\omega+2}\) satisfies the Collection
Principle.
\begin{repthm}{StrongLimitThm}
Suppose \(\epsilon\) is an even ordinal. Suppose \(j : V_{\epsilon+2}\to
V_{\epsilon+2}\) is a nontrivial elementary embedding. Assume
\(\kappa_\omega(j)\)-\textnormal{DC} and that \(V_{\epsilon+1}\) satisfies the
Collection Principle.\footnote{The choice principle \(\lambda\)-DC is defined in
\cref{DCSection}} Then \(\theta_{\epsilon+2}\) is a strong limit cardinal.
Moreover, for all \(\beta < \theta_{\epsilon+2}\), \(P(\beta)\) is the
surjective image of \(V_{\epsilon+1}\).
\end{repthm}
The proof of this theorem involves generalizing Woodin's Coding Lemma. The
theorem yields a new proof of the Kunen inconsistency theorem: assuming the
Axiom of Choice, the hypotheses of \cref{StrongLimitThm} hold, yet
\(\theta_{\epsilon+2} = |V_{\epsilon+1}|^+\) is not a strong limit cardinal, and
it follows that there is no elementary embedding from \(V_{\epsilon+2}\) to
\(V_{\epsilon+2}\). (A slightly more detailed proof appears in \cref{KIT}.)

To drive home the contrast between the even and odd levels, we show that the
final conclusion of \cref{StrongLimitThm} fails at the even levels:
\begin{repthm}{NoCodingThm}
Suppose \(\epsilon\) is an even ordinal and there is an elementary embedding
from \(V_{\epsilon+2}\) to \(V_{\epsilon+2}\). Then for any ordinal \(\gamma\),
there is no surjection from \(V_{\epsilon}\times \gamma\) onto
\(P(\theta_{\epsilon})\).
\end{repthm}

\cref{KetonenSection} concerns the theory of ultrafilters assuming choiceless
large cardinals. Woodin proved that choiceless large cardinal axioms (combined
with ``\(\kappa_\omega(j)\)-DC'') imply the existence of measurable successor
cardinals. The ultrafilters he produced bear a strong resemblance to the
ultrafilters arising in the context of AD. Here we expand upon that theme. 

First, we study the ordinal definability of ultrafilters over ordinals:
\begin{repthm}{ODThm}
Suppose \(j : V_{\epsilon+3}\to V_{\epsilon+3}\) is an elementary embedding. Let
\(\lambda = \kappa_\omega(j)\). Assume \(\lambda\)-\textnormal{DC}. Suppose
\(U\) is a \(\lambda^+\)-complete ultrafilter over an ordinal less than
\(\theta_{\epsilon+2}\). Then the following hold:
\begin{enumerate}[(1)]
\item \(U\cap \textnormal{HOD}\) belongs to \(\textnormal{HOD}\).
\item \(U\) belongs to an ordinal definable set of cardinality less than
\(\lambda\).
\item For an \(\textnormal{OD}\)-cone of \(x\in V_\lambda\), the ultrapower
embedding \(j_U\) is amenable to \(\textnormal{HOD}_x\).
\end{enumerate}
\end{repthm}
This result uses an analog of the Ultrapower Axiom of \cite{UA} that is provable
from choiceless large cardinals (\cref{AlmostUA}).

Finally, we prove a form of strong compactness for \(\kappa_\omega(j)\) where
\(j :V\to V\) is an elementary embedding:
\begin{repthm}{FilterExtensionThm} 
Suppose \(j : V\to V\) is a nontrivial elementary embedding. Let \(\lambda =
\kappa_\omega(j)\). Assume \(\lambda\)-\textnormal{DC} holds. Then every
\(\lambda^+\)-complete filter over an ordinal extends to a \(\lambda^+\)-complete
ultrafilter.
\end{repthm}
This result is an application of the Ketonen order on filters, a wellfounded
partial order on countably complete filters over ordinals that simultaneously
generalizes the Jech order on stationary sets and the Mitchell order on normal
ultrafilters.

Like many of the arguments of this paper (e.g., \cref{ODThm}), the proof of
\cref{FilterExtensionThm} is general enough that it yields a new consequence of
\(I_0\):
\begin{repthm}{I0FilterExtensionThm}[ZFC] Suppose \(\lambda\) is a cardinal and
there is an elementary embedding from \(L(V_{\lambda+1})\) to
\(L(V_{\lambda+1})\) with critical point less than \(\lambda\). Then in
\(L(V_{\lambda+1})\), every \(\lambda^+\)-complete filter over an ordinal less
than \(\theta_{\lambda+2}\) extends to a \(\lambda^+\)-complete ultrafilter.
\end{repthm}

In the last section of this paper, \cref{ConsistencySection}, we turn to
consistency results. Most of these results predate the groundbreaking theorem of
Schlutzenberg \cite{SchlutzenbergI0} that the existence of an elementary
embedding \(j : L(V_{\lambda+1})\to L(V_{\lambda+1})\) with critical point below
\(\lambda\) is equiconsistent with the existence of an elementary embedding from
\(V_{\lambda+2}\) to \(V_{\lambda+2}\), but it is useful to keep this theorem in
mind to appreciate the statements of our theorems.

We prove the equiconsistency of various choiceless large cardinals associated
with the Kunen inconsistency:
\begin{repthm}{Equicon}
	The following statements are equiconsistent over \textnormal{ZF}:
	\begin{enumerate}[(1)]
	\item For some \(\lambda\), there is a nontrivial elementary embedding from
	\(V_{\lambda+2}\) to \(V_{\lambda+2}\).
	\item For some \(\lambda\), there is an elementary embedding from
	\(L(V_{\lambda+2})\) to \(L(V_{\lambda+2})\) with critical point below
	\(\lambda\).
	\item There is an elementary embedding \(j\) from \(V\) to an inner model
	\(M\) that is closed under \(V_{\kappa_\omega(j)+1}\)-sequences.
	\end{enumerate}
\end{repthm}
Combined with Schlutzenberg's Theorem, this shows that all of these principles
are equiconsistent with the the existence of an elementary embedding from
\(L(V_{\lambda+1})\) to \(L(V_{\lambda+1})\) with critical point below
\(\lambda\).

Our next theorem shows that choiceless large cardinal axioms beyond an
elementary embedding from \(V_{\lambda+2}\) to \(V_{\lambda+2}\) are stronger
than \(I_0\):
\begin{repthm}{I0Con}
	Suppose \(\lambda\) is an ordinal and there is a \(\Sigma_1\)-elementary
	embedding \(j :V_{\lambda+3}\to V_{\lambda+3}\) with \(\lambda =
	\kappa_\omega(j)\). Assume \(\textnormal{DC}_{V_{\lambda+1}}\). Then there
	is a set generic extension \(N\) of \(V\) such that \((V_\delta)^N\)
	satisfies \(\textnormal{ZFC}+ I_0\) for some \(\delta < \lambda\).
\end{repthm}
The following result is an immediate corollary:
\begin{cor*}
Over \textnormal{ZF + DC}, the existence of an elementary embedding from
\(V_{\lambda+3}\) to \(V_{\lambda+3}\) implies the consistency of
\(\textnormal{ZFC} + I_0\).\qed
\end{cor*}
By Schlutzenberg's Theorem, the hypothesis of \cref{I0Con} cannot be reduced to
the existence of an elementary embedding from \(V_{\lambda+2}\) to
\(V_{\lambda+2}\), or even the existence of a \(\Sigma_0\)-elementary embedding
\(j :V_{\lambda+3}\to V_{\lambda+3}\) with \(j(V_{\lambda+2})=V_{\lambda+2}\).

Schlutzenberg \cite{SchlutzenbergI0} poses the problem of calculating the exact
consistency strength over ZF of the existence of an elementary embedding from
\(V_{\lambda+2}\) to \(V_{\lambda+2}\) in terms of large cardinal axioms
compatible with the Axiom of Choice. We sketch how to calculate the consistency
strength of this assertion over ZF + DC:
\begin{repthm}{I0EquiCon}
	The following statements are equiconsistent over \textnormal{ZF + DC}:
	\begin{enumerate}[(1)]
	\item For some ordinal \(\lambda\), there is an elementary embedding from
	\(V_{\lambda+2}\) to \(V_{\lambda+2}\).
	\item The Axiom of Choice\textnormal{ + \(I_0\)}.
	\end{enumerate}
\end{repthm}
We defer to the appendix some facts about countably complete filters and
ultrafilters that are used in \cref{UndefinabilitySection} and
\cref{KetonenSection}. This is accomplished by considering a version of the
Ketonen order studied in \cite{UA} that is applicable to countably complete
filters on complete Boolean algebras in the context of ZF + DC. This level of
generality is overkill, but it makes the proofs slicker.
\section{Notation and preliminaries}\label{Preliminaries}
In this section, we lay out some of the notational conventions we will use in
this paper. {\it Most importantly, we work throughout this paper in
\textnormal{ZF} alone, without assuming the Axiom of Choice, explicitly making
note of any other choice principles we use.} Most of the notation discussed here
is standard, with the notable exception of \cref{HSection}, which introduces a
class of structures \(\langle \mathcal H_\alpha\rangle_{\alpha \in \text{Ord}}\)
which will be very useful throughout the paper.
\subsection{Elementary embeddings}
We use the following notation for elementary embeddings:
\begin{defn}
Suppose \( M\) and \( N\) are structures in the same signature. Then \(\mathscr
E(M,N)\) denotes the set of elementary embeddings from \(M\) to \(N\), and
\(\mathscr E(M)\) denotes the set of elementary embeddings from \(M\) to itself.
\end{defn}
Typically the structures we consider are of the form \((M,\in)\) where \(M\) is
a transitive set. We will always suppress the membership relation, writing
\(\mathscr E(M)\) when we mean \(\mathscr E(M,{\in})\).

Our notation for the critical sequence of an embedding is pulled from
\cite{SEM2}:
\begin{defn}
Suppose \(M\) and \(N\) are transitive structures and \(j\in \mathscr E(M,N)\).
The {\it critical point of \(j\)}, denoted \(\text{crit}(j)\), is the least
ordinal moved by \(j\). The {\it critical sequence of \(j\)} is the sequence
\(\langle \kappa_n(j) \mid {n < \omega}\rangle\) defined by \(\kappa_0(j) =
\text{crit}(j)\) and \(\kappa_{n+1}(j) = j(\kappa_n(j))\). Finally, the {\it
critical supremum of \(j\)} is the ordinal \(\kappa_\omega(j) = \sup_{n <
\omega} \kappa_n(j)\).
\end{defn}
Of course, \(\text{crit}(j)\) may not be defined since \(j\) may have no
critical point. Even if \(\text{crit}(j)\) is defined, \(\kappa_{n+1}(j)\) may
not be for some \(n < \omega\), since it is possible that \(\kappa_n(j)\notin
M\). 
\subsection{Weak choice principles}\label{DCSection}
We say that \(T\subseteq X^{<\lambda}\) is a tree if for all \(s\in T\), for all
\(\alpha < \text{dom}(s)\), \(s\restriction \alpha \in T\). A tree \(T\subseteq
X^{<\lambda}\) is {\it \(\lambda\)-closed} if for any \(s\in X^{<\lambda}\) with
\(s\restriction \alpha\in T\) for all \(\alpha < \text{dom}(s)\), \(s\in T\). A
{\it cofinal branch} of a tree \(T\subseteq X^{<\lambda}\) is a sequence \(s\in
X^\lambda\) such that \(x\restriction \alpha\in T\) for all \(\alpha <
\lambda\).

Various weak choice principles will be used throughout the paper. The most
important are the following:
\begin{defn}
Suppose \(\lambda\) is a cardinal and \(X\) is a set.
\begin{itemize}
\item \(\lambda\textnormal{-DC}_X\) denotes the principle asserting that every
\(\lambda\)-closed tree of sequences \(T\subseteq X^{<\lambda}\) with no maximal
branches has a cofinal branch.
\item \(\lambda\)\textnormal{-DC} denotes the principle asserting that
\(\lambda\textnormal{-DC}_Y\) holds for all sets \(Y\).
\item The Axiom of Dependent Choice, or DC, is the principle \(\omega\)-DC.
\end{itemize}
\end{defn}

In the context of ZF, it may be that there is a surjection from \(X\) to \(Y\)
but no injection from \(Y\) to \(X\). We therefore use the following notation:
\begin{defn}\label{CardinalityEqn}
If \(X\) and \(Y\) are sets, then \(X\preceq^{*} Y\) if there is a partial
surjection from \(Y\) to \(X\). We let \([X]^{Y} = \{S\subseteq X \mid
S\preceq^{*} Y\}\).
\end{defn}
We use partial surjections because these are what arise naturally in practice,
but of course \(X\preceq^{*} Y\) if and only if either \(X = \emptyset\) or
there is a total surjection from \(Y\) to \(X\).

\subsection{Filters and ultrafilters}
We use the following convention: a filter {\it over} a set \(X\) is a filter
{\it on} the Boolean algebra \(P(X)\). Filters on Boolean algebras do not come
up until the appendix, so until then, we use the word filter to refer to a
filter over some set. 

\begin{defn}\label{CompletenessDef}
Suppose \(\gamma\) is an ordinal. A filter \(F\) is {\it \(\gamma\)-saturated}
if there is no sequence \(\langle S_\alpha \mid \alpha < \gamma\rangle\) of
\(F\)-positive sets such that \(S_\alpha\cap S_\beta\) is \(F\)-null for all
\(\alpha < \beta < \gamma\); \(F\) is {\it weakly \(\gamma\)-saturated} if there
is no sequence \(\langle S_\alpha \mid \alpha < \gamma\rangle\) of pairwise
disjoint \(F\)-positive sets.
\end{defn}
If \(F\) is \(\gamma\)-complete, then \(F\) is \(\gamma\)-saturated if and only
if \(F\) is weakly \(\gamma\)-saturated. 

\begin{defn}
If \(B\) is a set, a filter \(F\) is {\it \(B\)-complete} if for any \(b\in B\)
and any \(D\subseteq F\) such that \(D\preceq^{*} b\), \(\bigcap D\in F\). A
filter \(F\) is {\it \(X\)-closed} if \(F\) is \(\{X\}\)-complete.
\end{defn}
A filter is said to be {\it countably complete} if it is \(\omega_1\)-complete.

We will need the standard derived ultrafilter construction:
\begin{defn}
Suppose \(h : P(X)\to P(Y)\) is a homomorphism of Boolean algebras and \(a\in
Y\). The {\it ultrafilter over \(X\) derived from \(h\) using \(a\)} is the
ultrafilter over \(X\) defined by the formula \(\{A\subseteq X: a\in h(A)\}\).
\end{defn}

Our notation for ultrapowers is standard in set theory. If \(U\) is an
ultrafilter over a set \(X\), then \(j_U : V\to M_U\) denotes the associated
ultrapower. If \(\langle M_x\rangle_{x\in X}\) is a sequence of structures in
the same signature, then \(\prod_{x\in X}M_x/U\) denotes their ultraproduct. 
\subsection{The structures \(\mathcal H_\alpha\)}\label{HSection}
Although the subject of this paper is rank-to-rank embeddings (i.e., elements of
\(\mathscr E(V_\alpha)\) for some ordinal \(\alpha\)), it is often convenient to
lift these embeddings to act on larger structures. The issue is that many sets
are coded in \(V_\alpha\) but do not belong to \(V_\alpha\). This is especially
annoying when \(\alpha\) is a successor ordinal, in which case \(V_\alpha\)
fails to be closed under Kuratowski pairs. This motivates introducing the
following structures:
\begin{defn}\label{HAlphaLma}
For any set \(X\), let \(\mathcal H(X)\) denote the union of all transitive sets
\(M\) such that \(M\preceq^{*} S\) for some \(S\in X\). For any ordinal
\(\alpha\), let \(\mathcal H_\alpha = \mathcal H(V_\alpha)\).
\end{defn}
If \(\kappa\) is a (wellordered) cardinal, then \(\mathcal H(\kappa)\) is the
usual structure \(H(\kappa)\). In ZF, however, there may be other structures of
the form \(\mathcal H(X)\). Notice that \(\mathcal H_{\alpha+1}\) is the
collection of sets that are the surjective image of \(V_\alpha\). 
\begin{defn}\label{ThetaDef}
For any set \(X\), let \(\theta(X)\) denote the least ordinal that is not the
surjective image of some set \(S\in X\). Let \(\theta_\alpha =
\theta(V_\alpha)\). 
\end{defn}
Then \(\theta(X) = \mathcal H(X)\cap \text{Ord}\). The cardinals
\(\theta_\alpha\) are studied in \cref{ThetaSection}.

Note that for all \(\alpha\), \(V_\alpha\subseteq \mathcal H_\alpha\). We claim
that every embedding in \(\mathscr E(V_{\alpha})\) extends uniquely to an
embedding in \(\mathscr E(\mathcal H_\alpha)\). This is a consequence of a
coding of \(\mathcal H_\alpha\) inside the structure \(V_\alpha\). We proceed to
describe one such coding, and then we sketch how this yields the unique
extension of embeddings from \(\mathscr E(V_\alpha)\) to \(\mathscr E(\mathcal
H_\alpha)\).

Let \(p : V\to V\times V\) denote some {\it Quine-Rosser pairing function},
which is a bijection that is \(\Sigma_0\)-definable without parameters such that
for all infinite ordinals \(\alpha\), \(p[V_\alpha] = V_\alpha\times V_\alpha\).

Fix an infinite ordinal \(\alpha\). We will define a partial surjection
\[\Phi_\alpha : V_\alpha\to \mathcal H_\alpha\] For each \(x\in V_\alpha\), let
\(R_x = p[x]\) be the binary relation coded by \(x\), and let \(D_x\) denote the
field of \(R_x\). Note that \(D_x\in V_\alpha\). Let \(\textnormal{E}_\alpha\)
be the set of \(x\in V_\alpha\) such that \(R_x\) is a wellfounded
relation on \(D_x\). For \(x\in \text{E}_\alpha\), let \(\pi_x : D_x\to M_x\) be
the Mostowski collapse. Let \[\textnormal{C}_\alpha = \{(x,y) \mid x\in
\textnormal{E}_\alpha\text{ and }y\in D_x\}\] For \((x,y)\in
\textnormal{C}_\alpha\), let \(\Phi_\alpha(x,y) = \pi_x(y)\). Then
\(\text{ran}(\Phi_\alpha) = \bigcup_{x\in V_\alpha} M_x = \mathcal H_\alpha\).

Since \(\Phi_\alpha\) is definable over \(\mathcal H_\alpha\), one has that for
any \(i : \mathcal H_\alpha\to \mathcal H_\alpha\), \(i(\Phi_\alpha(x,y)) =
\Phi_\alpha(i(x),i(y))\). Conversely, since the sets \(\text{C}_\alpha\) and
\(\Phi_\alpha^{-1}[{\in}] = \{(u,w)\in \text{C}_\alpha\times \text{C}_\alpha :
\Phi_\alpha(u) \in \Phi_\alpha(w)\}\), and \(\Phi_\alpha^{-1}[{=}]\) are
definable over \(V_\alpha\), one has that if \(j : V_\alpha\to V_\alpha\) is
elementary, then setting \[j^\star(\Phi_\alpha(x,y)) = \Phi_\alpha(j(x),j(y)),\]
the embedding \(j^\star : \mathcal H_\alpha\to \mathcal H_\alpha\) is
well-defined and elementary.
\begin{defn}
For any \(j\in \mathscr E(V_\alpha)\), \(j^\star\) denotes the unique elementary
embedding \(k : \mathcal H_\alpha\to \mathcal H_\alpha\) such that
\(k\restriction V_\alpha = j\).
\end{defn}

The rest of the section contains an analysis of \(j^\star\) when \(j :
V_\alpha\to V_\alpha\) is only assumed to be \(\Sigma_n\)-elementary for some
\(n < \omega\). This is rarely relevant, and the reader can skip it for now.

We claim that if \(i : \mathcal H_\alpha\to \mathcal H_\alpha\) is a
\(\Sigma_0\)-elementary embedding, then \(i(\Phi_\alpha(x,y)) =
\Phi_\alpha(i(x),i(y))\) for all \((x,y)\in \text{C}_\alpha\). To see this, fix
\((x,y)\in \text{C}_\alpha\). Let \(a = \Phi_\alpha(x,y)\). Let \(M\) be an
admissible set in \(\mathcal H_\alpha\) such that \(x,y\in M\). Then \(M\)
satisfies the statement ``\(a = \pi_x(y)\) where \(\pi_x\) is the Mostowski
collapse of \((D_x,R_x)\)." Since \(i\) is \(\Sigma_0\)-elementary, it follows
that \(i(M)\) is an admissible set that satisfies ``\(i(a) =
\pi_{i(x)}(i(y))\).'' This is expressible as a \(\Sigma_1\) statement, so is
upwards absolute to \(V\). Therefore \(i(a) = \pi_{i(x)}(i(y))\), or in other
words, \(i(\Phi_\alpha(x,y)) = \Phi_\alpha(i(x),i(y))\), as desired.

Conversely, suppose \(j : V_\alpha\to V_\alpha\) is \(\Sigma_1\)-elementary, and
let \(j^\star :\mathcal H_\alpha\to \mathcal H_\alpha\) be defined by \(j^\star
(\Phi_\alpha(x,y)) = \Phi_\alpha(j(x),j(y))\). We claim \(j^\star : \mathcal
H_\alpha\to \mathcal H_\alpha\) is well-defined and \(\Sigma_0\)-elementary.
Note that \(\textnormal{C}_\alpha\), \(\Phi_\alpha^{-1}[{=}]\), and
\(\Phi_\alpha^{-1}[{\in}]\) are \(\Pi_1\)-definable over \(V_\alpha\). This
implies that \(j^\star\) is a well-defined \(\in\)-homomorphism. If \(M\in
\mathcal H_\alpha\) is a transitive set, then taking \(x\in \text{E}_\alpha\)
such that \(M = M_x\), the satisfaction predicate for \((D_x,R_x) \cong M\) is
\(\Delta_1\)-definable over \(V_\alpha\) from \(x\), and hence \(j\restriction M
:  M \to j(M)\) is fully elementary. Since every \(x\in \mathcal H_\alpha\)
belongs to some transitive \(M\in \mathcal H_\alpha\), and such a set \(M\) is a
\(\Sigma_0\)-elementary substructure of \(\mathcal H_\alpha\), it follows that
\(j\) is \(\Sigma_0\)-elementary.

\begin{defn}
Suppose \(i:V_\alpha\to V_\alpha\) is a \(\Sigma_1\)-elementary embedding. Then
\(i^\star\) denotes the unique \(\Sigma_0\)-elementary embedding from \(\mathcal
H_\alpha\) to \(\mathcal H_\alpha\) extending \(i\).
\end{defn}
By a localization of the arguments above, one obtains the following fact about
partially elementary \(\star\)-extensions:
\begin{lma}\label{LocalElementarity}
Suppose \(n < \omega\) and \(i : V_\alpha\to V_\alpha\) is a
\(\Sigma_{n+1}\)-elementary embedding. Then \(i^\star :\mathcal H_\alpha\to
\mathcal H_\alpha\) is \(\Sigma_n\)-elementary.
\begin{proof}[Sketch] For example, take the case \(n = 1\). Suppose \(\psi\) is
a \(\Sigma_0\)-formula, \(a\in \mathcal H_\alpha\) and \(\mathcal H_\alpha\)
satisfies \(\exists v\ \psi(j(a),v)\). Fix \((x_0,y_0)\in \text{C}_\alpha\) with
\(\Phi_\alpha(x_0,y_0) = a\). Then \(V_\alpha\) satisfies that there is some
\((x_1,y_1)\in \text{C}_\alpha\) such that \((D_{x_1},R_{x_1})\) is an
end-extension of \((D_{j(x_0)},R_{j(x_0)})\) and \((D_{x_1},R_{x_1})\) satisfies
\(\psi(j(y_0),y_1)\). This is \(\Sigma_2\)-expressible in \(V_\alpha\), so by
elementarity, there is some \((x_1,y_1)\in \text{C}_\alpha\) such that
\((D_{x_1},R_{x_1})\) is an end-extension of \((D_{x_0},R_{x_0})\) and
\((D_{x_1},R_{x_1})\) satisfies \(\psi(y_0,y_1)\). It follows that \(M_{x_1}\)
satisfies \(\psi(a,\Phi_\alpha(x_1,y_1))\), and hence \(\mathcal H_\alpha\)
satisfies \(\exists v\ \psi(a,v)\), as desired.
\end{proof}
\end{lma}

\section{The definability of rank-to-rank embeddings}
\subsection{Prior work}
The results of this section are inspired by the work of Schlutzenberg
\cite{SchlutzenbergReinhardt}, which greatly expands upon the following theorem
of Suzuki:\footnote{Suzuki actually proved the slightly stronger schema that no
elementary embedding from \(V\) to \(V\) is definable from parameters over
\(V\).}
\begin{thm}[Suzuki]\label{SuzukiThm}
If \(\kappa\) is an inaccessible cardinal, no nontrivial elementary embedding
from \(V_\kappa\) to \(V_\kappa\) is definable over \(V_\kappa\) from
parameters.
\end{thm}
Schlutzenberg \cite{SchlutzenbergReinhardt} extended this to limit
ranks:\footnote{Schlutzenberg also proved many other definability results for
rank-to-rank embeddings, incorporating, for example, constructibility and
ordinal definability.}
\begin{thm}[Schlutzenberg]\label{SchlutzenbergUndefinableThm}
Suppose \(\lambda\) is a limit ordinal. Then no nontrivial elementary embedding
from \(V_\lambda\) to \(V_{\lambda}\) is definable over \(V_\lambda\) from
parameters.
\end{thm}
Schlutzenberg noted that the situation for elementary embeddings from
\(V_{\lambda+1}\) to \(V_{\lambda+1}\), where \(\lambda\) is a limit ordinal, is
completely different: {\it every} elementary embedding from \(V_{\lambda+1}\) to
\(V_{\lambda+1}\) is definable from parameters over \(V_{\lambda+1}\). He then
asked the corresponding question for elementary embeddings of \(V_{\lambda+n}\)
for \(n > 1\). The answer is given by the following theorem, which is
established by the main results of this section:
\begin{thm}\label{FirstPeriodicity}
Suppose \(\epsilon\) is an even ordinal.\footnote{An ordinal \(\alpha\) is said
to be {\it even} if for some limit ordinal \(\lambda\) and some natural number
\(n\), \(\alpha = \lambda+2n\); otherwise, \(\alpha\) is {\it odd}.}
\begin{enumerate}[(1)]
\item No nontrivial elementary embedding from \(V_{\epsilon}\) to \(V_\epsilon\)
is definable over \(V_\epsilon\) from parameters.
\item Every elementary embedding from \(V_{\epsilon+1}\) to \(V_{\epsilon+1}\)
is definable over \(V_{\epsilon+1}\) from parameters.
\end{enumerate}
\end{thm}
This theorem is the first instance of a periodicity phenomenon in the hierarchy
of choiceless large cardinal axioms, leading to a generalization to arbitrary
ranks of the basic theory of rank-to-rank embeddings familiar from the ZFC
context. (1) is proved as \cref{UndefinabilityThm} of
\cref{UndefinabilitySection} and (2) as \cref{DefinabilityThm} of
\cref{DefinabilitySection}. 

We note that Schlutzenberg rediscovered \cref{FirstPeriodicity}, and this
theorem is the main subject of the joint paper \cite{BergBerg}.
\subsection{Extending embeddings to \(V_{\epsilon+1}\)}\label{DefinabilitySection}
That elementary embeddings of odd ranks are definable (\cref{FirstPeriodicity}
(2)) came as quite a surprise to the author, but with hindsight emerges as a
natural generalization a well-known phenomenon from the standard theory of
rank-to-rank embeddings.
\begin{defn}\label{LimitExtensionDef}
Suppose \(\lambda\) is a limit ordinal and \(j : V_{\lambda}\to V_\lambda\) is
an elementary embedding. Then the {\it canonical extension of \(j\)} is the
embedding \(j^{+} : V_{\lambda+1}\to V_{\lambda+1}\) defined by \(j^{+}(X) =
\bigcup_{\Gamma\in V_\lambda} j(X\cap \Gamma)\).
\end{defn}
While the canonical extension of an embedding in \(\mathscr E(V_\lambda)\) is
not necessarily an elementary embedding, it is true that {\it if} \(i\in
\mathscr E(V_{\lambda+1})\), then necessarily \(i = (i\restriction
V_\lambda)^{+}\). The proof is easy, but since it is relevant below, we give a
detailed sketch. Fix \(X\in V_{\lambda+1}\). Clearly \(i(X\cap \Gamma)\subseteq
i(X)\) for all \(\Gamma\in V_\lambda\), and this easily implies the inclusion
\((i\restriction V_\lambda)^{+}(X)\subseteq i(X)\). For the reverse inclusion,
suppose \(a\in i(X)\). Since \(\lambda\) is a limit ordinal, there is some
\(\Gamma\in V_\lambda\) such that \(a\in i(\Gamma)\); for example, one can take
\(\Gamma = V_{\xi+1}\) where \(\xi = \text{rank}(a)\). Now \(a\in i(X) \cap
i(\Gamma)= i(X\cap \Gamma)\), so \(a\in (i\restriction V_\lambda)^{+}(X)\). The
key property of \(\lambda\) that was used in this proof is that any \(i :
V_\lambda\to V_\lambda\) is a {\it cofinal embedding} in the sense that for all
\(a\in V_\lambda\), there is some \(\Gamma\in V_\lambda\) with \(a\in
i(\Gamma)\).

Suppose now that \(\alpha\) is an arbitrary infinite ordinal. We want to
generalize the canonical extension operation to act on  embeddings \(j\in
\mathscr E(V_{\alpha})\). It is easy to see that if \(\alpha\) is a successor
ordinal, the naive generalization (i.e., \(j^{+}(X) = \bigcup_{\Gamma\in
V_\alpha} j(X\cap \Gamma)\) for \(X\in V_{\alpha+1}\)) does not have the desired
effect. (For example, adopting this definition, one would have
\(j^{+}(\{V_{\alpha-1}\}) = \emptyset\).) Instead, one must make the following
tweak:
\begin{defn}\label{SuccessorExtensionDef}
Suppose \(\alpha\) is an infinite ordinal and \(j: V_\alpha\to V_\alpha\) is an
elementary embedding. Then the {\it canonical extension of \(j\)} is the
embedding \(j^{+} : V_{\alpha+1}\to V_{\alpha+1}\) defined by \(j^{+}(X) =
\bigcup_{\Gamma\in \mathcal H_\alpha} j^\star(X\cap \Gamma)\).
\end{defn}
See \cref{HSection} for the definition of the structure \(\mathcal H_\alpha\),
and the basic facts about lifting elementary embeddings from \(V_\alpha\) to
\(\mathcal H_\alpha\). Thus \(j^{+}(X)\) is the union of all sets of the form
\(j^\star(X\cap \Gamma )\) where \(\Gamma\) is {\it coded} in \(V_\alpha\). The
following easily verified lemma clarifies the definition:
\begin{prp}
Suppose \(\alpha\) is an infinite ordinal and \(j\in \mathscr E(V_\alpha)\).
Then 
\[\pushQED{\qed}j^{+}(X) = \begin{cases}
\bigcup_{\Gamma\in V_\alpha} j(X\cap \Gamma)&\text{ if \(\alpha\) is a limit
ordinal}\\
\bigcup_{\Gamma\in [V_\alpha]^{V_{\alpha-1}}} j^\star(X\cap \Gamma)&\text{ if \(\alpha\) is a successor ordinal}
\end{cases}\qedhere\popQED\]
\end{prp}

While the definition of the canonical extension operation directly generalizes
\cref{LimitExtensionDef}, a key new phenomenon arises at successor ranks: it is
no longer clear that every elementary embedding \(i : V_{\alpha+1}\to
V_{\alpha+1}\) satisfies \[i = (i\restriction V_{\alpha})^{+}\] It is easy to
show that for all \(X\in V_{\alpha+1}\), \((i\restriction V_{\alpha})^{+}(X)
\subseteq i(X)\), but the reverse inclusion is no longer clear. 

In fact, the reverse inclusion is only true for even values of \(\alpha\). This
is proved by an induction that simultaneously establishes the {\it canonical
extension property} and the {\it cofinal embedding property}, which we now
define.
\begin{defn}
Suppose \(\epsilon\) is an ordinal. Then \(\epsilon\) has the {\it canonical
extension property} if for any \(i\in \mathscr E(V_{\epsilon+1})\), \(i =
(i\restriction V_{\epsilon})^{+}\).
\end{defn}
The terminology is motivated by equivalence of the canonical extension property
with the statement that an elementary embedding in \(\mathscr E(V_{\epsilon})\)
extends to at most one embedding in \(\mathscr E(V_{\epsilon+1})\).

\begin{defn}
Suppose \(\epsilon\) is an ordinal. Then \(\epsilon\) has the {\it cofinal
embedding property} if for any \(i\in \mathscr E(V_{\epsilon})\), for any \(A\in
\mathcal H_\epsilon\), there is some \(\Gamma\in \mathcal H_\epsilon\) with
\(A\in i^\star(\Gamma)\).
\end{defn}
Thus the cofinal embedding property states that every embedding in \(\mathscr
E(V_\epsilon)\) induces a cofinal embedding in \(\mathscr E(\mathcal
H_\epsilon)\). The proof that limit ordinals have the canonical extension
property generalizes to all ordinals with the cofinal embedding property:
\begin{lma}\label{ExtensionLma}
Suppose \(\epsilon\) is an ordinal. If \(\epsilon\) has the cofinal embedding
property, then \(\epsilon\) has the canonical extension property.
\begin{proof}
Fix an elementary embedding \(i : V_{\epsilon+1}\to V_{\epsilon+1}\). We must
show \(i = (i\restriction V_{\epsilon})^{+}\). Take \(X\in V_{\epsilon+1}\). The
inclusion \((i\restriction V_{\epsilon})^{+}(X)\subseteq i(X)\) is true
regardless of parity: if \(\Gamma\in \mathcal H_\epsilon\), then \(i^\star(X\cap
\Gamma) = i(X\cap \Gamma) \subseteq i(X)\) by the elementarity of \(i\) and the
uniqueness of \(i^\star\), so \(i^{+}(X) = \bigcup_{\Gamma\in \mathcal H_\alpha}
i^\star(X\cap \Gamma)\subseteq i(X)\).

To show \(i(X) \subseteq (i\restriction V_{\epsilon})^{+}(X)\), suppose \(A\in
i(X)\). Since \(\epsilon\) has the cofinal embedding property, there is some
\(\Gamma\in \mathcal H_\epsilon\) such that \(A\in i^\star(\Gamma)\). Now \[A\in
i(X)\cap i^\star(\Gamma) = i(X)\cap i(\Gamma) = i(X\cap \Gamma) = i^\star(X\cap
\Gamma) \subseteq i^{+}(X)\] This completes the proof.
\end{proof}
\end{lma}
The periodicity phenomenon is a result of the following lemma:
\begin{lma}\label{RepresentationLma1}
Suppose \(\epsilon\) is an ordinal. If \(\epsilon\) has the canonical extension
property, then \(\epsilon+2\) has the cofinal embedding property.
\begin{proof}
Fix \(i \in \mathscr E(V_{\epsilon+2})\) and \(A\in V_{\epsilon+2}\). Let 
\[\Gamma = \{(k^{+})^{-1}[A] \mid k\in \mathscr E(V_{\epsilon})\}\] Then
\(\Gamma\preceq^{*} \mathscr E(V_{\epsilon})\preceq^{*} V_{\epsilon+1}\). Since
\(\Gamma\cup V_{\epsilon+1}\) is transitive, it follows that \(\Gamma\in
\mathcal H_{\epsilon+1}\).
\end{proof}
\end{lma}

This allows us to prove the cofinal embedding property and the canonical
extension property for even ordinals by induction:
\begin{cor}\label{ExtensionThm}
Every even ordinal has the cofinal embedding property.
\begin{proof}
Suppose \(\lambda\) is a limit ordinal. We show that \(\lambda+2n\) has the
cofinal embedding property by induction on \(n < \omega\). 

We first prove the base case, when \(n = 0\). Suppose \(A\in \mathcal
H_\lambda\). Fix \(\xi < \lambda\) such that \(A\in \mathcal H_{\xi}\). Then we
have \(\mathcal H_\xi\in \mathcal H_\lambda\) and \(A\in j^\star(\mathcal
H_\xi)\) since \(\mathcal H_\xi \subseteq \mathcal H_{j(\xi)} = j^\star(\mathcal
H_\xi)\). This shows that \(\lambda\) has the cofinal embedding property.

For the induction step, assume that \(\lambda+2n\) has the cofinal embedding
property. Then by \cref{RepresentationLma1}, \(\lambda+2n\) has the canonical
extension property, and so by \cref{ExtensionLma}, \(\lambda+2n+2\) has the
cofinal embedding property.
\end{proof}
\end{cor}

\begin{cor}\label{RepresentationThm}
Every even ordinal has the canonical extension property.\qed
\end{cor}
As an immediate consequence, we have \cref{FirstPeriodicity} (2):
\begin{thm}\label{DefinabilityThm}
Suppose \(\epsilon\) is an infinite even ordinal and \(i : V_{\epsilon+1}\to
V_{\epsilon+1}\) is an elementary embedding. Then \(i\) is definable over
\(V_{\epsilon+1}\) from \(i[V_\epsilon]\).
\begin{proof}
Clearly \(i\) is definable over \(\mathcal H_{\epsilon+1}\) from \(i\restriction
V_{\epsilon}\) since \(i = (i\restriction V_\epsilon)^{+}\) and the canonical
extension operation is explicitly defined over \(\mathcal H_{\epsilon+1}\).
Moreover, \(i\restriction V_\epsilon\) is definable over \(V_{\epsilon+1}\) from
\(i[V_{\epsilon}]\) as the inverse of the Mostowski collapse. It follows that
\(i\) is definable over \(\mathcal H_{\epsilon+1}\) from \(i[V_\epsilon]\). But
by coding elements of \(\mathcal H_{\epsilon+1}\) as elements of
\(V_{\epsilon+1}\) as in \cref{Preliminaries}, one can translate this into a
definition of \(i\) over \(V_{\epsilon+1}\) from \(i[V_\epsilon]\).
\end{proof}
\end{thm}

Let us put down for safe-keeping the following version of the cofinal embedding
property that is often useful:
\begin{defn}\label{RepresentationDef}
Suppose \(\sigma\) is a set such that \((\sigma,\in)\) is wellfounded and
extensional. Then \(j_\sigma: M_\sigma\to \sigma\) denotes the inverse of the
Mostowski collapse of \(\sigma\).

Suppose \(\epsilon\) is an even ordinal. For any \(A\in V_{\epsilon+2}\), let
\(f_A : V_{\epsilon+1}\to V_{\epsilon+2}\) be the partial function defined by
\(f_A(\sigma) = (j_\sigma^{+})^{-1}[A]\).
\end{defn}
We leave \(f_A(\sigma)\) undefined if one of the following holds:
\begin{itemize}
\item \((\sigma,{\in})\) is not wellfounded and extensional.
\item \(j_\sigma\) is not an elementary embedding from \(V_\epsilon\) to
\(V_\epsilon\).
\end{itemize}
\begin{prp}\label{RepresentationLma}
Suppose \(j : V_{\epsilon+2}\to V_{\epsilon+2}\) is an elementary embedding.
Then for any \(A\in V_{\epsilon+2}\), \(A = j^\star(f_A)(j[V_\epsilon])\).
\begin{proof}
Since \(f_A\) is definable from \(A\) over \(\mathcal H_{\epsilon+2}\),
\(j^\star(f_A)(j[V_\epsilon]) = f_{j(A)}(j[V_\epsilon])\). Note that
\(M_{j[V_\epsilon]} = V_\epsilon\) and \(j_{j[V_\epsilon]} = j\restriction
V_\epsilon\). Therefore by the elementarity of \(j^\star\), 
\[f_{j(A)}(j[V_\epsilon]) = ((j\restriction V_\epsilon)^{+})^{-1}[j(A)] =
(j\restriction V_{\epsilon+1})^{-1}[j(A)] = A\qedhere\]
\end{proof}
\end{prp}

Our original approach to \cref{DefinabilityThm} diverged from the one presented
here in that we used a superficially different definition of the canonical
extension operation from \cref{SuccessorExtensionDef}. This approach is not as
clearly motivated by the canonical extension operation for embeddings from
\(V_\lambda\) to \(V_\lambda\) where \(\lambda\) is a limit ordinal
(\cref{LimitExtensionDef}), but it might illuminate the underlying
combinatorics. This construction is described in more detail in \cite{BergBerg}.

Suppose \(j : V_{\epsilon+2}\to V_{\epsilon+2}\) is elementary, and let
\(\mathcal U\) be the ultrafilter derived from \(j\) using \(j[V_{\epsilon}]\).
Let \(j_\mathcal U : V\to M_\mathcal U\) denote the ultrapower associated to
\(\mathcal U\). It is easy to see that \(\text{Ult}(V_{\epsilon+1},\mathcal
U)\cong V_{\epsilon+1}\). Assume \(\epsilon\) has the cofinal embedding
property. Then moreover \(\text{Ult}(V_{\epsilon+2},\mathcal U)\cong
V_{\epsilon+2}\). Therefore we identify \(\text{Ult}(V_{\epsilon+2},\mathcal
U)\) with \(V_{\epsilon+2}\). As a consequence, for every \(X\in
V_{\epsilon+3}\), we can identify \(j_\mathcal U(X)\) with a subset of
\(V_{\epsilon+2}\); that is, we identify \(j_\mathcal U(X)\) with an element of
\(V_{\epsilon+3}\). Given the cofinal embedding property, it is not hard to show
that \(j_\mathcal U\restriction V_{\epsilon+3}\) is the only possible extension
of \(j\) to an elementary embedding of \(V_{\epsilon+3}\). Therefore one can set
\(j^{+} = j_\mathcal U\) instead of using \cref{SuccessorExtensionDef}, and then
prove \cref{ExtensionThm} and \cref{RepresentationThm} by very similar arguments
to the ones given above.

Obviously the two approaches to the canonical extension operation are very
similar, but we just want to highlight that the canonical extension operation,
like everything else in set theory, is really an ultrapower construction.
\subsection{Undefinability over \(V_\epsilon\)}\label{UndefinabilitySection}
In this section, we establish \cref{FirstPeriodicity} (1). At this point, we
have found three proofs, increasing chronologically in complexity and
generality. 

We begin by giving a sketch of the simplest of these proofs in the successor
ordinal case. (Note that the limit case is handled by Schlutzenberg's
\cref{SchlutzenbergUndefinableThm}.) Suppose \(\epsilon\) is an even ordinal and
\(j :V_{\epsilon+2}\to V_{\epsilon+2}\) is a nontrivial elementary embedding.
Let \(\mathcal U\) be the ultrafilter over \(V_{\epsilon+1}\) derived from \(j\)
using \(j[V_{\epsilon}]\). It turns out that if \(j\) is definable over
\(V_{\epsilon+2}\), then \(\mathcal U\) belongs to the ultrapower of \(V\) by
\(\mathcal U\). A fundamental fact from the ZFC theory of large cardinals,
proved for example in \cite{Kanamori}, is that no countably complete ultrafilter
belongs to its own ultrapower. The idea of the proof of \cref{UndefinabilityThm}
is to try to push this through to the current context. 

Recall that the {\it Mitchell order} is defined on countably complete
ultrafilters \(U\) and \(W\) by setting \(U\mo W\) if \(U\in \text{Ult}(V,W)\).
The ``fundamental fact'' mentioned above simply states that assuming the Axiom
of Choice, the Mitchell order is irreflexive. (In fact, it is wellfounded.) In
the context of ZF, especially given the failure of \L o\'s's Theorem, the
Mitchell order is fairly intractable, and in particular, we do not know how to
prove its irreflexivity. Instead, we use a variant of the Mitchell order called
the internal relation (introduced in the author's thesis \cite{UA}) that is more
amenable to combinatorial arguments.
\begin{defn}\label{IDef}
Suppose \(U\) and \(W\) are countably complete ultrafilters over sets \(X\) and
\(Y\). We say \(U\) is {\it internal to \(W\)}, and write \(U\I W\), if there is
a sequence of countably complete ultrafilters \(\langle U_y \mid y\in Y\rangle\)
such that for any relation \(R\subseteq X\times Y\):\footnote{For any predicate
\(P\), we write ``\(\forall^U x\ P(x)\)'' to mean that \(\{x \in X : P(x)\}\in
U\).}
\begin{equation}\label{IEquivalence}\forall^U x\ \forall^W y\ R(x,y) \iff \forall^W y\ \forall^{U_y} x\ R(x,y)\end{equation}
\end{defn}
Using notation that is standard in ultrafilter theory, \cref{IEquivalence} states that \(U\times W\) is canonically isomorphic to \(W\text{-}\sum_{y\in Y}U_y\).

Of course, from this combinatorial definition, it is not clear that the internal relation is
related to the Mitchell order at all. Using \L o\'s's Theorem, however, the following is easy to verify:
\begin{defn}
If \(U\) and \(W\) are ultrafilters, then the {\it pushforward of \(U\) to \(M_W\)} is the \(M_W\)-ultrafilter
\(s_W(U) = \{A\in j_W(P(X)) \mid (j_W)^{-1}[A]\in U\}\).
\end{defn}
\begin{prp}[ZFC]\label{InternalChar}
Suppose \(U\) and \(W\) are countably complete ultrafilters. Let \(X\) be the underlying set of \(U\). Then the following are equivalent:
\begin{enumerate}[(1)]
\item \(U\I W\).
\item \(s_W(U)\in M_W\).
\item \(j_U\restriction \textnormal{Ult}(V,W)\) is definable from parameters over \(\textnormal{Ult}(V,W)\).\qed
\end{enumerate}
\end{prp}
(3) is not the right perspective when \L o\'s's Theorem does not hold for \(W\). The equivalence of (1) and (2), however, is essentially a consequence of ZF:
\begin{lma}\label{sLemma}
Suppose \(U\) and \(W\) are countably complete ultrafilters over \(X\) and \(Y\). 
Then \(U\I W\) if and only if there is a sequence of countably complete ultrafilters
\(\langle U_y\rangle_{y\in Y}\) such that \([\langle U_y\rangle_{y\in Y}]_W = s_W(U)\).
\begin{proof}
One shows that \([\langle U_y\rangle_{y\in Y}]_W = s_W(U)\) if and only if 
the equivalence \cref{IEquivalence} from \cref{IDef} holds.

For the forwards direction, assume \([\langle U_y\rangle_{y\in Y}]_W = s_W(U)\). 
Fix \(R\subseteq X\times Y\). 
We verify the equivalence \cref{IEquivalence} from \cref{IDef}. 

Suppose that
\(\forall^W y\ \forall^{U_y}x\ R(x,y)\).
Then \(R^y\in U_y\) for \(W\)-almost all \(y\in Y\).\footnote{
	For \(y\in Y\), \(R^y\) denotes the set \(\{x\in X \mid (x,y)\in R\}\).}
By assumption, this means \([\langle R^y\rangle_{y\in Y}]_W\in s_W(U)\), 
so by the definition of \(s_W(U)\), \((j_W)^{-1}([\langle R^y\rangle_{y\in Y}]_W) \in U\).
In other words, \(\{x\in X \mid \forall^Wy\ x\in R^y\}\in U\),
which means that \(\forall^U x\ \forall^W y\ R(x,y)\).

It follows immediately that \(\forall^U x\ \forall^W y\ R(x,y)\) implies
\(\forall^W y\ \forall^{U_y}x\ R(x,y)\): notice that
\(Z = \{R\subseteq X\times Y : \forall^U x\ \forall^W y\ R(x,y)\}\)
and \(Z' = \{R\subseteq X\times Y : \forall^W y\ \forall^{U_y} x\ R(x,y)\}\)
are both ultrafilters, so since the previous paragraph shows \(Z'\subseteq Z\),
in fact, \(Z = Z'\).

We omit the proof that \cref{IEquivalence} from \cref{IDef} 
implies \([\langle U_y\rangle_{y\in Y}]_W = s_W(U)\), 
since the proof is straightforward
and the result is never actually cited.
\end{proof}
\end{lma}
The key advantage of the internal relation over the Mitchell order
is that its irreflexivity can be proved in ZF
by a combinatorial argument that will be given in the appendix:
\begin{repthm}{InternalIrrefl}
Suppose \(U\) is a countably complete ultrafilter and \(\textnormal{crit}(j_U)\) exists. Then \(U\not \I U\).
\end{repthm}

With this in hand, we can prove the undefinability theorem.

\begin{prp}\label{UndefinabilityThm}
Suppose \(\epsilon\) is an even ordinal and \(j :V_{\epsilon+2}\to V_{\epsilon+2}\) is a nontrivial elementary embedding. Then \(j\) is not definable from parameters over \(V_{\epsilon+2}\).
\begin{proof}
Assume towards a contradiction that \(j\) is definable over \(V_{\epsilon+2}\) from parameters. Let \(\mathcal U\) be the ultrafilter over \(V_{\epsilon+1}\) derived from \(j\) using \(j[V_{\epsilon}]\). Clearly \(\mathcal U\) is definable over \(V_{\epsilon+2}\) from \(j\), and hence \(\mathcal U\) is definable over \(V_{\epsilon+2}\) from parameters.

Let \(k : \text{Ult}(V_{\epsilon+2},\mathcal U)\to V_{\epsilon+2}\) be the canonical factor embedding defined by \(k([f]_\mathcal U) = j^\star(f)(j[V_\epsilon])\). The elementarity of \(j^\star\) implies that \(k\) is a well-defined injective homomorphism of structures. Moreover, \cref{RepresentationLma} implies that \(k\) is surjective. So \(\text{Ult}(V_{\epsilon+2},\mathcal U)\) is canonically isomorphic to \(V_{\epsilon+2}\), and we will identify the two structures.

Fix a formula \(\varphi(v,w)\) and a parameter \(B\in V_{\epsilon+2}\) such that \[\mathcal U = \{A\in V_{\epsilon+2} \mid V_{\epsilon+2}\vDash \varphi(A,B)\}\]
Let \(\mathcal U_\sigma = \{A\in V_{\epsilon+2} \mid V_{\epsilon+2}\vDash \varphi(f_A(\sigma),f_B(\sigma))\}\) where \(f_A\) and \(f_B\) are as defined in \cref{RepresentationDef}.

We claim that \([\langle\mathcal U_\sigma \mid {\sigma\in V_{\epsilon+1}}\rangle]_\mathcal U = s_{\mathcal U}(\mathcal U)\).
This follows from the elementarity of \(j\), which implies 
\begin{align*}
[\langle\mathcal U_\sigma \mid {\sigma\in V_{\epsilon+1}}\rangle]_\mathcal U &= \{A\in V_{\epsilon+2} \mid V_{\epsilon+2}\vDash \varphi(f_A(j[V_\lambda]),B)\} \\
&= \{A\subseteq V_{\epsilon+1} \mid j^{-1}[A]\in \mathcal U\} \\&= s_\mathcal U(\mathcal U)\end{align*}
It is also easy to show that \(\mathcal U_\sigma\) is a countably complete ultrafilter for \(\mathcal U\)-almost all \(\sigma\in V_{\epsilon+1}\). By \cref{sLemma}, \(\langle \mathcal U_\sigma \mid \sigma\in V_{\epsilon+1}\rangle\) witnesses \(\mathcal U\I \mathcal U\). Since \(j_\mathcal U\restriction \epsilon = j\restriction \epsilon\), \(j_\mathcal U\) has a critical point. This contradicts \cref{InternalIrrefl}.
\end{proof}
\end{prp}

\subsection{Supercompactness properties of ultrapowers}
The proof of \cref{UndefinabilityThm} raises a number of questions 
that also arise naturally in the study of choiceless cardinals. 
The one we will focus on concerns the supercompactness properties of ultrapowers. 
Suppose \(\epsilon\) is an even ordinal, \(j :V_{\epsilon+2}\to V_{\epsilon+2}\) is an elementary embedding, 
and \(\mathcal U\) is the ultrafilter over \(V_{\epsilon+1}\)
derived from \(j\) using \(j[V_{\epsilon}]\). 
Motivated by the proof of \cref{RepresentationThm}, 
one might ask when \(j_\mathcal U[S]\in M_\mathcal U\) for various sets \(S\). 
This is related to the definability of elementary embeddings 
because if \(S\in V_\beta\) is transitive and \(j\) extends
to an elementary embedding \(i : V_\beta\to N\),
then \(j_\mathcal U[S]\in M_\mathcal U\)
if and only if \(i\restriction S\) is definable over \(N\)
from parameters in \(i[V_\beta]\cup V_{\epsilon+2}\). 

Of course, by \cref{RepresentationThm}, \(j_\mathcal U[V_{\epsilon+1}]\) belongs to \(M_\mathcal U\). It follows easily that \(j_\mathcal U [S]\in M_\mathcal U\) for all \(S \preceq^{*} V_{\epsilon+1}\). (See \cref{CardinalityEqn} for this notation.) The converse remains open: if \(j_\mathcal U[S]\in M_\mathcal U\), must \(S\preceq^{*} V_{\epsilon+1}\)? 
We still do not know the answer to this question, even in the following special case: with \(\mathcal U\) as above, it is not hard to show that \(j_\mathcal U[P(\epsilon^+)]\) belongs to \(M_\mathcal U\), yet it is far from clear whether one can prove \(P(\epsilon^+)\preceq^{*} V_{\epsilon+1}\) without making further assumptions. Of course, assuming the Axiom of Choice, if \(U\) is an ultrafilter over \(X\), \(S\) is a set, and \(j_U[S]\in M_U\), then \(|S|\leq |X|\). (See \cite[Proposition 4.2.31]{UA}.) We are simply asking whether a very special case of this fact can be proved in ZF.

We begin by proving a general theorem that subsumes the ZFC result mentioned above:
\begin{thm}\label{UltrafilterUndefinabilityThm}
Suppose \(X\) is a set such that \(X\times X\preceq^{*} X\). Suppose \(U\) is an ultrafilter over \(X\), \(\kappa = \textnormal{crit}(j_U)\), and \(\alpha\geq \kappa\) is an ordinal. If \(j_U[\alpha]\in M_U\), then \(\alpha \preceq^{*} X\).
\end{thm}
Note that we implicitly assume that \(U\) has a critical point. 
\begin{defn}
Suppose \(U\) is an ultrafilter over a set \(X\). A sequence \(\langle S_x \mid x\in X\rangle\) is an {\it \(A\)-supercompactness sequence for \(U\)} if it has the following properties:
\begin{itemize}
\item For all \(x\in X\), \(S_x\subseteq A\).
\item Every \(a\in A\) belongs to \(S_x\) for \(U\)-almost all \(x\in X\).
\item For any \(X\)-indexed sequence \(\langle a_x \mid x\in X\rangle\) with \(a_x\in S_x\) for \(U\)-almost all \(x\in X\), there is some \(a\in A\) with \(a_x = a\) for \(U\)-almost all \(x\in X\).\footnote{In the case that \(A\) is not wellorderable, the right concept seems to be that of a {\it normal} \(A\)-supercompactness sequence, which has the stronger property that for any \(X\)-indexed sequence \(\langle B_x \mid x\in X\rangle\) with \(\emptyset \neq B_x\subseteq S_x\) for \(U\)-almost all \(x\in X\), there is some \(a\in A\) with \(a\in B_x\) for \(U\)-almost all \(x\in X\).}
\end{itemize}
\end{defn}
Supercompact ultrafilters are the combinatorial manifestation of supercompact embeddings:
\begin{lma}\label{SuperUltra}
\(\langle S_x \mid x\in X\rangle\) is an \(A\)-supercompactness sequence for \(U\) if and only if \([\langle S_x \mid x\in X\rangle]_U = j_U[A]\).\qed
\end{lma}

The first step of \cref{UltrafilterUndefinabilityThm} is the following well-known fact:
\begin{lma}\label{CompletenessSurjection}
Suppose \(U\) is an ultrafilter over a set \(X\) and \(\kappa = \textnormal{crit}(j_U)\). Then there is a surjection from \(X\) to \(\kappa\).
\begin{proof}
Since \(U\) is \(\kappa\)-complete but not \(\kappa^+\)-complete, there is a strictly decreasing sequence \(\langle A_\alpha \mid \alpha < \kappa\rangle\) of sets \(A_\alpha\in U\) such that \(\bigcap_{\alpha < \kappa} A_\alpha = \emptyset\). Let \(f : X\to \kappa\) be defined by \[f(x) = \min \{\alpha \mid x\notin A_{\alpha+1}\}\] Since the sequence \(\langle A_\alpha \mid \alpha < \kappa\rangle\) is strictly decreasing, \(f\) is a surjection.
\end{proof}
\end{lma}
The second step of \cref{UltrafilterUndefinabilityThm} is more involved:
\begin{lma}\label{SuperLemma}
Suppose \(U\) is an ultrafilter over a set \(X\), \(\kappa = \textnormal{crit}(j_U)\), \(\gamma\geq \kappa\) is an ordinal, and \(\langle S_x \mid x\in X\rangle\) is a \(\gamma\)-supercompactness sequence for \(U\). Then for any surjection \(p\) from \(X\) to \(\kappa\), there is a cofinal function from \(X\) to \(\gamma\) that is ordinal definable from \(\langle S_x \mid x\in X\rangle\) and \(p\).
\end{lma}
In fact, the function produced by \cref{SuperLemma} will be \(\Sigma_1\)-definable from \(\langle S_x\rangle_{x\in X}\) and \(p\), but ordinal definability will suffice for our applications.
\begin{proof}
There are two cases.

\begin{case}\label{EasyCase} For \(U\)-almost all \(x\in X\), \(\sup S_x < \gamma\).\end{case} 
It does no harm to assume that \(\sup S_x < \gamma\) for all \(x\in X\). This is because the sequence obtained from \(\langle S_x \mid x\in X\rangle\) by replacing \(S_x\) with the empty set whenever \(\sup S_x\geq \gamma\) is a supercompactness sequence (and is ordinal definable from \(\langle S_x \mid x\in X\rangle\)).  

Let \(f : X \to \gamma\) be the function defined by \[f(x) = \text{sup}(S_x\cap \gamma)\] Then \(f\) is cofinal in \(\gamma\), which proves \cref{SuperLemma} in \cref{EasyCase}. To see that \(f\) is cofinal, fix an ordinal \(\alpha < \gamma\). The definition of a supercompactness sequence implies that the set \(B_\alpha  = \{x \in X \mid \alpha\in S_x\}\) belongs to \(U\), so we may fix an \(x\in X\) such that \(\alpha\in S_x\). Then \(f(x) = \sup S_x > \alpha\). Thus \(f\) is cofinal in \(\gamma\).


\begin{case}\label{HardCase} For \(U\)-almost all \(x\in X\), \(\sup S_x = \gamma\).
\end{case}
As in \cref{EasyCase}, it does no harm to assume \(\sup S_x = \gamma\) for all \(x\in X\). 

The first step is to show that the sets \(S_x\) cannot have a common limit point of uniform cofinality \(\kappa\) in the following sense:
\begin{clm*}
Suppose \(\nu < \gamma\). There is no sequence of sets \(\langle E_x \mid x\in X\rangle\) such that for all \(x\in X\), \(E_x\subseteq S_x\cap \nu\), \(E_x\) has ordertype \(\kappa\), and \(E_x\) is cofinal in \(\nu\).
\begin{proof}
Fix one last cofinal set \(E\subseteq \nu\) of ordertype \(\kappa\). Let \[M_x = L[S_x,E_x,E]\] Since there is a definable sequence \(\langle {<}_x \mid x\in X\rangle\) such that \(<_x\) is a wellorder of \(M_x\), \L o\'s's Theorem holds for the ultraproduct \(M = \prod_{x\in X} M_x/U\). Thus \(M\) is a proper class model of ZFC, although it may be that \(M\) is illfounded. That being said, \(\gamma+1\) is contained in the wellfounded part of \(M\). (As usual, the wellfounded part of \(M\) is taken to be transitive.) Indeed, \([\langle S_x \mid x\in X\rangle]_U = j_U[\gamma]\) by \cref{SuperUltra}, so \(M\) contains a wellorder of ordertype \(\gamma\) and hence is wellfounded up to \(\gamma+1\).

The purpose of including the set \(E\) in each \(M_x\) is to ensure that 
\(M\) correctly computes the cofinality of \(\sup j_U[\nu]\). 
The argument is standard, at least in the context of the Axiom of Choice. 
Since \(E\in M_x\) for every \(x\in X\),
\(j_U[E] = j_U(E)\cap j_U[\gamma]\) belongs to \(M\). 
Since \(\text{ot}(E) = \kappa\), \(\text{ot}(j_U[E]) = \kappa\). 
Therefore \(M\) satisfies that \(\sup j_U[E]\) has cofinality \(\kappa\). 
Since \(E\) is a cofinal subset of \(\nu\), \(\sup j_U[E] = \sup j_U[\nu]\). Thus: 
\begin{equation}\label{cfnu}\text{cf}^M(\sup j_U[\nu]) = \kappa\end{equation}

Let \(E_* = [\langle E_x \mid x\in X\rangle]_U\). Then \L o\'s's Theorem implies
that the following hold in \(M\):
\begin{itemize}
\item \(E_*\subseteq j_U[\nu]\).
\item \(E_*\) has ordertype \(j_U(\kappa)\).
\item \(E_*\) is cofinal in \(j_U(\nu)\).
\end{itemize} 
The first bullet point uses that \(j_U[\gamma]\cap j_U(\nu) = j_U[\nu]\). 

Since \(E_*\subseteq j_U[\nu]\), \(j_U[\nu]\) is cofinal in \(j_U(\nu)\), and
hence \[\sup j_U[\nu] = j_U(\nu)\] Combining this with \cref{cfnu},
\[\text{cf}^M(j_U(\nu)) = \kappa\] But \(\text{cf}(\nu) = \kappa\) in \(M_x\)
for every \(x\in X\), so by \L o\'s's Theorem
\[\text{cf}^M(j_U(\nu)) = j_U(\kappa)\] Thus \(\kappa = j_U(\kappa)\). This
contradicts that \(\kappa\) is the critical point of \(j_U\).
\end{proof}
\end{clm*}

We now sketch the last idea of the proof. Let \(C_x\) denote the set of limit
points of \(S_x\). Suppose towards a contradiction that there is no cofinal
function from \(X\) to \(\gamma\). Then the intersection \(\bigcap_{x\in X}C_x\)
is a closed unbounded subset of \(\gamma\), and hence should contain a point
\(\nu\) of cofinality \(\kappa\). This almost contradicts the claim. Therefore
to finish, we carefully go through the standard proof that \(\bigcap_{x\in
X}C_x\) is closed unbounded, checking that it either produces a cofinal function
from \(X\) to \(\gamma\) that is (ordinal) definable from \(\langle
S_x\rangle_{x\in X}\) and \(p\) or else produces a genuine counterexample to the
claim. 

Now, the details. We define various objects by a transfinite recursion with
stages indexed by ordinals \(\alpha\). The first \(\alpha\) stages of the
construction produce ordinals \[\langle \delta_x^\beta \mid (x,\beta)\in X\times
\alpha\rangle\] such that for each \(x\in X\), \(\langle \delta_x^\beta \mid
\beta < \alpha\rangle\) is an increasing sequence of elements of \(S_x\). It
remains to define \(\delta_x^\alpha\) for each \(x\in X\). There are two
possibilities:
\begin{scase}
The set \(\{\delta_x^\beta + 1 \mid x\in X,\beta < \alpha\}\) is bounded below
\(\gamma\).
\end{scase}
In this case, set \[\delta^\alpha = \sup \{\delta_x^\beta + 1: x\in X,\beta <
\alpha\}\] and for each \(x\in X\): 
\[\delta_x^\alpha = \min (S_x\setminus \delta^\alpha)\]

\begin{scase}
The set \(\{\delta_x^\beta + 1 \mid x\in X,\beta < \alpha\}\) is cofinal in
\(\gamma\). 
\end{scase}
If this case arises, the construction terminates. 

The construction must halt at some stage \(\alpha_* \leq \kappa\). To see this,
assume towards a contradiction that it does not. Let \(E_x = \{\delta_x^\beta
\mid \beta < \kappa\}\). Since the construction did not halt at stage
\(\kappa\), \(E_x\) is bounded strictly below \(\gamma\). The construction
ensures that if \(\alpha < \beta <\kappa\), then \(\delta_x^\alpha <
\delta_x^\beta\) for any \(x,y\in X\). It follows that all the sets \(E_x\) have
the same supremum, say \(\nu\). But for every \(x\in X\), \(E_x\subseteq S_x\cap
\nu\), \(E_x\) has ordertype \(\kappa\), and \(E_x\) is cofinal in \(\nu\). This
contradicts the claim.

Suppose first that \(\alpha_*\) is a limit ordinal. Then for each \(\alpha <
\alpha_*\), let \(\delta_\alpha = \sup_{\beta < \alpha} \delta_x^\beta + 1\).
Since \(\alpha_*\) is the first stage at which the construction halts,
\(\delta_\alpha < \gamma\) for all \(\alpha < \alpha_*\). Let \[f(x) =
\delta_{p(x)}\] for those \(x\) such that \(p(x) < \alpha_*\). Since \(p\) is a
surjection from \(X\) to \(\kappa\), the range of \(f\) is equal to
\(\{\delta_\beta \mid \beta < \alpha_*\}\), which is cofinal in \(\gamma\).

Otherwise \(\alpha = \beta+1\) for some ordinal \(\beta\). Then of course the
function \[f(x) = \delta_x^\beta\] must be cofinal in \(\gamma\). 
\end{proof}

\cref{UltrafilterUndefinabilityThm} is a consequence of \cref{SuperLemma} and
the following elementary fact:
\begin{lma}\label{LadderSystem}
Suppose \(X\) is a set, \(\delta\) is an ordinal, and for each \(\gamma \leq
\delta\), \(f_\gamma\) is a cofinal function from \(X\) to \(\gamma\). Suppose
\(d : X \to X \times X\) is a surjection. Then there is a surjection \(g\) from
\(X\) to \(\delta\) that is ordinal definable from \(d\) and \(\langle f_\gamma
\mid \gamma \leq \delta\rangle\).
\begin{proof}
We will define a sequence \(\langle g_\gamma \mid \gamma \leq \delta\rangle\) by
recursion and set \(g = g_\delta\). For \(\alpha \leq \delta\), suppose
\(\langle g_\gamma \mid \gamma < \alpha\rangle\) is given. Define \(h : X\times
X\to \delta\) by setting \(h(x,y) = g_\gamma(y)\) where \(\gamma =
f_\alpha(x)\). Then let \(g_\alpha = h\circ d\). 
\end{proof}
\end{lma}

\begin{proof}[Proof of \cref{UltrafilterUndefinabilityThm}] By
\cref{CompletenessSurjection}, there is a surjection \(p : X\to \kappa\). For
each \(\gamma \leq \delta\), \(\langle S_x\cap \gamma \mid x\in X\rangle\) is a
\(\gamma\)-supercompactness sequence. Applying \cref{SuperLemma}, let \(f_\gamma
: X\to \gamma\) be the least cofinal function ordinal definable from \(p\) and
\(\langle S_x\cap \gamma \mid x\in X\rangle\). The hypotheses of
\cref{LadderSystem} are now satisfied by taking \(d : X\times X\to X\) to be any
surjection. As a consequence, \(\delta \preceq^{*} X\), which proves the
theorem.
\end{proof}

\subsection{Ordinal definability and ultrapowers}
We finally turn to a generalization of \cref{UltrafilterUndefinabilityThm} that
has no ZFC analog:
\begin{thm}\label{FixedPointUndefinabilityThm}
Suppose \(X\) is a set, \(\gamma\) is an ordinal, \(U\) is an ultrafilter over
\(X\times \gamma\), \(\kappa\) is the critical point of \(j_U\), and \(\theta
\geq \kappa\) is an ordinal. If \(j_U[\theta]\in M_U\) and \(j_U(\theta) =
\theta\), then \(\theta \preceq^{*} X\).
\end{thm}
Very roughly, this theorem says that supercompactness up to a fixed point
\(\theta\) cannot be the result of the wellorderable part of an ultrafilter.
In another sense, however, the proof of \cref{UltrafilterUndefinabilityThm} is
more general than that of \cref{FixedPointUndefinabilityThm}. A routine
modification of the proof of \cref{UltrafilterUndefinabilityThm} shows the
following fact:
\begin{thm}\label{ExternalUltrafilterUndefinabilityThm}
Suppose \(M\) is an inner model and \(X\in M\) is a set such that \(X\times
X\preceq^{*} X\) in \(M\). Suppose \(U\) is an \(M\)-ultrafilter over \(X\),
\(\kappa\) is the critical point of \(j_U\), and \(\alpha\geq \kappa\) is an
ordinal. If \(j_U[\alpha]\in \textnormal{Ult}(M,U)\), then \(\alpha \preceq^{*}
X\).
\end{thm}
The difference is that we do not require \(U\in M\). It is not clear that it is
possible to modify the proof of \cref{FixedPointUndefinabilityThm} to obtain
such a result.

The proof of \cref{FixedPointUndefinabilityThm} uses the following version of
Vop\v enka's Theorem, which for reasons of citation we reduce to Bukovsky's
Theorem:
\begin{thm}[Vop\v enka]\label{VopenkaThm}
Suppose \(X\) and \(T\) are sets. Then for any \(x\in X\),
\(\textnormal{HOD}_{T,x}\) is an \(\iota\)-cc generic extension of
\(\textnormal{HOD}_T\) where \(\iota\) is the least regular cardinal of
\(\textnormal{HOD}_{T,x}\) greater than or equal to \(\theta\{X\}\).
\begin{proof}
Recall that \(\theta\{X\}\) is the least ordinal not the surjective image of
\(X\).

By Bukovsky's Theorem (\cite{Usuba}, Fact 3.9), it suffices to verify that
\(\textnormal{HOD}_T\) has the uniform \(\iota\)-covering property in
\(\textnormal{HOD}_{T,x}\). This amounts to the following task. Suppose
\(\alpha\) and \(\beta\) are ordinals and \(f : \alpha \to \beta\) is a function
in \(\textnormal{HOD}_{T,x}\). We must find a function \(F : \alpha \to
P(\beta)\) in \(\text{HOD}_T\) such that for all \(\xi < \alpha\), \(F(\xi)\) is
a set of cardinality less than \(\iota\) containing \(f(\xi)\) as an element. 

Since \(f\) is \(\textnormal{OD}_{T,x}\), there is an \(\text{OD}_T\) function
\(g : \alpha\times X\to \beta\) such that \(g(\xi,x) = f(\xi)\) for all \(\xi <
\alpha\). Let \(F(\xi) = \{g(\xi,u) \mid u\in X\}\). Clearly \(F\) is
\(\text{OD}_T\), so \(F\in \text{HOD}_T\). Fix \(\xi < \alpha\). By definition,
\(f(\xi) = g(\xi,x) \in F(\xi)\). Finally, since \(F(\xi) \preceq^{*} X\),
\(\iota\not \preceq^{*} X\), and \(F(\xi)\) is wellorderable, \(|F(\xi)| <
\iota\).
\end{proof}
\end{thm}

\begin{proof}[Proof of \cref{FixedPointUndefinabilityThm}] Let \(j = j_U\).

Fix a function \(S : X\times \gamma\to P(\theta)\) such that \([S]_U =
j[\theta]\). (That is, \(S\) is a \(\theta\)-supercompactness sequence for
\(U\).) For any set \(T\), let
\[M_T = \prod_{(x,\xi)\in X\times \gamma}\text{HOD}_{T,x}/U\] Notice that
\(j[\theta]\in M_S\) since \(S(x,\xi)\in \text{HOD}_{S,x}\) for all \((x,\xi)\in
X\times \gamma\). It follows that for any set \(T\), \[P(\theta)\cap
\text{HOD}_{S,T}\subseteq M_{S,T}\] This is a basic fact about supercompactness,
recast in our context. The point is that if \(A\in P(\theta)\cap
\text{HOD}_{S,T}\), then \(j(A)\) and \(j\restriction \theta\) both belong to
\(M_{S,T}\), so \(A\in M_{S,T}\) since \(A = (j\restriction\theta)^{-1}[j(A)]\).

The key idea of the proof is to construct a sequence \(T = \langle T(\nu) \mid
\nu < \beta_*\rangle\) of subsets of \(\theta\) such that
\begin{equation}\label{UltraproductEqn}\{T(\nu) \mid \nu < \beta_*\} = P(\theta)\cap M_{S,T}\end{equation}
The construction proceeds by recursion. Suppose \(T\restriction \beta = \langle T(\nu) \mid \nu < \beta\rangle\) has been defined. Assume that \(\{T(\nu) \mid \nu < \beta\} \subsetneq P(\theta)\cap M_{S,T\restriction \beta}\), and let \(T_\beta\subseteq \theta\) be the least set in the canonical wellorder of \(P(\theta)\cap M_{S,T\restriction \beta}\) that does not belong to \(\{T(\nu) \mid \nu < \beta\}\).

Eventually, one must reach an ordinal \(\beta\) such that \(\{T(\nu) \mid \nu < \beta\} = P(\theta)\cap M_{S,T\restriction \beta}\): otherwise one obtains a sequence \(\langle T(\nu) \mid \nu \in \text{Ord}\rangle\) of distinct subsets of \(\theta\), violating the Replacement and Powerset Axioms. At the least such ordinal \(\beta\), the construction terminates, and one sets \(\beta_* = \beta\) and \(T = \langle T(\nu) \mid \nu < \beta_*\rangle\), securing \cref{UltraproductEqn}.

Let \(\delta\) be the least ordinal such that \((2^{\delta})^{\text{HOD}_{S,T}} > \theta\). We claim that 
\begin{equation}\label{PowerCompactClm}j[P(\delta)\cap \text{HOD}_{S,T}]\in M_{S,T}\end{equation}
Let \(P_{\text{bd}}(\delta)\) denote the set of bounded subsets of \(\delta\).
Since \((2^{<\delta})^{\text{HOD}_{S,T}}\leq \theta\) and \(j[\theta]\in
M_{S,T}\), \(j[P_{\text{bd}}(\delta)\cap \text{HOD}_{S,T}]\in M_{S,T}\) by a
standard argument: letting \(f : \theta\to P_{\text{bd}}(\delta)\cap
\text{HOD}_{S,T}\) be a surjection with \(f\in \text{HOD}_{S,T}\),
\(j[P_\text{bd}(\delta)\cap \text{HOD}_{S,T}] = j(f)[j[\theta]]\in M_{S,T}\).

We claim that \(j(\delta) = \sup j[\delta]\). This will imply
\cref{PowerCompactClm}, since then \(j[P(\delta)\cap \text{HOD}_{S,T}]\) is
equal to the set of \(A\in P(j(\delta))\cap M_{S,T}\) such that \(A\cap
\alpha\in j[P_{\text{bd}}(\delta)\cap \text{HOD}_{S,T}]\) for all \(\alpha <
j(\delta)\). Here we make essential use of the equality \cref{UltraproductEqn}.

Let \(\delta_* = \sup j[\delta]\). Let \[P = \text{Ult}(\text{HOD}_{S,T},U)\] so
\(j\) restricts to an elementary embedding from \(\text{HOD}_{S,T}\) to \(P\).
To prove that \(j(\delta) = \delta_*\), it suffices by the minimality of
\(\delta\) and the elementarity of \(j\) to show that \((2^{\delta_*})^P >
j(\theta)\), or, since \(j(\theta) = \theta\), that \((2^{\delta_*})^P >
\theta\). Suppose towards a contradiction that this is false, so
\((2^{\delta_*})^P \leq \theta\). 

Since \(j\restriction P_{\text{bd}}(\delta)\cap \text{HOD}_{S,T}\in M_{S,T}\),
\(j\restriction P_{\text{bd}}(\delta)\cap \text{HOD}_{S,T}\in
\text{HOD}_{S,T}\). Therefore the one-to-one function \(h : P(\delta)\cap
\text{HOD}_{S,T}\to P(\delta_*)\cap P\) defined by \(h(A) = j(A)\cap \delta_*\)
belongs to \(\text{HOD}_{S,T}\). Since \((2^{\delta_*})^P \leq \theta\), there
is an injective function from \(\text{ran}(h)\) to \(\theta\) in \(P\), hence in
\(M_{S,T}\), and hence in \(\text{HOD}_{S,T}\) by \cref{UltraproductEqn}.
Therefore in \(\text{HOD}_{S,T}\), 
\[P(\delta)\cap \text{HOD}_{S,T}\preceq \text{ran}(h)\preceq^{*} \theta\]

Since \(\text{HOD}_{S,T}\) satisfies the Axiom of Choice, it follows that
\((2^\delta)^{\text{HOD}_{S,T}} \leq \theta\), which contradicts the definition
of \(\delta\). This contradiction establishes that \(\delta_* = \sup
j[\delta]\), finishing the proof of \cref{PowerCompactClm}.

Let \(\iota = \theta^{+\text{HOD}_{S,T}}\). Then \[\iota \leq j(\iota) =
\theta^{+P}\leq \theta^{+M_{S,T}}\leq \iota\] The first equality uses that
\(j(\theta) = \theta\), and the final inequality follows from
\cref{UltraproductEqn}. Hence \(j(\iota) = \iota\). Since \(\iota \leq
(2^\delta)^{\textnormal{HOD}_{S,T}}\), \(j[\iota]\in M_{S,T}\) as an immediate
consequence of \cref{PowerCompactClm}. We omit the proof, which is similar to
the proof above that \(j[P_{\text{bd}}(\delta)\cap \text{HOD}_{S,T}]\in
M_{S,T}\).

We finally show that \(\theta \preceq^{*} X\). Assume towards a contradiction
that this fails. Then by \cref{VopenkaThm}, for every \(x\in X\),
\(\textnormal{HOD}_{S,T,x}\) is an \(\iota\)-cc generic extension of
\(\textnormal{HOD}_{S,T}\). By elementarity, it follows that \(M_{S,T}\) is an
\(\iota\)-cc generic extension of \(P\). In particular, \(P\) is stationary
correct in \(M_{S,T}\) at \(\iota\). This allows us to run Woodin's proof
\cite{Woodin} of the Kunen Inconsistency Theorem to reach our final
contradiction.

Let \(B = \{\xi < \iota \mid \text{cf}(\xi) = \omega\}\). Recall that \(\kappa\)
denotes the critical point of \(j\). Since \(\text{HOD}_{S,T}\) satisfies the
Axiom of Choice, the Solovay Splitting Theorem \cite{SolovayRVM} applied in
\(\text{HOD}_{S,T}\) yields a partition \(\langle B_\nu \mid \nu <
\kappa\rangle\) of \((B)^{\text{HOD}_{S,T}}\) into
\(\text{HOD}_{S,T}\)-stationary sets. Let \(\langle B'_\nu \mid \nu <
j(\kappa)\rangle = j(\langle B_\nu \mid \nu < \kappa\rangle)\). Then
\(B'_\kappa\) is \(P\)-stationary in \(\iota\). Therefore \(B'_\kappa\) is
\(M_{S,T}\)-stationary in \(\iota\) since \(P\) is stationary correct in
\(M_{S,T}\) at \(\iota\). Since \(j[\iota]\in M_{S,T}\) is an \(\omega\)-closed
unbounded set in \(M_{S,T}\) and \(M_{S,T}\) satisfies that \(B'_\kappa\) is a
stationary set of ordinals of cofinality \(\omega\), the intersection
\(j[\iota]\cap B'_\kappa\) is nonempty. Fix \(\xi < \iota\) such that
\(j(\xi)\in B'_\kappa\). Clearly \(\xi\in B\), so since \(\langle B_\nu \mid \nu
< \kappa\rangle\) partitions \(B\), there is some \(\nu < \iota\) such that
\(\xi\in B_\nu\). Now \(j(\xi)\in j(B_\nu) = B'_{j(\nu)}\). Therefore the
intersection \(B'_{j(\nu)}\cap B'_\kappa\) is nonempty, and so since \(\langle
B'_\nu \mid \nu < j(\kappa)\rangle\) is a partition, \(j(\nu) = \kappa\). This
contradicts that \(\kappa\) is the critical point of \(j\).
\end{proof}

As a corollary of \cref{FixedPointUndefinabilityThm}, we answer the following
question of Schlutzenberg. Suppose \(j : V\to V\) is an elementary embedding. Is
every set ordinal definable from parameters in the range of \(j\)? The question
is motivated by the well-known ZFC fact that if \(j : V\to M\) is an elementary
embedding, then every element in \(M\) is ordinal definable in \(M\) from
parameters in the range of \(j\). 

The answer to Schlutzenberg's question, however, is no.
\begin{thm}\label{FarmerThm}
Suppose \(\epsilon\leq \eta\leq \eta'\) are ordinals, \(\epsilon\) is even, and
\(j : V_\eta\to V_{\eta'}\) is a cofinal elementary embedding such that
\(j(\epsilon) = \epsilon\). Then \(j^\star\restriction \theta_{\epsilon}\) is
not definable over \(\mathcal H_{\eta'}\) from parameters in \(j^\star [\mathcal
H_{\eta}]\cup V_\epsilon\cup \theta_{\eta'}\).
\end{thm}
To put this theorem in a more familiar context, let us state a special case.
\begin{cor}
Suppose \(j : V\to V\) is an elementary embedding. Let
\(\lambda=\kappa_\omega(j)\). Then \(j[\lambda]\) is not ordinal definable from
parameters in \(j[V]\cup V_\lambda\).\qed
\end{cor}

\begin{proof}[Proof of \cref{FarmerThm}] Let \(\theta = \theta_{\epsilon}\).
Since \(j(\epsilon) = \epsilon\), \(j^\star(\theta) = \theta\). Suppose towards
a contradiction that the theorem fails. Then there is a set \(p\in V_{\eta}\), a
set \(a\in V_\epsilon\), an ordinal \(\alpha < \eta'\), and a formula
\(\varphi\) such that \(j^\star[\theta]\) is the unique set \(k\in \mathcal
H_{\eta'}\) such that \(\mathcal H_{\eta'}\) satisfies
\(\varphi(k,j(p),a,\alpha)\). By \cref{RepresentationLma} (or trivially if
\(\epsilon\) is a limit ordinal), there is an ordinal \(\xi\) such that
\(\xi+2\leq\epsilon\), a set \(x\in V_{\xi+1}\), and a function \(f :
V_{\xi+1}\to V_{\epsilon}\) such that \(j(f)(x) = a\). 

Let \(\gamma\) be an ordinal such that \(\alpha < j(\gamma)\). Define \(g :
V_{\xi+1}\times \gamma\to P(\theta)\) so that \(j^\star(g)(x,\alpha) =
j^\star[\theta]\): let \(g(u,\beta)\) be the unique set \(k\in \mathcal H_\eta\)
such that \(\mathcal H_\eta\) satisfies \(\varphi(k,p,f(u),\beta)\). Let \(U\)
be the ultrafilter over \(V_{\xi+1}\times \gamma\) derived from \(j\) using
\((x,\alpha)\). It is easy to check that \([g]_U = j_U[\theta]\): \([h]_U\in
[g]_U\) if and only if \(j(h)(x,\alpha)\in j(g)(x,\alpha)\) if and only if
\(j(h)(x,\alpha) = j(\nu)\) for some \(\nu < \theta\) if and only if there is
some \(\nu < \theta\) such that \(h(u,\beta) = \nu\) for \(U\)-almost all
\((u,\beta)\). By \cref{FixedPointUndefinabilityThm}, there is a surjection from
\(V_{\xi+1}\) to \(\theta\), and since \(V_{\xi+1}\in V_\epsilon\), this
contradicts the definition of \(\theta\).
\end{proof}

We include a final result \(\text{AD}_\mathbb R\)-like result about ordinal
definability assuming an elementary embedding from \(V_{\epsilon+3}\) to
\(V_{\epsilon+3}\):
\begin{thm}\label{NotODThm}
Suppose \(\epsilon\) is an even ordinal and there is a \(\Sigma_1\)-elementary
embedding from \(V_{\epsilon+3}\) to \(V_{\epsilon+3}\). Then there is no
sequence of functions \(\langle f_\alpha \mid \alpha <
\theta_{\epsilon+2}\rangle\) such that for all \(\alpha < \theta_{\epsilon+2}\),
\(f_\alpha\) is a surjection from \(V_{\epsilon+1}\) to \(\alpha\).
\end{thm}
This result cannot be proved from the existence of an elementary embedding \(j :
V_{\epsilon+2}\to V_{\epsilon+2}\) (if this hypothesis is consistent): the
inner model \(L(V_{\epsilon+1})[j]\) satisfies that there is an elementary
embedding from \(V_{\epsilon+2}\) to \(V_{\epsilon+2}\), but using that
\(L(V_{\epsilon+1})[j]\) satisfies that \(V = \text{HOD}_{V_{\epsilon+1},i}\)
where \(i = j\restriction L(V_{\epsilon+1})[j]\), one can easily show that in
\(L(V_{\epsilon+1})[j]\), there is a sequence \(\langle f_\alpha \mid \alpha <
\theta_{\epsilon+2}\rangle\) such that for all \(\alpha < \theta_{\epsilon+2}\),
\(f_\alpha\) is a surjection from \(V_{\epsilon+2}\) to \(\alpha\).

\begin{proof}[Proof of \cref{NotODThm}] Suppose towards a contradiction that
\(\langle f_\alpha \mid \alpha < \theta_{\epsilon+2}\rangle\) is such a
sequence. As a consequence of this assumption and \cref{LadderSystem},
\(\theta_{\epsilon+2}\) is regular. 

Note that \(\langle f_\alpha \mid \alpha < \theta_{\epsilon+2}\rangle\in
\mathcal H_{\epsilon+3}\). Using the notation from \cref{Preliminaries}, fix
\(u\in \text{C}_{\epsilon+3}\) such that \(\Phi_{\epsilon+3}(u) = \langle
f_\alpha \mid \alpha < \theta_{\epsilon+2}\rangle\). (Recall that
\(\text{C}_{\epsilon+3}\) is the set of codes in \(V_{\epsilon+3}\) for elements
of \(\mathcal H_{\epsilon+3}\).) 

For \(\ell = 0,1\), suppose \[j_\ell : (V_{\epsilon+2},u_\ell)\to
(V_{\epsilon+2},u)\] is an elementary embedding. Necessarily, \(u_\ell =
j_\ell^{-1}[u]\). Assume \(j_0\restriction V_{\epsilon} = j_1\restriction
V_{\epsilon}\). We claim that \(j_0^\star[\theta_{\epsilon+2}] =
j_1^\star[\theta_{\epsilon+2}]\). First, let \(F\subseteq \theta_{\epsilon+2}\)
be the set of common fixed points of \(j_0\restriction \theta_{\epsilon+2}\) and
\(j_1\restriction \theta_{\epsilon+2}\). Since \(\theta_{\epsilon+2}\) is
regular, \(F\) is \(\omega\)-closed unbounded. For each \(\alpha\in F\), there
is some \(g_\alpha^\ell\in \mathcal H_{\epsilon+2}\) such that
\(j_\ell^\star(g_\alpha^\ell) = f_\alpha\). Indeed, one can set \(g_\alpha^\ell
= \Phi(u_\ell)_\alpha\). Now \(j_0^\star[\alpha] =
j_0^\star[g_\alpha^\ell[V_{\epsilon+1}]] = f_\alpha[j_0[V_{\epsilon+1}]]\).
Similarly \(j_1^\star[\alpha] = f_\alpha[j_1[V_{\epsilon+1}]]\). Since
\(j_0\restriction V_\epsilon = j_1\restriction V_\epsilon\), \(j_0\restriction
V_{\epsilon+1} = j_1\restriction V_{\epsilon+1}\) by \cref{ExtensionThm}. It
follows that \(j_0^\star[\alpha] = j_1^\star[\alpha]\). Since \(F\) is unbounded
in \(\theta_{\epsilon+2}\), this implies \(j_0^\star[\theta_{\epsilon+2}] =
j_1^\star[\theta_{\epsilon+2}]\).

Let 
\begin{align*}\mathcal E_0 &= \{k \mid k \in \mathscr E((V_{\epsilon+2},k^{-1}[u]),(V_{\epsilon+2},u))\}\\
\mathcal E_1 &= \{k\restriction V_{\epsilon} \mid k\in \mathcal E_1\}\end{align*} 
Let \(C = \bigcap_{k\in \mathcal E_0} k^\star[\theta_{\epsilon+2}]\). For any
\(i\in \mathcal E_1\), let \(A_i\subseteq \theta_{\epsilon+2}\) be equal to
\(k^\star[\theta_{\epsilon+2}]\) for any \(k\in \mathcal E_1\) extending \(i\).
This is well-defined by the previous paragraph. Clearly \(C = \bigcap \{A_i \mid
i\in \mathcal E_0\}\). Since \(\theta_{\epsilon+2}\) is regular and \(\mathcal
E_0\preceq^{*} V_{\epsilon+1}\), it follows that \(C\) is \(\omega\)-closed
unbounded in \(\theta_{\epsilon+2}\).

Now suppose \(j : V_{\epsilon+3}\to V_{\epsilon+3}\) is \(\Sigma_1\)-elementary.
Then 
\begin{align}j^\star(\mathcal E_0) &= \{k \mid k \in \mathscr E((V_{\epsilon+2},k^{-1}[u]),(V_{\epsilon+2},u))\}\label{ECorrect}\\
j^\star(C) &= \bigcap_{k\in j^\star(\mathcal E_0)} k^\star[\theta_{\epsilon+2}]\label{CCorrect}
\end{align}
Verifying these equalities is a bit tricky since we only know that \(j^\star:
\mathcal H_{\epsilon+3}\to \mathcal H_{\epsilon+3}\) is \(\Sigma_0\)-elementary.
(See \cref{LocalElementarity}.) \cref{ECorrect} is proved by writing \(\mathcal
E_0 = \bigcap_{n < \omega} \mathcal E_0^n\) where \[\mathcal E_0^n = \{k \mid k
: (V_{\epsilon+2},k^{-1}[u])\to_{\Sigma_n}(V_{\epsilon+2},u)\}\] Notice that
\(\mathcal E_0^n\) is \(\Sigma_0\)-definable over \(\mathcal H_{\epsilon+3}\)
from \(u\) and \(V_{\epsilon+2}\), and \(j^\star(\mathcal E_0) = \bigcap_{n <
\omega} j^\star(\mathcal E_0^n)\). This easily yields \cref{ECorrect}.
\cref{CCorrect} is proved by checking that the \(\star\)-operation on \(\mathscr
E(V_{\epsilon+2})\) is \(\Sigma_0\)-definable over \(\mathcal H_{\epsilon+3}\)
from the parameter \(\mathcal H_{\epsilon+2}\).

It follows that \(j\restriction V_{\epsilon+2}\in j^\star(\mathcal E_0)\), and
hence \(j^\star(C)\subseteq j^\star[\theta_{\epsilon+2}]\). The only way this is
possible is if \(|C| < \text{crit}(j)\). But \(C\) is unbounded in
\(\theta_{\epsilon+2}\). This contradicts that \(\theta_{\epsilon+2}\) is
regular.
\end{proof}
The Axiom of Choice implies the existence of a sequence \(\langle f_\alpha \mid
\alpha < \theta_{\epsilon+2}\rangle\) such that for all \(\alpha <
\theta_{\epsilon+2}\), \(f_\alpha\) is a surjection from \(V_{\epsilon+1}\) to
\(\alpha\), so \cref{NotODThm} yields a new proof of the Kunen Inconsistency
Theorem.
\section{The \(\theta_\alpha\) sequence}\label{ThetaSection}
\subsection{The main conjecture}
In this section we study the sequence of cardinals \(\theta_\alpha\)
(\cref{ThetaDef}). The results we will prove suggest that if \(\epsilon\) is an
even ordinal, then assuming choiceless large cardinal axioms,
\(\theta_{\epsilon}\) should be relatively large and \(\theta_{\epsilon+1}\)
should be relatively small.
\begin{conj}\label{ThetaConj}
Suppose \(\epsilon\) is an even ordinal and there is an elementary embedding
from \(V_{\epsilon+1}\) to \(V_{\epsilon+1}\).
\begin{itemize}
\item \(\theta_{\epsilon}\) is a strong limit cardinal.\footnote{Recall that in the context of ZF, a cardinal is defined to be a {\it strong limit cardinal} if it is not the surjective image of the powerset of any smaller cardinal.}
\item \(\theta_{\epsilon+1} = (\theta_{\epsilon})^+\).
\end{itemize} 
\end{conj}
We note that if \(\epsilon\) is a limit ordinal, then one can prove in ZF that
\(\theta_{\epsilon}\) is a strong limit cardinal and \(\theta_{\epsilon+1} =
(\theta_{\epsilon})^+\).
\subsection{ZF theorems}
Our first two theorems towards \cref{ThetaConj} are the following:

\begin{thm}\label{EvenSmallThm}
Suppose \(\epsilon\) is an even ordinal. Suppose \(j : V_{\epsilon+3}\to
V_{\epsilon+3}\) is an elementary embedding with critical point \(\kappa\). Then
the interval \((\theta_{\epsilon+2},\theta_{\epsilon+3})\) contains fewer than
\(\kappa\) regular cardinals.
\end{thm}
This theorem shows that \(\theta_{\epsilon+3}\) is not too much larger than
\(\theta_{\epsilon+2}\). \cref{RegularThm} shows that under further assumptions,
\(\theta_{\epsilon+2}\) is an inaccessible limit of regular cardinals, so
\cref{EvenSmallThm} captures a genuine difference between the even and odd
levels.

\begin{thm}\label{OddBigThm}
Suppose \(\epsilon\) is an even ordinal. Suppose \(j : V_{\epsilon+2}\to
V_{\epsilon+2}\). Then for any \(\alpha < \kappa_{\omega}(j)\), there is no
surjection from \(P(\theta_{\epsilon+1}^{+\alpha})\) onto
\(\theta_{\epsilon+2}\).
\end{thm}
While we cannot show that \(\theta_{\epsilon+2}\) is a strong limit, this
theorem shows that it has some strong limit-like properties.
\cref{StrongLimitThm} below proves the stronger fact that \(V_{\epsilon+1}\)
surjects onto \(P(\alpha)\) for all \(\alpha < \theta_{\epsilon+2}\), but this
theorem requires weak choice assumptions. 

The following lemma, which is a key aspect of the proof of both
\cref{EvenSmallThm} and \cref{OddBigThm}, roughly states that rank-to-rank
embeddings have no generators in the interval
\((\theta_{\epsilon+2},\theta_{\epsilon+3})\).
\begin{lma}\label{NoGeneratorLma}
Suppose \(\epsilon\) is an even ordinal and \(j : V_{\epsilon+1}\to
V_{\epsilon+1}\) is an elementary embedding. Then for every ordinal \(\nu <
\theta_{\epsilon+1}\), there are ordinals \(\alpha,\beta < \theta_{\epsilon}\)
and a function \(g : \alpha\to \theta_{\epsilon+1}\) such that \(\nu =
j^\star(g)(\beta)\).
\begin{proof}
Let \(R\) be a prewellorder of \(V_{\epsilon}\) of length \(\nu+1\). Then
\(j(R)\) is a prewellorder of \(V_{\epsilon}\) of length at least \(\nu+1\). Fix
\(a\in V_{\epsilon}\) with \(\text{rank}_{j(R)}(a) = \nu\) in \(R\). By
\cref{RepresentationThm}, find an ordinal \(\xi\) such that \(\xi+2\leq
\epsilon\) and \(a = j^\star(f)(x)\) for some \(f : V_{\xi+1}\to V_{\epsilon}\)
and \(x\subseteq V_\xi\). Define \(d : V_{\xi+1}\to \theta_{\epsilon+1}\) by
setting \(d(u) = \text{rank}_R(f(u))\). Then \(j^\star(d)(x) = \nu\), so
\(\nu\in \text{ran}(j^\star(d))\). Let \(g : \alpha\to \theta_{\epsilon+1}\) be
the order-preserving enumeration of \(\text{ran}(d)\), and note that \(\alpha <
\theta_{\epsilon}\) since \(d\) witnesses \(\text{ran}(d)\preceq^{*}
V_{\xi+1}\). By elementarity, \(j^\star(g)\) enumerates
\(\text{ran}(j^\star(d))\), and hence there is some \(\beta <
\theta_{\epsilon}\) such that \(j^\star(g)(\beta) = \nu\).
\end{proof}
\end{lma}

\begin{proof}[Proof of \cref{EvenSmallThm}] Let \(\eta\) be the ordertype of the
set of regular cardinals in the interval
\((\theta_{\epsilon+2},\theta_{\epsilon+3})\). Then \(\eta\) is fixed by
\(j^\star\), so \(\eta\neq \kappa\). Suppose towards a contradiction that \(\eta
> \kappa\). Let \(\delta\) be the \(\kappa\)-th regular cardinal in
\((\theta_{\epsilon+2},\theta_{\epsilon+3})\). Then \(j^\star(\delta)\) is a
regular cardinal strictly above \(\delta\), so \(j^\star(\delta)\) is not equal
to \(\sup j^\star[\delta]\), which has cofinality \(\delta\). Therefore \(\sup
j^\star[\delta] < j^\star(\delta)\). By \cref{NoGeneratorLma}, there are
ordinals \(\alpha,\beta < \theta_{\epsilon+2}\) and a function \(g : \alpha \to
\delta\) such that \(j^\star(g)(\beta) = \sup j^\star[\delta]\). Since
\(\delta\) is regular, there is some ordinal \(\rho < \delta\) such that
\(\text{ran}(g) \subseteq \rho\). Therefore \(j^\star(g)(\beta) < j^\star(\rho)
< \sup j^\star[\delta]\), which is a contradiction.
\end{proof}

We now turn to the size of \(\theta_{\epsilon+2}\) for \(\epsilon\) even.
\begin{proof}[Proof of \cref{OddBigThm}] Let \(E = \langle D(j,a) \mid a \in
[\theta_{\epsilon}]^{<\omega}\rangle\) be the extender of length
\(\theta_{\epsilon}\) derived from \(j\). Notice that \(E\) is definable over
\(\mathcal H_{\epsilon+2}\) from \(j^\star\restriction
P_{\text{bd}}(\theta_{\epsilon})\), hence from \(j\restriction V_{\epsilon+1}\),
and hence from \(j\restriction V_{\epsilon}\) by \cref{ExtensionThm}. In fact,
there is a partial sequence \(\langle F(\sigma) \mid \sigma \in
V_{\epsilon+1}\rangle\) definable without parameters over \(\mathcal
H_{\epsilon+2}\) such that \(E = F({j[V_{\epsilon}]})\): to be explicit,
\(F(\sigma)\) is the extender of length \(\theta_{\epsilon}\) derived from \(k\)
where \(k = ((\pi_\sigma)^{+})^\star\) for \(\pi_\sigma : \sigma\to M\) the
Mostowski collapse of \(\sigma\).

Let \(j_E : \mathcal H_{\epsilon+2}\to N\) be the associated ultrapower
embedding, and let \(k : N\to \mathcal H_{\epsilon+2}\) be the associated factor
embedding, defined by \(k([f,a]_E) = j^\star(f)(a)\). Let \(\nu =
\text{crit}(k)\). Note that \(\nu\) must exist or else \(j_E\restriction
\theta_{\epsilon+2} = j^\star\restriction \theta_{\epsilon+2}\), contrary to the
undefinability of \(j^\star\restriction \theta_{\epsilon+2}\) over \(\mathcal
H_{\epsilon+3}\) from parameters, like \(E\), that lie in \(V_{\epsilon+2}\)
(\cref{SuperUltra}).

Note that \(\nu\) is a generator of \(j^\star\), in the sense that for any
ordinals \(\alpha,\beta < \nu\) and any function \(g : \alpha\to
\theta_{\epsilon+2}\), \(\nu \neq j^\star(g)(\beta)\): otherwise \(\nu\in
\text{ran}(k)\) by definition. As an immediate consequence of
\cref{NoGeneratorLma}, it follows that \(\nu \geq \theta_{\epsilon+1}\). Since
\(j^\star\) fixes every cardinal in the interval
\((\theta_{\epsilon+1},\theta_{\epsilon+1}^{+\kappa})\), it follows that
\(\nu\geq \theta_{\epsilon+1}^{+\kappa}\).

Notice that for any \(\eta < \theta_{\epsilon+1}^{+\kappa}\) and any set
\(A\subseteq \eta\), \(j_E(A) = j^\star(A)\). Indeed, \(j_E(A)\cap \nu =
j^\star(A)\cap \nu\) for any set of ordinals \(A\). Suppose towards a
contradiction that \(p : P(\eta)\to \theta_{\epsilon+2}\) is a surjection.
Define \(g : V_{\epsilon+1}\to P(\theta_{\epsilon+2})\) by setting \(g(\sigma) =
p\circ j_{F(\sigma)}[P(\eta)]\). 

Let \(\mathcal U\) be the ultrafilter over \(V_{\epsilon+1}\) derived from \(j\)
using \(j[V_{\epsilon}]\). It is easy to check that \(j_\mathcal U\restriction
\mathcal H_{\epsilon+2} = j^\star\restriction \mathcal H_{\epsilon+2}\).
Therefore \([g]_\mathcal U = j_\mathcal U(g)(j[V_{\epsilon}]) = j_\mathcal
U(p)\circ j_E[P(\eta)] = j_\mathcal U(p)\circ j_\mathcal U[P(\eta)] = j_\mathcal
U\circ p[P(\eta)] = j_\mathcal U[\theta_{\epsilon+2}]\). This shows that
\(j_\mathcal U[\theta_{\epsilon+2}]\in M_\mathcal U\), so by
\cref{UltrafilterUndefinabilityThm}, it follows that
\(\theta_{\epsilon+2}\preceq^{*}V_{\epsilon+1}\), which is a contradiction.

This shows that for any \(\eta < \theta_{\epsilon+1}^{+\kappa}\), \(P(\eta)\)
does not surject onto \(\theta_{\epsilon+2}\). The same argument applied to the
finite iterates of \(j\) shows that for any \(n <\omega\), \(\eta <
\theta_{\epsilon+1}^{+\kappa_n(j)}\), \(P(\eta)\) does not surject onto
\(\theta_{\epsilon+2}\). This proves the theorem.
\end{proof}
As a corollary of the proof of \cref{OddBigThm}, we have the following fact,
which exhibits a difference between the even and odd levels with regard to
\cref{NoGeneratorLma}:
\begin{prp}\label{YesGeneratorPrp}
Suppose \(\epsilon\) is an even ordinal and \(j : V_{\epsilon+2}\to
V_{\epsilon+2}\) is an elementary embedding. Then \(j\) has a generator in the
interval \((\theta_{\epsilon+1}, \theta_{\epsilon+2})\).\qed
\end{prp}
\subsection{The Coding Lemma}
One of the central theorems in the analysis of \(L(V_{\lambda+1})\) assuming the
axiom \(I_0\) is Woodin's generalization of the Moschovakis Coding Lemma. Here
we prove a new Coding Lemma. This Coding Lemma lifts Woodin's to structures of
the form \(L(V_{\epsilon+1})\) where \(\epsilon\) is even and the appropriate
generalization of \(I_0\) holds. But moreover, the proof adds a new twist to
Woodin's and as a consequence it applies to a host of models beyond
\(L(V_{\epsilon+1})\). For example, the Coding Lemma holds in
\(\text{HOD}(V_{\epsilon+1})\), and more interestingly, the Coding Lemma holds
in \(V\) itself under what seem to be reasonable assumptions.
\begin{defn}
Suppose \(\epsilon\) and \(\eta\) are ordinals, \(\varphi : V_{\epsilon+1}\to
\eta\) is a surjection, and \(R\) is a binary relation on \(V_{\epsilon+1}\).
\begin{itemize}
\item A relation \(\bar R\subseteq R\) is a {\it \(\varphi\)-total subrelation
of} \(R\) if \(\varphi[\text{dom}(\bar R)] = \varphi[\text{dom}(R)]\). 
\item A set of binary relations \(\Gamma\) on \(V_{\epsilon+1}\) is a {\it
code-class for \(\eta\)} if for any surjection \(\psi : V_{\epsilon+1}\to
\eta\), every binary relation on \(V_{\epsilon+1}\) has a \(\psi\)-total
subrelation in \(\Gamma\).
\item The {\it Coding Lemma holds at \(\epsilon\)} if every ordinal \(\eta <
\theta_{\epsilon+2}\) has a code-class \(\Gamma\) such that \(\Gamma \preceq^{*}
V_{\epsilon+1}\).
\end{itemize}
\end{defn}
The Coding Lemma has a number of important consequences. For example:
\begin{prp}\label{StrongLimitPrp}
Suppose \(\epsilon\) is an ordinal at which the Coding Lemma holds. Then
\(\theta_{\epsilon+2}\) is a strong limit cardinal. In fact, for any \(\eta <
\theta_{\epsilon+2}\), \(P(\eta)\preceq^{*} V_{\epsilon+1}\).
\begin{proof}
Fix a code-class \(\Gamma\) for \(\eta\) with \(\Gamma\preceq^{*}
V_{\epsilon+1}\). Fix a surjection \(\varphi : V_{\epsilon+1}\to \eta\). It is
immediate that \(P(\eta) = \{\varphi[\text{dom}(R)] \mid R\in \Gamma\}\).
Therefore since \(\Gamma\preceq^{*}V_{\epsilon+1}\),
\(P(\eta)\preceq^{*}V_{\epsilon+1}\).
\end{proof}
\end{prp}

\begin{defn}
Then the {\it Collection Principle} states that every class binary relation
\(R\) whose domain is a set has a set-sized subrelation \(\bar R\) such that
\(\text{dom}(\bar R) = \text{dom}(R)\).
\end{defn}
It seems that one needs a local form of the Collection Principle to prove the
Coding Lemma:
\begin{thm}\label{CodingLemma}
Suppose \(\epsilon\) is an even ordinal and \(M\) is an inner model containing
\(V_{\epsilon+1}\). Suppose there is an embedding \(j\in \mathscr
E(V_{\epsilon+2}\cap M)\) with \(\textnormal{crit}(j) = \kappa\). Assume
\((\mathcal H_{\epsilon+2})^M\) satisfies \(\kappa\textnormal{-DC}\) and the
Collection Principle. Then \(M\) satisfies the Coding Lemma at \(\epsilon\).
\end{thm}
Note that we only require the {\it first-order} Collection Principle to hold in
\((\mathcal H_{\epsilon+2})^M\). 

We begin by proving a Weak Coding Lemma, which requires some more definitions.
\begin{defn}
Suppose \(\epsilon\) and \(\eta\) are ordinals, \(\varphi : V_{\epsilon+1}\to
\eta\) is a surjection, and \(R\) is a binary relation on \(V_{\epsilon+1}\). 
\begin{itemize}
\item A relation \(\bar R\subseteq R\) is a {\it \(\varphi\)-cofinal subrelation
of} \(R\) if either \(\varphi[\text{dom}(R)]\) is not cofinal in \(\eta\) or
\(\varphi[\text{dom}(\bar R)]\) is cofinal in \(\varphi[\text{dom}(R)]\).
\item A set of binary relations \(\Gamma\) on \(V_{\epsilon+1}\) is a {\it weak
code-class for \(\eta\)} if for any surjection \(\psi : V_{\epsilon+1}\to
\eta\), every binary relation on \(V_{\epsilon+1}\) has a \(\psi\)-cofinal
subrelation in \(\Gamma\). 
\item The {\it Weak Coding Lemma holds at \(\epsilon\)} if every ordinal \(\eta
< \theta_{\epsilon+2}\) has a weak code-class \(\Gamma\) such that
\(\Gamma\preceq^{*} V_{\epsilon+1}\).
\end{itemize}
\end{defn}
\begin{lma}\label{WeakCodingLemma}
Suppose \(\epsilon\) is an even ordinal and \(M\) is an inner model containing
\(V_{\epsilon+1}\). Suppose there is an elementary embedding from
\(V_{\epsilon+2}\cap M\) to \(V_{\epsilon+2}\cap M\) with critical point
\(\kappa\). Assume \((\mathcal H_{\epsilon+2})^M\) satisfies
\(\kappa\textnormal{-DC}\). Then the Weak Coding Lemma holds at \(\epsilon\) in
\(M\).
\begin{proof}
Assume towards a contradiction that \(M\) does not satisfy the Weak Coding Lemma
at \(\epsilon\). Let \(\eta\) be the least ordinal for which there is no weak
code-class \(\Gamma\) with \(\Gamma \preceq^{*} V_{\epsilon+1}\). Notice that
\(\eta\) is definable in \((\mathcal H_{\epsilon+2})^M\), and hence is fixed by
any embedding in \(\mathscr E((\mathcal H_{\epsilon+2})^M)\).

We now use that \((\mathcal H_{\epsilon+2})^M\) satisfies
\(\kappa\textnormal{-DC}\) to construct a sequence \[\langle
(A_{\alpha},\varphi_\alpha) \mid \alpha < \kappa\rangle\] by recursion. Each
\(A_\alpha\) will be a binary relation on \(V_{\epsilon+1}\), and each
\(\varphi_\alpha\) will be a surjection from \(V_{\epsilon+1}\) to \(\eta\).
Suppose \(\langle (A_{\alpha},\varphi_\alpha) \mid \alpha < \beta\rangle\) has
been defined. Let \(\Gamma\) be the collection of binary relations on
\(V_{\epsilon+1}\) definable (from parameters) over
\((V_{\epsilon+1},A_{\alpha})\) for some \(\alpha < \beta\). Obviously,
\(\Gamma\preceq^{*} V_{\epsilon+1}\), so by choice of \(\eta\), \(\Gamma\) is
not a weak code-class for \(\eta\). We can therefore choose a binary relation
\(A_\beta\) on \(V_{\epsilon+1}\) and a surjection \(\varphi_\beta\) from
\(V_{\epsilon+1}\) to \(\eta\) such that \(A_\beta\) has no
\(\varphi_\beta\)-cofinal subrelation in \(\Gamma\). This completes the
construction.

Fix an embedding \(j\in \mathscr E(V_{\epsilon+2}\cap M)\) with critical point
\(\kappa\). Then \(j\) extends uniquely to \(j^\star : (\mathcal
H_{\epsilon+2})^M\to (\mathcal H_{\epsilon+2})^M\). Apply \(j^\star\) twice to
\(\langle (A_{\alpha},\varphi_\alpha) \mid \alpha < \kappa\rangle\):
\begin{align*}\langle (A_\alpha^1,\varphi^1_\alpha): \alpha < \kappa_1\rangle &= j^\star(\langle (A_{\alpha},\varphi_\alpha) \mid \alpha < \kappa\rangle)\\
\langle (A_\alpha^2,\varphi^2_\alpha) \mid \alpha < \kappa_2\rangle &=j^\star(\langle (A_\alpha^1,\varphi^1_\alpha) \mid \alpha < \kappa_1\rangle)
\end{align*}
Notice the following equality:
\begin{equation}
j^\star(A_{\kappa}^1,\varphi_\kappa^1) = (A_{\kappa_1}^2,\varphi_{\kappa_1}^2)\label{Old}
\end{equation}
\cref{Old} implies: 
\[\text{\(j[A_{\kappa}^1]\) is a \(\varphi_{\kappa_1}^2\)-cofinal subrelation of
\(A_{\kappa_1}^2\).}\] The point here is that \(j^\star(\eta) = \eta\), so
\(j^\star[\eta]\) is cofinal in \(\eta\). Hence 
\[\varphi_{\kappa_1}^2[\text{dom}(j[A_{\kappa}^1])] =
j^\star[\varphi^1_\kappa[\text{dom}(A^1_\kappa)]]\] is cofinal in
\(\text{dom}(A^2_{\kappa_1})\). (Note that for every \(\beta < \kappa\), the set
\(\text{dom}(A_\beta)\) is cofinal in \(\eta\): otherwise for all \(\alpha <
\beta\), \(A_\alpha\) is vacuously a \(\varphi_\beta\)-cofinal subrelation of
\(A_\beta\).)

Recall, however, that our construction ensured that for all \(\beta < \kappa\),
\(A_\beta\) has no \(\varphi_\beta\)-cofinal subrelation that is definable over
\((V_{\epsilon+1},A_{\alpha})\) for some \(\alpha < \beta\). By elementarity,
\(A^2_{\kappa_1}\) has no \(\varphi_{\kappa_1}^2\)-cofinal subrelation that is
boldface definable over \((V_{\epsilon+1},A_{\alpha})\) for some \(\alpha <
\kappa_1\). We reach a contradiction by showing that \(j^\star[A_{\kappa}^1]\)
is definable over \((V_{\epsilon+1},A_\kappa^1)\).

Combinatorially, the following equation is the new ingredient in this proof:
\begin{equation}
j(j)(A_{\kappa}^1) = A_{\kappa}^2\label{New}
\end{equation}
(More formally \((j^\star(j\restriction V_{\epsilon}))^{+}(A_{\kappa}^1) =
A_{\kappa}^2\); this notation is just too unwieldy.) \cref{New} implies that
\(A_{\kappa}^1 = j(j)^{-1}[A_{\kappa}^2]\). By \cref{DefinabilityThm},
\(j(j)\restriction V_{\epsilon+1}\) is definable over \(V_{\epsilon+1}\) from
its restriction to \(V_{\epsilon}\), and therefore \(A_{\kappa}^1\) is definable
from parameters over \((V_{\epsilon+1},A_{\kappa}^2)\). Similarly,
\(j[A_\kappa^1]\) is definable over \((V_{\epsilon+1},A_{\kappa}^1)\). It
follows that \(j[A_\kappa^1]\) is definable over
\((V_{\epsilon+1},A_{\kappa}^2)\).
\end{proof}
\end{lma}
The proof that the Weak Coding Lemma implies the Coding Lemma is a direct
generalization of Woodin's:
\begin{proof}[Proof of \cref{CodingLemma}] Assume towards a contradiction that
the Coding Lemma fails in \(M\) at \(\epsilon\). Let \(\eta\) be the least
ordinal for which there is no code-class \(\Gamma\) such that \(\Gamma
\preceq^{*} V_{\epsilon+1}\). 

Let \(\psi : V_{\epsilon+1}\to \eta\) be an arbitrary surjection. We begin by
using the Collection Principle to construct a set \(\Lambda\preceq^{*}
V_{\epsilon+1}\) that is a code-class for every \(\alpha < \eta\). Consider the
relation \(S\subseteq V_{\epsilon+1}\times \mathcal H_{\epsilon+2}\) defined by
\[S(a,\Gamma) \iff \Gamma\text{ is a code-class for }\psi(a)\text{ with }\Gamma
\preceq^{*}V_{\epsilon+1}\] By the minimality of \(\eta\), \(S\) is a total
relation. Since \(\mathcal H_{\epsilon+2}\) satisfies the Collection Principle,
there is a total relation \(\bar S\subseteq S\) with \(\bar S\in \mathcal
H_{\epsilon+2}\). Let \[\Lambda = \bigcup_{\Gamma\in \text{ran}(\bar S)}
\Gamma\] Clearly \(\Lambda\) is a code-class for every \(\alpha < \eta\). Since
\(\bar S\in \mathcal H_{\epsilon+2}\), \(\Lambda\preceq^{*} V_{\epsilon+1}\), so
there is a surjection \(\pi : V_{\epsilon+1}\to \Lambda\).

Applying the Weak Coding Lemma, fix a weak code-class \(\Sigma\) for \(\eta\)
with \(\Sigma\preceq^{*} V_{\epsilon+1}\). Let \(\Gamma\) be the collection of
binary relations on \(V_{\epsilon+1}\) definable over \((\mathcal
H_{\epsilon+2},\pi)\) using parameters in \(\Sigma\). Clearly
\(\Gamma\preceq^{*}V_{\epsilon+1}\). We finish by showing that \(\Gamma\) is a
code-class for \(\eta\).

Fix a surjection \(\varphi : V_{\epsilon+1}\to \eta\) and a binary relation
\(R\) on \(V_{\epsilon+1}\). We must find a \(\varphi\)-total subrelation \(\bar
R\) of \(R\) that belongs to \(\Gamma\). First consider the relation 
\[S(a,u) \iff \pi(u)\text{ is a \(\varphi\)-total subrelation of
}R\restriction_\varphi \psi(a)\] (Here \(R\restriction_\varphi \beta =
R\restriction \{b \mid \varphi(b) < \beta\}\).) Notice that \(S\) is a total
relation due to the construction of \(\Lambda\). Let \(\bar S\in\Sigma\) be a
\(\psi\)-cofinal subrelation of \(S\). Since \(\bar S\subseteq S\), for all
\(u\in \text{ran}(\bar S)\), \(\pi(u)\subseteq R\). Moreover since \(\bar S\) is
a \(\psi\)-cofinal subrelation of \(S\), for cofinally many \(\beta < \eta\),
there is some \(u\in \text{ran}(\bar S)\) such that \(\pi(u)\) is a
\(\varphi\)-total subrelation of \(R\restriction \beta\). Let \[\bar R =
\bigcup_{u\in \text{ran}(\bar S)} \pi(u)\] Then \(\bar R\) is a
\(\varphi\)-total subrelation of \(R\). Moreover \(\bar R\in \Gamma\) since
\(\bar R\) is definable over \((\mathcal H_{\epsilon+2},\pi)\) using the
parameter \(\bar S\in \Sigma\).
\end{proof}

\begin{thm}\label{StrongLimitThm}
Suppose \(\epsilon\) is an even ordinal. Suppose there is an elementary
embedding from \(V_{\epsilon+2}\) to \(V_{\epsilon+2}\) with critical point
\(\kappa\). Assume \(\mathcal H_{\epsilon+2}\) satisfies the Collection
Principle and \(\kappa\)-\textnormal{DC}. Then \(\theta_{\epsilon+2}\) is a
strong limit cardinal.\qed
\end{thm}

As a corollary of \cref{StrongLimitThm}, we have a proof of the Kunen
Inconsistency Theorem that seems new:
\begin{cor}[ZFC]\label{KIT}
There is no elementary embedding from \(V_{\lambda+2}\) to \(V_{\lambda+2}\).
\begin{proof}
Assume towards a contradiction that there is an elementary embedding from
\(V_{\lambda+2}\) to \(V_{\lambda+2}\). The Axiom of Choice implies that all the
hypotheses of \cref{CodingLemma} are satisfied when \(M = V\). Therefore by
\cref{StrongLimitThm}, \(\theta_{\lambda+2}\) is a strong limit cardinal. On the
other hand, the Axiom of Choice implies \(\theta_{\lambda+2} =
|V_{\lambda+1}|^+\), which is not a strong limit cardinal.
\end{proof}
\end{cor}
Another consequence of the Coding Lemma beyond \cref{StrongLimitThm} is the
following theorem:
\begin{thm}\label{RegularThm}
Suppose \(\epsilon\) is an even ordinal. Suppose there is an elementary
embedding from \(V_{\epsilon+2}\) to \(V_{\epsilon+2}\) with critical point
\(\kappa\). Assume \(\mathcal H_{\epsilon+2}\) satisfies the Collection
Principle and \(\kappa\)-\textnormal{DC}. Then \(\theta_{\epsilon+2}\) is a
limit of regular cardinals.
\end{thm}
The theorem also generalizes easily to inner models \(M\) as in
\cref{CodingLemma}, but here it seems one must require that \(j\) is a proper
embedding in the sense of \cite{SEM2}, and since we do not want to introduce the
notion of a proper embedding, we omit the proof.

The proof uses the following lemma, which is a direct generalization of a
construction due to Woodin:
\begin{lma}[\cite{SEM2}, Lemma 6]\label{ProperFix}
Suppose \(\epsilon\) is an even ordinal and \(j : V_{\epsilon+2}\to
V_{\epsilon+2}\) is an elementary embedding. Then for any set \(A\subseteq
V_{\epsilon+2}\), there is some set \(B\in V_{\epsilon+2}\) such that \(j(B) =
B\) and \(A\in L(V_{\epsilon+1},B)\).\qed
\end{lma}
We warn that (2.2) in the proof of Lemma 6 of \cite{SEM2} contains a typo. We
also need a routine generalization of another theorem of Woodin:
\begin{thm}[\cite{SEM2}, Lemma 22]\label{WoodinRegularThm}
Suppose \(\epsilon\) is an even ordinal, \(B\) is a subset of
\(V_{\epsilon+1}\), and \(j : L(V_{\epsilon+1},B)\to L(V_{\epsilon+1},B)\) is an
elementary embedding such that \(j(B) = B\). Let \(\theta\) be
\(\theta_{\epsilon+2}\) as computed in \(L(V_{\epsilon+1},B)\). Then \(\theta\)
is a limit of regular cardinals in \(L(V_{\epsilon+1},B)\).\qed
\end{thm}
We note that although Woodin's proof seems to use \(\lambda\)-DC, this is not
really necessary by the proof of \cref{InnerModelCodingThm} below.
\begin{proof}[Proof of \cref{RegularThm}] Fix a cardinal \(\eta <
\theta_{\epsilon+2}\). We will show that there is a regular cardinal in the
interval \((\eta,\theta_{\epsilon+2})\). By the Coding Lemma, there is a
code-class \(\Gamma\) for \(\eta\) such that
\(\Gamma\preceq^{*}V_{\epsilon+1}\). Let \(\varphi : V_{\epsilon+1}\to \eta\) be
a surjection in \(M\) and let \(A\subseteq V_{\epsilon+1}\) be a set such that
\(\Gamma\subseteq L(V_{\epsilon+1},A)\) and such that \(L(V_{\epsilon+1},A)\)
satisfies that \(\eta < \theta_{\epsilon+2}\). Let \(B\) be a set such that
\(j(B) = B\) and \(A\in L(V_{\epsilon+1},B)\).

Let \(\theta\) be  \(\theta_{\epsilon+2}\) as computed in
\(L(V_{\epsilon+1},B)\). We claim that for every ordinal \(\gamma < \theta\),
\(\gamma^\eta\subseteq L(V_{\epsilon+1},B)\). To see this, fix \(s: \eta\to
\gamma\), and we will show that \(s\in L(V_{\epsilon+1},B)\). Let \(\psi :
V_{\epsilon+1}\to \gamma\) be a surjection. Let \(R = \{(x,y) \mid s(\varphi(x))
=\psi(y)\}\). Since \(\Gamma\) is a code-class for \(\eta\), there is a
subrelation \(\bar R\) of \(R\) in \(\Gamma\) such that \(\text{dom}(\bar R) =
\text{dom}(R)\). But \(\bar R\in L(V_{\epsilon+1},B)\) and \(s\) is clearly
coded by \(\bar R\). Therefore \(s\in L(V_{\epsilon+1},B)\), as desired.

By \cref{WoodinRegularThm}, \(\theta\) is a limit of regular cardinals in
\(L(V_{\epsilon+1},B)\). Therefore let \(\iota \in (\eta,\theta)\) be a regular
cardinal of \(L(V_{\epsilon+1},B)\). Since \(\iota^\eta\subseteq
L(V_{\epsilon+1},B)\), \(\text{cf}(\iota) \in (\eta,\theta)\). In particular,
there is a regular cardinal in the interval \((\eta,\theta_{\epsilon+2})\).
\end{proof}

\begin{qst}
Suppose \(\epsilon\) is an even ordinal. Suppose there is an elementary
embedding from \(V_{\epsilon+2}\) to \(V_{\epsilon+2}\) with critical point
\(\kappa\). Assume \(\mathcal H_{\epsilon+2}\) satisfies the Collection
Principle and \(\kappa\)-\textnormal{DC}. Must \(\theta_{\epsilon+2}\) be a
limit of measurable cardinals?
\end{qst}

Let us now show that for certain inner models, one can avoid the extra
assumptions in \cref{CodingLemma}.
\begin{thm}\label{InnerModelCodingThm}
Suppose \(N\) is an inner model of \textnormal{ZFC}, \(\epsilon\) is an even
ordinal, \(A\subseteq V_{\epsilon+1}\), and \(W\) is a set. Let \(M =
N(V_{\epsilon+1},A)[W]\). Suppose there is an elementary embedding from \(M\cap
V_{\epsilon+2}\) to \(M\cap V_{\epsilon+2}\). Then the Coding Lemma holds in
\(M\) at \(\epsilon\).
\begin{proof}[Sketch] We just describe how to modify the proofs above to avoid
assuming Collection and Dependent Choice.

The first point is that one can prove \((\mathcal H_{\epsilon+2})^M\) satisfies
the Collection Principle. To see this, note that by the definition of \(M\), for
any \(X\in M\), there is an ordinal \(\gamma\) such that \(X\preceq^{*}
V_{\epsilon+1}\times \gamma\) in \(M\). Therefore fix a surjection \[f :
V_{\epsilon+1}\times \gamma \to (\mathcal H_{\epsilon+2})^M\] with \(f\in M\).
For each \(x\in V_{\epsilon+1}\), let \(\Gamma_x = f[\{x\}\times \gamma]\) and
let \(<_x\) be the wellorder of \(H_x\) induced by \(f\). Suppose \(R\subseteq
(\mathcal H_{\epsilon+2})^M\) is a definable relation whose domain belongs to
\((\mathcal H_{\epsilon+2})^M\). Then \(R\in M\). For each \(x\in
V_{\epsilon+1}\), let \(r_x(a)\) be the \(<_x\)-least \(b\) in \(H_x\) such that
\((a,b)\in R\). The sequence \(\langle r_x \mid x\in V_{\epsilon+1}\rangle\)
belongs to \(M\), and indeed it belongs to \((\mathcal H_{\epsilon+2})^M\). This
is because one can define a partial surjection from \(\text{dom}(R)\) to \(r_x\)
uniformly in \(x\). Therefore \(\bigcup_{x\in V_{\epsilon+1}} r_x\in (\mathcal
H_{\epsilon+2})^M\), and letting \(\bar R = \bigcup_{x\in V_{\epsilon+1}} r_x\),
we have \(\text{dom}(\bar R) = \text{dom}(R)\). This verifies that \((\mathcal
H_{\epsilon+2})^M\) satisfies the Collection Principle.

To finish the sketch, we describe how in this special case one can avoid the use
of Dependent Choice in the proof of the Weak Coding Lemma
(\cref{WeakCodingLemma}). One need only modify the construction of the sequence
\(\langle (A_\alpha,\varphi_\alpha) \mid \alpha < \kappa\rangle\). One instead
constructs a sequence \(\langle (A_{x,\alpha},\varphi_{x,\alpha}) \mid
(x,\alpha)\in V_{\epsilon+1}\times \kappa\rangle\). The construction proceeds as
follows. Suppose that \(\langle (A_{x,\alpha},\varphi_{x,\alpha}) \mid
(x,\alpha)\in V_{\epsilon+1}\times \beta\rangle\) has been defined. By the
failure of the Weak Coding Lemma, there is some \((A,\varphi)\) such that \(A\)
has no \(\varphi\)-cofinal subrelation that is definable over
\((V_{\epsilon+1},A_{x,\alpha})\) for some \((x,\alpha)\in V_{\epsilon+1}\times
\beta\). For each \(x\in V_{\epsilon+1}\), let
\((A_{x,\beta},\varphi_{x,\beta})\) be the \(<_x\)-least such \((A,\varphi)\in
H_x\), if one exists. This completes the construction.

Now one considers:
\begin{align*}\langle (A^1_{x,\alpha},\varphi^1_{x,\alpha}) \mid (x,\alpha)\in V_{\epsilon+1}\times \kappa_1\rangle &= j(\langle (A_{x,\alpha},\varphi_{x,\alpha}) \mid (x,\alpha)\in V_{\epsilon+1}\times \kappa\rangle)
\end{align*}
Fix any \(x\in V_{\epsilon+1}\) such that
\((A^1_{x,\kappa},\varphi^1_{x,\kappa})\) is defined. One then uses
\((A^1_{x,\kappa},\varphi^1_{x,\kappa})\) in place of
\((A^1_\kappa,\varphi^1_\kappa)\). The rest of the proof is unchanged.
\end{proof}
\end{thm}

We conclude this section by showing that the Coding Lemma fails at odd ordinals
in a strong sense. For example, we show that if \(\epsilon\) is even and there
is an elementary embedding from \(V_{\epsilon+3}\) to \(V_{\epsilon+3}\), then
\(V_{\epsilon+2}\) does not surject onto \(P(\theta_{\epsilon+2})\). By
\cref{StrongLimitPrp}, this implies that the Coding Lemma does not hold at
\(\epsilon+1\).
\begin{thm}\label{NoCodingThm}
Suppose \(\epsilon\) is an even ordinal and there is an elementary embedding
from \(V_{\epsilon+2}\) to \(V_{\epsilon+2}\). Then for any ordinal \(\gamma\),
there is no surjection from \(V_{\epsilon}\times \gamma\) onto
\(P(\theta_{\epsilon})\).
\end{thm}
\begin{proof}
Suppose towards a contradiction that the theorem fails. Let \(\theta =
\theta_{\epsilon}\). Let \(\gamma\) be the least ordinal such that
\(P(\theta)\preceq^{*} V_{\epsilon}\times \gamma\). We first show
\begin{equation}\label{InequalityEqn}\gamma \preceq^{*} V_{\epsilon}\times
P(\theta)\preceq^{*} V_{\epsilon+1}\end{equation} For the first inequality, fix
a surjection \(f : V_{\epsilon}\times \gamma\to P(\theta)\). Define \[g :
V_{\epsilon}\times P(\theta)\to \gamma\] by setting \(g(A,S) = \min \{\xi \mid
f(A,\xi) = S\}\). Let \(T\) be the range of \(g\). Then \(f[V_{\epsilon}\times
T] = P(\theta)\), so by the minimality of \(\gamma\), it must be that \(|T| =
\gamma\). Thus \(\gamma \preceq^{*} V_{\epsilon}\times P(\theta)\).

For the second inequality, note that \(\theta\preceq^{*} V_{\epsilon}\) so
\(P(\theta)\preceq^{*} P(V_{\epsilon}) = V_{\epsilon+1}\). It follows easily
that \(V_{\epsilon}\times P(\theta)\preceq^{*} V_{\epsilon+1}\).

By our large cardinal hypothesis, there is an elementary embedding \(j :
\mathcal H_{\epsilon+2}\to \mathcal H_{\epsilon+2}\). Note that \(f\in \mathcal
H_{\epsilon+2}\) by \cref{InequalityEqn}. By the elementarity of \(j\), \(j(f)\)
is a surjection from \(V_{\epsilon}\times j(\gamma)\) to \(P(\theta)\);
therefore, for some \((a,\alpha)\in V_{\epsilon}\times j(\gamma)\),
\(j(f)(a,\alpha) = j[\theta]\). By the cofinal embedding property
(\cref{RepresentationLma}), or trivially if \(\epsilon\) is a limit ordinal,
there is an ordinal \(\xi\) such that \(\xi+2\leq \epsilon\), a set \(x\subseteq
V_{\xi}\), and a function \(g : V_{\xi+1}\to V_{\epsilon}\) such that \(j(g)(x)
= a\). Let \(U\) be the ultrafilter over \(V_{\xi+1}\times \gamma\) derived from
\(j\) using \((x,\alpha)\). Let \(h : V_{\xi+1}\times \gamma\to P(\theta)\) be
defined by \(h(u,\beta)= f(g(u),\beta)\). Then it is easy to see that \([h]_U =
j_U[\theta]\). This contradicts \cref{FixedPointUndefinabilityThm}. 
\end{proof}
It is a bit strange that for example we require an embedding from
\(V_{\epsilon+2}\) to \(V_{\epsilon+2}\) to show this structural property of
\(P(\theta_{\epsilon})\). The theorem implies that \(P(\theta_{\epsilon})\)
cannot be wellordered, so for example in the case ``\(\epsilon =
\kappa_\omega(j)\),'' one cannot reduce the large cardinal hypothesis to an
embedding \(j:V_{\lambda+1}\to V_{\lambda+1}\) assuming the consistency of ZFC
plus \(I_1\). (Similar results hold for \(j : V_{\epsilon+2}\to
V_{\epsilon+2}\), considering the model \(L(V_{\epsilon+1})[j]\).) Inspecting
the proof, however, one obtains the following result:
\begin{thm}\label{NoCodingThm2}
Suppose \(\epsilon\) is an even ordinal and there is an elementary embedding
from \(V_{\epsilon+1}\) to \(V_{\epsilon+1}\). Then there is no surjection from
\(V_{\epsilon}\) onto \(P(\theta_{\epsilon})\).\qed
\end{thm}

The following question is related to \cref{NoCodingThm} and might be more
tractable than the question of whether \(\theta_{\epsilon+1} =
(\theta_{\epsilon})^+\):
\begin{qst}
Suppose \(\epsilon\) is an even ordinal and there is an elementary embedding
from \(V_{\epsilon+2}\) to \(V_{\epsilon+2}\). Is there a surjection from
\(P(\theta_{\epsilon})\) to \(\theta_{\epsilon+1}\)?
\end{qst}
\section{Ultrafilters and the Ketonen order}\label{KetonenSection}
\subsection{Measurable cardinals and the Ulam argument}
With the goal of refuting strong choiceless large cardinal axioms in mind,
Woodin \cite{SEM} showed that various consequences of the Axiom of Choice follow
from the existence of large cardinals at the level of supercompact and
extendible cardinals. While developing set theoretic geology in the choiceless
context, Usuba realized that the apparently much weaker notion of a
L\"owenheim-Skolem cardinal does just as well as a supercompact.
\begin{defn}
A cardinal \(\kappa\) is a {\it L\"owenheim-Skolem cardinal} if for all ordinals
\(\alpha < \kappa\leq \gamma\), for any \(a\in V_\gamma\), there is an
elementary substructure \(X\prec V_{\gamma+1}\) such that \([X]^{V_\alpha}\cap
V_\gamma\subseteq X\), \(a\in X\), and for some \(\beta < \kappa\),
\(X\preceq^{*} V_\beta\).
\end{defn}

Here we give a proof of Ulam's theorem on the atomicity of saturated filters in
ZF assuming the existence of two strategically placed L\"owenheim-Skolem
cardinals. Our arguments are inspired by the ones in \cite{Cutolo}, and our
result generalizes some of the theorems of that paper while simultaneously
reducing their large cardinal hypotheses. 

Recall that the usual proof of Ulam's theorem uses a splitting argument that
seems to make heavy use of a strong form of the Axiom of Dependent Choice. Here
it is shown that this can be avoided if one is allowed to take two elementary
substructures.
\begin{thm}\label{LSUlam}
Suppose \(\gamma\) is a cardinal, \(\kappa_0 < \kappa_1\) are L\"owenheim-Skolem
cardinals above \(\gamma\), and \(\delta\) is an ordinal. Suppose \(F\) is a
filter over \(\delta\) that is \(V_{\kappa_1}\)-complete and weakly
\(\gamma\)-saturated. Then for some cardinal \(\eta < \gamma\), there is a
partition \(\langle S_\alpha \mid \alpha < \eta\rangle\) of \(\delta\) such that
\(F\restriction S_\alpha\) is an ultrafilter for all \(\alpha < \eta\).
\begin{proof}
Since \(\kappa_1\) is a L\"owenheim-Skolem cardinal, we can fix an elementary
substructure \(X\prec V_{\delta+\omega+1}\) with the following properties:
\begin{itemize}
\item \(X\preceq^{*} V_{\beta}\) for some \(\beta < \kappa_1\).
\item \(\gamma,\kappa_0,\kappa_1,\delta,\) and \(F\) belong to \(X\).
\item \([X]^{V_\alpha}\cap V_{\delta+\omega}\subseteq X\) for every \(\alpha <
\kappa_0\).
\end{itemize}
Let \(\pi : H_X\to V_{\delta+\omega+1}\) be the inverse of the Mostowski
collapse of \(X\), and let \(\bar \gamma,\bar \kappa_0,\bar \kappa_1,\bar
\delta,\) and \(\bar F\) be the preimages under \(\pi\) of
\(\gamma,\kappa_0,\kappa_1,\delta,\) and \(F\) respectively. 

For each ordinal \(\xi < \delta\), let \(U_\xi\) denote the \(H_X\)-ultrafilter 
over \(\bar \delta\)
derived from \(\pi\) using \(\xi\). Since \(\pi\) has critical point above
\(\kappa_0\) and \(H_X\) is closed under \(V_\alpha\)-sequences for every
\(\alpha < \kappa_0\), for all \(\xi < \delta\), \(U_\xi\) is
\(V_{\kappa_0}\)-complete. (More precisely, \(U_\xi\) generates a
\(V_{\kappa_0}\)-complete filter.) Since \(\{U_\xi \mid \xi < \delta\}\subseteq
P(H_X)\) and \(H_X\preceq^{*} V_{\beta}\), \(\{U_\xi \mid \xi <
\delta\}\preceq^{*} V_{\beta+1}\).

For each \(\xi < \delta\), let \(B_\xi = \{\xi' \mid U_{\xi'} = U_\xi\}\). We
make the obvious observation that the map sending \(B_\xi\) to \(U_\xi\) is a
(well-defined) one-to-one correspondence. It follows that \(\{B_\xi \mid \xi <
\delta\}\preceq^{*} V_{\beta+1}\). Let \(T\subseteq \delta\) be the set of \(\xi
< \delta\) such that \(B_\xi\) is \(F\)-positive. Since \(F\) is
\(V_{\kappa_1}\)-complete and \(\kappa_1 > \beta+1\), \(T\in F\). (This is a
standard argument: \(\{B_\xi \mid \xi \in \delta\setminus T\}\) is a collection
of \(F\)-null sets with \(\{B_\xi \mid \xi \in \delta\setminus T\}\preceq^{*}
V_{\beta+1}\), and hence \(\bigcap_{\xi \in \delta\setminus T} B_\xi\) is
\(F\)-null by the \(V_{\kappa_1}\)-completeness of \(F\). Therefore the
complement of \(\bigcap_{\xi \in \delta\setminus T} B_\xi\) belongs to \(F\).
Note that \(\xi\in T\) if and only if \(B_\xi\subseteq T\), and hence  \(T =
\bigcup_{\xi\in T} B_\xi\). It follows that \(T\) is the complement of
\(\bigcap_{\xi \in \delta\setminus T} B_\xi\), so \(T\in F\) as desired.) Since
\(F\) is weakly \(\gamma\)-saturated and \(\{B_\xi \mid \xi \in T\}\) is a
partition of \(\delta\) into positive sets, \(|\{B_\xi \mid \xi \in T\}| <
\gamma\). Given the one-to-one correspondence described above, it follows that
\(|\{U_\xi \mid \xi \in T\}| < \gamma\).

Now since \(\kappa_0\) is a L\"owenheim-Skolem cardinal and \(\gamma <
\kappa_0\), we can fix an elementary substructure \(Y\prec V_{\delta+\omega+2}\)
with \(\{U_\xi \mid \xi \in T\}\in Y,\) \(X\in Y\), \(\gamma\subseteq Y\), and
\(Y\preceq^{*} V_{\beta'}\) for some \(\beta' < \kappa_0\). Since \(|\{U_\xi
\mid \xi \in T\}| < \gamma\), it follows that \(U_\xi\in Y\) for all \(\xi\in
T\). Notice, however, that \(U_\xi\cap Y\in X\) since
\([X]^{V_{\beta'}}\cap V_{\delta+\omega}\subseteq X\). For \(\xi \in T\), let \[A_\xi = \bigcap \{A
\in U_\xi \mid A\in Y\}\] Notice that \(A_\xi\in U_\xi\) for all \(\xi\in T\)
since \(U_\xi\) is \(V_{\kappa_0}\)-complete.

We claim that \(A_{\xi_0}\cap A_{\xi_1} = \emptyset\) whenever \(U_{\xi_0}\neq
U_{\xi_1}\). (Obviously if \(U_{\xi_0} = U_{\xi_1}\), then \(A_{\xi_0} =
A_{\xi_1}\).) To see this, note that since \(U_{\xi_0}\neq U_{\xi_1}\) and
\(Y\prec V_{\delta+\omega + 2}\), there is some \(A\in H_X\cap Y\) with \(A\in
U_{\xi_0}\) and \(\bar \delta\setminus A\in U_{\xi_1}\). It follows that
\(A_{\xi_0}\subseteq A\) and \(A_{\xi_1}\subseteq \bar \delta\setminus A\), and
hence \(A_{\xi_0}\cap A_{\xi_1} = \emptyset\), as desired.

Let \(S = \{\xi\in T \mid \bar F\subseteq U_\xi\}\). Thus \(S = T\cap
\bigcap\{A\in F \mid A\in X\}\), so since \(T\) and \(\bigcap\{A\in F \mid A\in
X\}\) belong to \(F\), \(S\in F\). We claim that for all \(\xi\in S\), \(A_\xi\)
is an atom of \(\bar F\) in \(H_X\). Fix \(\xi \in S\), and suppose towards a
contradiction that \(E_0\) and \(E_1\) are disjoint \(\bar F\)-positive subsets
of \(A_\xi\) that belong to \(H_X\). Since \(\pi(E_0)\) is \(F\)-positive,
\(\pi(E_0)\cap S\neq \emptyset\), so fix \(\xi_0\in \pi(E_0)\cap S\). Note that
\(E_0\in U_{\xi_0}\) since \(U_{\xi_0}\) is the ultrafilter derived from \(\pi\)
using \(\xi_0\). Similarly fix \(\xi_1\in \pi(E_1)\cap S\), and note that
\(E_1\in U_{\xi_1}\). Since \(E_0\) and \(E_1\) are disjoint, it follows that
\(U_{\xi_0}\neq U_{\xi_1}\). In particular, one of them is not equal to
\(U_\xi\). Assume without loss of generality that \(U_{\xi_0}\neq U_\xi\). Since
\(E_0\in U_{\xi_0}\), we have \(E_0\cap A_{\xi_0}\neq \emptyset\). It follows
that \(A_\xi\cap A_{\xi_0}\neq \emptyset\), and this contradicts that
\(A_{\xi_0}\cap A_{\xi_1} = \emptyset\) whenever \(U_{\xi_0} \neq U_{\xi_1}\).

Therefore \(\{A_\xi \mid \xi \in S\}\) is a set of atoms for \(\bar F\). We
claim \(\bigcup_{\xi\in S} A_\xi\in F\). Suppose not, towards a contradiction.
In other words, the set \[E = \bar \delta\setminus \bigcup_{\xi\in S} A_\xi\] is
\(F\)-positive. For every \(\xi\in S\), since \(A_\xi\in U_\xi\) and \(\bar
F\subseteq U_\xi\) for all \(\xi\in S\), necessarily \(\bar F\restriction A_\xi
= U_\xi\). Note that \(E\) belongs to \(U_\xi\) for some \(\xi\in S\): indeed,
since \(\pi(E)\) is \(F\)-positive and \(S\in F\), \(\pi(E)\cap S\neq
\emptyset\); now for any \(\xi\in \pi(E)\cap S\), \(E\in U_\xi\). But then
\(E\cap A_\xi\neq \emptyset\), which contradicts that \(E = \bar \delta\setminus
\bigcup_{\xi\in S} A_\xi\).

Finally, note that \(\{A_\xi \mid \xi \in T\}\in H_X\) since
\([H_X]^\gamma\cap V_{\bar \delta+\omega}\subseteq H_X\). Hence \(H_X\) satisfies that there is a partition
of an \(\bar F\)-large set into fewer than \(\gamma\)-many atoms. By the
elementarity of \(\pi\), \(V_{\delta+\omega+1}\) satisfies that there is a
partition of an \(F\)-large set into fewer than \(\gamma\)-many atoms. Obviously
this is absolute to \(V\), which completes the proof. 
\end{proof}
\end{thm}

\begin{cor}
Suppose \(\eta\) is a limit of L\"owenheim-Skolem cardinals and there is an
elementary embedding from \(V_{\eta+2}\) to \(V_{\eta+2}\). If \(\eta\) is
regular, then \(\eta\) is measurable, and if \(\eta\) is singular, then
\(\eta^+\) is measurable.\qed
\end{cor}

\subsection{Preliminaries on the Ketonen order}
Recall the notion of an ultrafilter comparison (\cref{NormedKetonenDef}) that
played a role in \cref{UndefinabilityThm}. One obtains an order on ultrafilters
over ordinals by setting \(U\sE W\) if there is an ultrafilter comparison from
\((U,\text{id})\) to \((W,\text{id})\). Let us give a more concrete definition
of this order.
\begin{defn}
Suppose \(F\) is a filter over \(X\) and \(\langle G_x \mid x\in X\rangle\) is a
sequence of filters over \(Y\). Then the {\it \(F\)-limit of \(\langle G_x \mid
x\in X\rangle\)} is the filter \[F\text{-lim}_{x\in X} G_x = \{A\subseteq Y \mid
\{x\in X \mid A\in G_x\}\in F\}\]
\end{defn}

\begin{defn}\label{KetonenDef}
Suppose \(\delta\) is an ordinal. The {\it Ketonen order} is defined on
countably complete ultrafilters over \(\delta\) by setting \(U\E W\) if \(U\) is
of the form \(W\text{-lim}_{\alpha < \delta}Z_\alpha\) where \(Z_\alpha\) is a
countably complete ultrafilter over \(\delta\) concentrating on \(\alpha+1\).
\end{defn}
The following fact is easy to verify:
\begin{lma}
The Ketonen order is transitive and anti-symmetric.\qed
\end{lma}

\begin{defn}
Set \(U\sE W\) if \(U\E W\) and \(W\not \E U\).
\end{defn}

Equivalently, \(U\sE W\) if \(U\E W\) and \(U\neq W\). Also \(U\sE W\) if and
only if \(U\) is of the form \(W\text{-lim}_{\alpha < \delta}Z_\alpha\) where
\(Z_\alpha\) is a countably complete ultrafilter over \(\delta\) concentrating
on \(\alpha\). Also, as we mentioned above, \(U\sE W\) if there is an
ultrafilter comparison from \((U,\text{id})\) to \((W,\text{id})\).
In the appendix, we give a proof of the following theorem:
\begin{repthm}{KetonenWF}[DC] The Ketonen order is wellfounded.\qed
\end{repthm}

\begin{defn}[DC]\label{ODef}
The {\it Ketonen rank of \(U\)}, denoted \(\sigma(U)\), is the rank of \(U\) in
the Ketonen order.
\end{defn}

A straightforward alternate characterization of the Ketonen order turns out to
be important here.
\begin{defn}
Suppose \(\delta\) is an ordinal. A function \(h : P(\delta)\to P(\delta)\) is
{\it Lipschitz} if for all \(\alpha < \delta\) and all \(A,B\subseteq \delta\)
with \(A\cap \alpha = B\cap \alpha\), \(h(A)\cap \alpha = h(B)\cap\alpha\), and
\(h\) is {\it strongly Lipschitz} if for all \(\alpha < \delta\) and all
\(A,B\subseteq \delta\) with \(A\cap \alpha = B\cap \alpha\), \(h(A)\cap
(\alpha+1) = h(B)\cap(\alpha+1)\).
\end{defn}
\begin{lma}\label{LipKet}
Suppose \(U\) and \(W\) are countably complete ultrafilters over \(\delta\).
\begin{itemize}
\item \(U\E W\) if and only if there is a countably complete Lipschitz
homomorphism \(h : P(\delta)\to P(\delta)\) such that \(h^{-1}[W] = U\).
\item \(U\sE W\) if and only if there is a countably complete strongly Lipschitz
homomorphism \(h : P(\delta)\to P(\delta)\) such that \(h^{-1}[W] = U\).\qed
\end{itemize}
\end{lma}
Notice that an elementary embedding from \(P(\delta)\) to \(P(\delta)\) is a
countably complete Lipschitz homomorphism.
\subsection{Ordinal definability and the Ultrapower Axiom}
The Ultrapower Axiom (UA) is an inner model principle introduced by the author
in \cite{UA} to develop the general theory of countably complete ultrafilters
and in particular the theory of strongly compact and supercompact cardinals.
Implicit in the statement of UA is the assumption of the Axiom of Choice, and
dropping that assumption, there are a number of inequivalent reformulations of
the principle. In ZFC, however, the Ultrapower Axiom is equivalent to the
linearity of the Ketonen order. The following theorem therefore shows that in
one sense, the existence of a Reinhardt cardinal almost implies the Ultrapower
Axiom.

\begin{thm}
Suppose \(j : V\to V\) is an elementary embedding and
\(\kappa_\omega(j)\)-\textnormal{DC} holds. Then for any ordinals \(\delta\) and
\(\xi\), the set of countably complete ultrafilters over \(\delta\) of Ketonen
rank \(\xi\) has cardinality strictly less than \(\kappa_\omega(j)\).\qed
\end{thm}

We will prove a more technical theorem that also applies to \(L(V_{\lambda+1})\)
and other small models.
\begin{thm}\label{AlmostUA}
Suppose \(\epsilon\) is an ordinal, \(M\) is an inner model containing
\(V_{\epsilon+1}\), and \(\delta < \theta_{\epsilon+2}^M\) is an ordinal.
Suppose there is a nontrivial embedding \(j\in \mathscr E(V_{\epsilon+3}\cap
M)\) such that \(j\restriction P^M(\delta)\) belongs to \(M\). Assume \(M\)
satisfies \(\lambda\)-\textnormal{DC} where \(\lambda = \kappa_\omega(j)\). Then
in \(M\), for any ordinal \(\xi < \theta_{\epsilon+3}\), the set of countably
complete ultrafilters over \(\delta\) of Ketonen rank \(\xi\) is wellorderable
and has cardinality strictly less than \(\lambda\).
\end{thm}
Why \(V_{\epsilon+3}\)? The point of this large cardinal hypothesis is that
working in \(M\), every countably complete ultrafilter over an ordinal less than
\(\theta_{\epsilon+2}\) belongs to \(\mathcal H_{\epsilon+3}\), and moreover the
Ketonen order and its rank function are definable over \(\mathcal
H_{\epsilon+3}\). Therefore by the remarks following \cref{HAlphaLma}, an
embedding \(j\in \mathscr E(V_{\epsilon+3}\cap M)\) lifts to an embedding
\(j^\star \in \mathscr E((\mathcal H_{\epsilon+3})^M)\) that is Ketonen order
preserving and in addition respects Ketonen ranks in the sense that
\(\sigma(j^\star(U)) = j^\star(\sigma(U))\) for any countably complete
ultrafilter \(U\) over an ordinal less than \(\theta_{\epsilon+2}\). This is what
is needed for the proof of \cref{AlmostUA}. 

\begin{proof}[Proof of \cref{AlmostUA}] By considering the least counterexample,
we may assume without loss of generality that \(j\) fixes \(\delta\) and
\(\xi\).  Since \(M\) satisfies \(\lambda\)-DC, it suffices to show that in
\(M\), there is no \(\lambda\)-sequence of distinct countably complete
ultrafilters over \(\delta\) of Ketonen rank \(\xi\). Suppose towards a
contradiction that \(\langle U_\alpha \mid \alpha < \lambda\rangle\) is such a
sequence. Note that \(\langle U_\alpha \mid \alpha < \lambda\rangle\) is coded
by an element of \(V_{\epsilon+3}\cap M\), so we can apply \(j\) to it. This
yields:
\begin{align*}
\langle U_\alpha^1 \mid \alpha < \lambda\rangle &= j(\langle U_\alpha \mid \alpha < \lambda\rangle)\\
\langle U_\alpha^2 \mid \alpha < \lambda\rangle &= j(\langle U_\alpha^1 \mid \alpha < \lambda\rangle)
\end{align*}

Let \(\kappa\) be the critical point of \(j\). We claim that the following hold
for \(A\in P^M(\delta)\):
\begin{align}A\in U_\kappa^1 &\iff j(A)\in U_{j(\kappa)}^2\label{Trivial}\\
 A\in U_\kappa^1 &\iff j(i)(A)\in U_{\kappa}^2\label{LilSubtle}\end{align}
where \(i = j\restriction P^M(\delta)\). \cref{Trivial} is trivial since
 \(j(U_\kappa^1) = U_{j(\kappa)}^2\). \cref{LilSubtle} is slightly more subtle
 because we are not assuming \(j\restriction P(P(\delta))\) belongs to \(M\),
 and therefore \(j(j)(U_\kappa^1)\) is not obviously well-defined. Note however
 that for all \(\alpha < \kappa\), \(A\in U_\alpha\) if and only if \(i(A)\in
 U_\alpha^1\). By the elementarity of \(j\), for all \(\alpha < j(\kappa)\),
 \(A\in U_\alpha^1\) if and only if \(j(i)(A)\in U_\alpha^2\). In particular,
 \(A\in U_\kappa^1\) if and only if \(j(i)(A)\in U_\kappa^2\), proving
 \cref{LilSubtle}.

Now notice that in \(M\), \(i\) and \(j(i)\) are countably complete Lipschitz
homomorphisms from \(P(\delta)\) to \(P(\delta)\). Therefore by the
characterization of the Ketonen order in terms of Lipschitz homomorphisms
(\cref{LipKet}), \cref{Trivial} and \cref{LilSubtle} imply: 
\begin{align*}U_\kappa^1&\E U_{j(\kappa)}^2\\
U_\kappa^1&\E U_{\kappa}^2\end{align*} 
Now we use that \(j\) fixes \(\xi\), which implies that the three ultrafilters
\(U_\kappa^1, U_{\kappa}^2,\) and \(U_{j(\kappa)}^2\) have Ketonen rank \(\xi\).
Since the Ketonen order is wellfounded, it follows that the inequalities above
cannot be strict, and so since the Ketonen order is antisymmetric,
\(U_{j(\kappa)}^2 = U_\kappa^1 = U_\kappa^2\). This contradicts that \(\langle
U_\alpha^2 \mid \alpha < \lambda\rangle \) is a sequence of distinct
ultrafilters.
\end{proof}

Although \cref{AlmostUA} shows that the Ketonen order is {\it almost} linear
under choiceless large cardinal assumptions, true linearity is incompatible with
\(\kappa_\omega(j)\)-DC:
\begin{prp}Suppose \(j : V_{\lambda+2}\to V_{\lambda+2}\) is an elementary
embedding with \(\lambda = \kappa_\omega(j)\). Suppose \(\lambda\) is a limit of
L\"owenheim-Skolem cardinals. Assume the restriction of the Ketonen order to
\(\lambda^+\)-complete ultrafilters over \(\lambda^+\) is linear. Then
\(\omega_1\)\textnormal{-DC} is false.
\begin{proof}
We begin by outlining the proof. Assume towards a contradiction that
\(\omega_1\)\textnormal{-DC} is true. We first prove that the Ketonen least
ultrafilter \(U\) on \(\lambda^+\), which exists by the linearity of the Ketonen
order, extends the \(\omega\)-closed unbounded filter. Next, we show that,
assuming \(\omega_1\)-DC, there is a normal ultrafilter \(W\) on \(\lambda^+\)
extending the \(\omega_1\)-closed unbounded filter. 

Let us show that the existence of the ultrafilters \(U\) and \(W\) actually
implies that \(\omega_1\) is measurable, contradicting \(\omega_1\)-DC. Clearly
\(U\sE W\). As a consequence \(U = W\text{-lim}_{\alpha < \lambda^+}U_\alpha\)
where \(U_\alpha\) is a countably complete ultrafilter such that \(\alpha\in
U_\alpha\) for \(W\)-almost all \(\alpha < \lambda^+\). For each \(\alpha <
\lambda^+\), let \(\delta_\alpha\) be the least ordinal such that
\(\delta_\alpha\in U_\alpha\). Then \(\delta_\alpha = \alpha\) for \(W\)-almost
all \(\alpha < \lambda^+\): otherwise, since \(W\) is normal, there is a
\(\delta < \lambda^+\) such that \(\delta_\alpha = \delta\) for \(W\)-almost all
\(\alpha < \lambda^+\), which implies \(\delta\in W\text{-lim}_{\alpha <
\lambda^+}U_\alpha = U\), contradicting that \(U\) is a uniform ultrafilter over
\(\lambda^+\). Fix an ordinal \(\alpha < \lambda^+\) of cofinality \(\omega_1\)
such that \(\delta_\alpha = \alpha\). Then \(U_\alpha\) is a fine ultrafilter
over \(\alpha\); that is, every set in \(U\) is cofinal in \(\alpha\). This
implies that \(\text{cf}(\alpha)\) carries a uniform countably complete
ultrafilter \(D\): let \(f : \alpha\to \omega_1\) be any monotone function and
let \(D = f_*(U_\alpha)\). This means that \(\omega_1\) is measurable, which is
a contradiction.

The first step is to show that DC implies that there is a normal filter over
\(\lambda^+\) extending the \(\omega\)-closed unbounded filter. For this, we
show that the {\it weak club filter} is normal. This is the filter \(F\)
generated by sets of the form \[\{\sup (\sigma\cap \lambda^+) \mid \sigma \prec
M\}\] where \(M\) is a structure in a countable language containing
\(\lambda^+\). The normality of this filter, given the L\"owenheim-Skolem
hypothesis, is proved by Usuba as \cite[Proposition 3.5]{UsubaLS}. By DC, the
set \[S = \{\alpha < \lambda^+ \mid \text{cf}(\alpha) = \omega\}\] is
\(F\)-positive. Hence \(F\restriction S\) is a normal filter extending the
\(\omega\)-closed unbounded filter. By the Woodin argument and \cref{LSUlam},
\(F\restriction S\) is atomic. Therefore there is some \(T\subseteq S\) such
that \(F\restriction T\) is an ultrafilter. Of course \(F\restriction T\) is
normal since \(F\) is, and hence we have obtained a normal ultrafilter \(U\)
extending the \(\omega\)-closed unbounded filter. 

We claim that no uniform ultrafilter over \(\lambda^+\) lies below \(U\) in the
Ketonen order. To see this, suppose \(Z\sE U\), and we will show that there is
some \(\delta < \lambda^+\) such that \(\delta\in Z\). Suppose \(Z =
U\text{-lim}_{\xi < \lambda^+}D_\xi\) where \(D_\xi\) is a countably complete
ultrafilter over \(\lambda^+\) with \(\xi \in D_\xi\) for \(U\)-almost all \(\xi
< \lambda^+\). For \(U\)-almost all \(\xi < \lambda^+\), \(\xi\) has cofinality
\(\omega\) and \(D_\xi\) is a countably complete ultrafilter with \(\xi\in
D_\xi\), so there is some \(\delta_\xi < \xi\) with \(\delta_\xi\in D_\xi\).
Since \(U\) is normal, there is a fixed ordinal \(\delta < \lambda^+\) such that
\(\delta_\xi = \delta\) for \(U\)-almost all \(\xi < \lambda^+\). Hence \(\delta
\in D_\xi\) for \(U\)-almost all \(\xi < \lambda^+\), so since \(Z =
U\text{-lim}_{\xi < \lambda^+}D_\xi\), \(\delta\in Z\).

Thus \(U\) is the Ketonen least ultrafilter over \(\lambda^+\). 

Using \(\omega_1\)-DC, one can show that \(\{\alpha < \lambda^+ \mid
\text{cf}(\alpha) = \omega_1\}\) is positive with respect to the weak club
filter. It follows as above that the \(\omega_1\)-closed unbounded filter
extends to a normal ultrafilter. As explained in the first two paragraphs, this
leads to the conclusion that \(\omega_1\) is measurable, which contradicts our
assumption that \(\omega_1\)-DC holds.
\end{proof}
\end{prp}
It would not be that surprising if it turned out to be possible to refute the
linearity of the Ketonen order outright from choiceless large cardinals. If the
linearity of the Ketonen order is consistent with choiceless large cardinals,
however, then perhaps there is an interesting theory of choiceless large
cardinals in which choice fails low down. We will not pursue this idea further
here since it leads to highly speculative territory. We do note that one can
make do with a weaker choice assumption in the proof of \cref{AlmostUA}:
\begin{thm}\label{ChoicelessUA}
Suppose \(\epsilon\) is an ordinal, \(M\) is an inner model of \textnormal{DC}
containing \(V_{\epsilon+1}\), and \(\delta < \theta_{\epsilon+2}^M\) is an
ordinal. Suppose there is a nontrivial \(j\in \mathscr E(V_{\epsilon+3}\cap M)\)
such that \(j\restriction P^M(\delta)\) belongs to \(M\). Assume \(\lambda =
\kappa_\omega(j)\) is a limit of L\"owenheim-Skolem cardinals in \(M\). Then the
following hold in \(M\):
\begin{enumerate}[(1)]
\item For any ordinal \(\xi < \theta_{\epsilon+3}\), the set of countably
complete ultrafilters over \(\delta\) of Ketonen rank \(\xi\) is the surjective
image of \(V_\alpha\) for some \(\alpha < \lambda\).
\item The set of \(V_\lambda\)-complete ultrafilters of Ketonen rank \(\xi\) is
wellorderable and has cardinality strictly less than \(\lambda\).
\end{enumerate}
\end{thm}

The proof uses the following lemma, whose analog in the context of AC is well
known and does not require the supercompactness assumption that we make below:
\begin{lma}\label{CompletenessLSLemma}
Assume \(M\) is a model of set theory and \(j : M\to N\) is an elementary
embedding with critical point \(\kappa\). Assume there is a L\"owenheim-Skolem
cardinal \(\eta\) in \(M\) such that \(\kappa < \eta < j(\kappa)\) and
\(j[V_\xi\cap M]\in N\) for all \(\xi < \eta\). Then for any set \(S\in M\) such
that \(j(S) = j[S]\), there is some \(\alpha < \kappa\) such that \(M\)
satisfies \(S\preceq^{*} V_\alpha\).
\begin{proof}
We first observe that if there is a surjection \(f : S\to S'\) in \(M\), then
\[j(S') = j(f)[j(S)] = j(f)[j[S]] = j[S']\] Work in \(M\), and assume towards a
contradiction that \(S\not \preceq^{*} V_\alpha\) for any \(\alpha < \kappa\).
Fix \(\gamma > \beta\) and an elementary substructure \(X\prec V_\gamma\) with
\(S\in X\), \([X]^{V_\alpha}\subseteq X\), and for some \(\nu <\eta\),
\(X\preceq^{*} V_\nu\). Let \(S' = X\cap S\). Notice that there is no surjection
from \(V_\alpha\) to \(S'\) for any \(\alpha < \kappa\): if there is, then
\(S'\in X\) since \(X\) is closed under \(V_\alpha\)-sequences, and hence \(S' =
S\) because \(S\in X\) and \((S\setminus S')\cap X = \emptyset\); but then
\(S\preceq^{*} V_\alpha\), which is a contradiction. Let \(\xi\) be the least
rank of a set \(a\) that is in bijection with \(S'\). Then \(\kappa \leq \xi <
\eta\). 

We now leave \(M\). On the one hand, \(j(\xi) > j(\kappa) > \eta > \xi\). On the
other hand, \(j(S') = j[S']\) by our first observation. Let \(f : a\to S'\) be a
bijection in \(M\), and notice that \(a\in N\) and \(j(f)\circ j\restriction
a\in N\) is a bijection between \(j[S']\) and \(a\) that belongs to \(N\).
Therefore in \(N\), \(|j(S')| = |a|\). It follows that in \(N\), \(\xi\) is the
least rank of a set in bijection with \(j(S')\). This contradicts that \(j(\xi)
> \xi\). 
\end{proof}
\end{lma}

\begin{proof}[Proof of \cref{ChoicelessUA}] Suppose towards a contradiction that
the theorem fails. 

We work in \(M\) for the time being. Let \(\xi < \theta_{\epsilon+3}\) be the
least ordinal such that the set of countably complete ultrafilters over
\(\delta\) of Ketonen rank \(\xi\) is not the surjective image of \(V_\beta\)
for any \(\beta < \lambda\). Let \(S\) be any set of countably complete
ultrafilters over \(\delta\) of Ketonen rank \(\xi\). Leaving \(M\), a
generalization of the proof of \cref{AlmostUA} will show that \(j(S) = j[S]\).
Assume otherwise, fix \(U\in j(S)\setminus j[S]\) and consider \(j(U)\) and
\(j(j)(U)\). On the one hand, these ultrafilters must be distinct: by
elementarity \(j(U)\notin j(j[S]) = j(j)[j(S)]\), whereas evidently \(j(j)(U)\in
j(j)[j(S)]\). On the other hand, \(j\restriction P(\delta)\) and \(j(j)
\restriction P(\delta) \) witness that \(U\E j(U),j(j)(U)\) in \(M\), and
therefore \(U = j(U) = j(j)(U)\) since \(\sigma(U) = \sigma(j(U)) =
\sigma(j(j)(U)) = \xi\) since of course \(j\) and \(j(j)\) fix the definable
ordinal \(\xi\). This is a contradiction, so in fact \(j(S) = j[S]\).

By \cref{CompletenessLSLemma}, \(M\) satisfies that \(S \preceq^{*} V_\alpha\)
for some \(\alpha < \text{crit}(j)\). This contradiction proves (1).

Now consider the set \(S\) of \(V_\lambda\)-complete ultrafilters over
\(\delta\) of Ketonen rank \(\xi\). By an argument similar to that of
\cref{LSUlam}, one can use the L\"owenheim-Skolem assumption to find a
``discretizing family'' for \(S\), or in other words a function \(f : S\to
P(\delta)\) such that \(f(U) \in U\setminus W\) for all \(W\in S\) except for
\(U\). Then the function \(g(U) = \min f(U)\) is an injection from \(S\) into
\(\delta\), so \(S\) is wellorderable. Since \(\theta_{\alpha} < \lambda\) for
all \(\alpha < \lambda\), it follows that \(|S| < \lambda\), proving (2).
\end{proof}
We remark that an argument similar to the proof of \cref{ChoicelessUA} can be
used to establish the Coding Lemma (\cref{CodingLemma}) from a
L\"owenheim-Skolem hypothesis rather than dependent choice.

The semi-linearity of the Ketonen order given by \cref{AlmostUA} implies that
\(V\) is in a sense ``close to \(\textnormal{HOD}\).'' (No such closeness result
is known to be provable from large cardinal axioms consistent with the Axiom of
Choice, so this perhaps complicates the intuition that choiceless large cardinal
axioms imply that \(\text{HOD}\) is a small model.) We first state the theorem
in two special cases:

\begin{thm}
Suppose \(\lambda\) is a cardinal such that \(\lambda\)-\textnormal{DC} holds.
Assume \(M = L(V_{\lambda+1})\) or \(M = V\). Suppose there is an elementary
embedding \(j : M \to M\) with \(\kappa_\omega(j) = \lambda\). Then \(M\)
satisfies the following statements:
\begin{enumerate}[(1)]
\item Every countably complete ultrafilter over an ordinal belongs to an ordinal
definable set of size less than \(\lambda\).
\item Every \(\lambda^+\)-complete ultrafilter over an ordinal \(\delta\) is
ordinal definable from a subset of \(\delta\).
\item For any set of ordinals \(S\), every \(\lambda^+\)-complete ultrafilter is
amenable to \(\textnormal{HOD}_S\).
\item For any \(\lambda^+\)-complete ultrafilter \(U\) over an ordinal, the
ultrapower embedding \(j_U\) is amenable to \(\textnormal{HOD}_x\) for a cone of
\(x\in V_\lambda\).\qed
\end{enumerate}
\end{thm}

We now prove a more technical result that immediately implies the previous
theorem.
\begin{thm}\label{ODThm}
Suppose \(\epsilon\) is an even ordinal, \(M\) is an inner model of
\textnormal{DC} containing \(V_{\epsilon+1}\). Assume there is an elementary
embedding \(j\) from \(V_{\epsilon+3}\cap M\) to \(V_{\epsilon+3}\cap M\) such
that \(j\restriction P(\delta)\in M\) for all \(\delta <
\theta_{\epsilon+2}^M\). Assume \(\lambda = \kappa_\omega(j)\) is a limit of
L\"owenheim-Skolem cardinals in \(M\). Then the following hold in \(M\):
\begin{enumerate}[(1)]
\item Every countably complete ultrafilter over an ordinal below
\(\theta_{\epsilon+2}\) belongs to an ordinal definable set of size less than
\(\lambda\).
\item Every \(\lambda^+\)-complete ultrafilter over an ordinal \(\delta <
\theta_{\epsilon+2}\) is ordinal definable from a subset of \(\delta\).
\item For any set of ordinals \(S\), every \(\lambda^+\)-complete ultrafilter
over an ordinal \(\delta < \theta_{\epsilon+2}\) is amenable to
\(\textnormal{HOD}_S\).
\item For any \(\lambda^+\)-complete ultrafilter \(U\) over an ordinal less than
\(\theta_{\epsilon+2}\), the ultrapower embedding \(j_U\) is amenable to
\(\textnormal{HOD}_x\) for a cone of \(x\in V_\lambda\).
\end{enumerate}
\begin{proof}
We work entirely in \(M\), using only the conclusion of \cref{ChoicelessUA}.

(1) is clear from \cref{AlmostUA}. 

For (2), suppose \(U\) is a countably complete ultrafilter over \(\delta\). Let
\(\xi\) be the Ketonen rank of \(U\), and let \(\langle U_\alpha: \alpha <
\eta\rangle\) enumerate the \(\lambda^+\)-complete ultrafilters of Ketonen rank
\(\xi\). Choose a set \(A\subseteq \delta\) such that \(A\in U\) and \(A\notin
U_\alpha\) for any \(\alpha < \eta\); this is possible because \(\eta <
\lambda\) and the ultrafilters in question are \(\lambda^+\)-complete. Since
\(U\) is the unique \(\lambda^+\)-complete ultrafilter over \(\delta\) of
Ketonen rank \(\xi\) such that \(A\in U\), \(U\) is ordinal definable from
\(A\).

We now prove (3). Let \(\bar U = U\cap \textnormal{HOD}_S\). We must show that
\(\bar U\in \textnormal{HOD}_S\). Fix an OD set \(P\) of cardinality less than
\(\lambda\) such that \(U\in P\). Note that \(F = \bigcap P\) is ordinal
definable. Let \(\bar F = F\cap \textnormal{HOD}_S\). Then \(\bar F\in
\textnormal{HOD}_S\). Using \(\lambda^+\)-completeness, it is obvious that \(F\)
is \(\lambda\)-saturated, and it follows that \(\bar F\) is
\(\lambda\)-saturated in \(\textnormal{HOD}_S\). Applying the Ulam splitting
theorem inside \(\textnormal{HOD}_S\), there is some \(\eta < \lambda\) and a
partition \(\langle A_\alpha \mid \alpha <\eta\rangle\in \text{HOD}_S\) of
\(\delta\) into atoms of \(\bar F\). Since \(\bigcup A_\alpha = \delta\), there
is some \(\alpha < \eta\) such that \(A_\alpha\in U\). It follows that \(\bar
F\restriction A_\alpha \subseteq U\cap \text{HOD}_S = \bar U\), and since \(\bar
F\restriction A_\alpha\) is a \(\textnormal{HOD}_S\)-ultrafilter, this implies
that \(\bar F \restriction A_\alpha = \bar U\). Clearly \(\bar F\restriction
A_\alpha\in \textnormal{HOD}\), and therefore so is \(\bar U\).

We only sketch the proof of (4), which requires knowledge of the proof of Vop\v
enka's Theorem. (This is the theorem stating that every set of ordinals is
set-generic over HOD; see \cite{Woodin}.) We first show that for any cardinal
\(\gamma\), there is some \(x\in V_\lambda\) such that \(j_U\restriction
P(\gamma)\) is amenable to \(\textnormal{HOD}_x\). Let \(E\) be the
\(\textnormal{HOD}\)-extender of length \(j_U(\gamma)\) derived from \(j_U\).
Notice that \(E\subseteq \textnormal{HOD}\) by (3), and moreover \(E\) belongs
to an ordinal definable set \(X\) of size less than \(\lambda\) since \(U\)
does. 

The set \(X\) is (essentially) a condition in the Vop\v enka forcing to add
\(E\) to \(\textnormal{HOD}\), and below this condition, the Vop\v enka algebra
has cardinality less than \(\lambda\), since it is isomorphic to \(P(X)\cap
\text{OD}\). It follows that \(E\) belongs to \(\textnormal{HOD}_x\) where
\(x\in V_\lambda\) is the generic for this Vop\v enka forcing below the
condition given by \(X\). Each ultrafilter of \(E\) lifts uniquely to an
ultrafilter of \(\textnormal{HOD}_x\) by the L\'evy-Solovay Theorem
\cite{LevySolovay}: these ultrafilters are \(\lambda^+\)-complete and \(x\) is
HOD-generic for a forcing of size less than \(\lambda\). It follows that the
\(\textnormal{HOD}_x\)-extender of \(j_U\) of length \(j_U(\gamma)\) can be
computed from \(E\) inside \(\textnormal{HOD}_x\), simply by lifting all the
measures of \(E\) to \(\textnormal{HOD}_x\). But from this extender, one can
decode \(j_U\restriction P(\gamma)\cap \textnormal{HOD}_x\). This shows that
there is some \(x\in V_\lambda\) such that \(j_U\restriction P(\gamma)\) is
amenable to \(\textnormal{HOD}_x\). 

If follows from the pigeonhole principle that there is some \(x_0\in V_\lambda\)
such that \(j_U\) is amenable to \(\textnormal{HOD}_{x_0}\). Now for any
\(x\geq_{\text{OD}} {x_0}\), \(j_U\) is amenable to \(\text{HOD}_x\) by exactly
the same argument we used above to show that \(E\) extends from \(\text{HOD}\)
to \(\text{HOD}_x\). This proves (4).
\end{proof}
\end{thm}

The fixed point filter associated to a set of elementary embeddings plays a key
role in the theory developed in \cite{SEM2}:
\begin{defn}\label{FixedPointFilter}
Suppose \(j\) is a function and \(X\) is a set. Then 
\[\text{Fix}(j,X) = \{x\in X \mid j(x) = x\}\] Suppose \(\sigma\) is a set of
functions. Then \(\text{Fix}(\sigma,X) = \bigcap_{j\in \sigma}\text{Fix}(j,X)\).

Suppose \(\mathcal E\) is a set of elementary embeddings whose domains contain
the set \(X\). Suppose \(B\) is a set. Then the {\it fixed point filter
\(B\)-generated by \(\mathcal E\) on \(X\)}, denoted \(\mathcal F^{B}(\mathcal
E,X)\), is the filter over \(X\) generated by sets of the form
\(\text{Fix}(\sigma,X)\) where \(\sigma\subseteq \mathcal E\) and
\(\sigma\preceq^{*} b\) for some \(b\in B\).
\end{defn}
The sort of techniques we have been using yield the following representation
theorem for ultrafilters over ordinals, which says that in the land of
choiceless cardinals, every ultrafilter over an ordinal is one set away from a
fixed point filter:
\begin{thm}[DC]\label{FixChar}
Suppose \(\epsilon\) is an ordinal and \(j: V_{\epsilon+3}\to V_{\epsilon+3}\)
is an elementary embedding. Assume \(\lambda = \kappa_\omega(j)\) is a limit of
L\"owenheim-Skolem cardinals. Suppose \(\delta < \theta_{\epsilon+2}\) is an
ordinal and \(U\) is a \(V_\lambda\)-complete ultrafilter over \(\delta\). Then
there is an ordinal definable set of elementary embeddings \(\mathcal E\) and a
set \(A\subseteq \delta\) such that \(U = \mathcal F^{V_\lambda}(\mathcal
E,\delta)\restriction A\).
\begin{proof}
The proof is by contradiction. Suppose \(\xi\) is least possible Ketonen rank of
a \(V_{\lambda}\)-complete ultrafilter over an ordinal for which the theorem
fails. Obviously \(\xi < \theta_{\epsilon+3}\). Let \(j : V_{\epsilon+3}\to
V_{\epsilon+3}\) be an elementary embedding. Note that \(\xi\) is definable in
\(\mathcal H_{\epsilon+3}\), and therefore \(j(\xi)= \xi\). It follows that
\(j(U) = U\) for any \(V_{\lambda}\)-complete ultrafilter over an ordinal
\(\delta < \theta_{\epsilon+2}\) of Ketonen rank \(\xi\): as in \cref{AlmostUA},
\(j\restriction P(\delta)\) is a countably complete Lipschitz homomorphism
witnessing \(U\E j(U)\), while \(j(U)\) has rank \(\xi\) since \(j(\xi) = \xi\).
Let \(\mathcal E\) be the set of \(\Sigma_5\)-elementary embeddings \(k :
V_{\epsilon+3}\to V_{\epsilon+3}\) such that \(k(\xi) = \xi\). The argument we
have just given shows that \(k(U) = U\) for any \(k\in \mathcal E\).

Let \(F = \mathcal F^{V_\lambda}(\mathcal E,\delta)\). Clearly, \(F\) is ordinal
definable: in fact, \(F\) is definable over \(\mathcal H_{\epsilon+3}\) from an
ordinal parameter. We claim \(F\) is \(\kappa\)-saturated where \(\kappa =
\text{crit}(j)\). This follows from Woodin's proof of the Kunen inconsistency
theorem. Suppose \(F\) is not \(\kappa\)-saturated, so there is a partition
\(\langle S_\alpha \mid \alpha < \kappa\rangle\) of \(\delta\) into pairwise
disjoint \(F\)-positive sets. Let \(\langle T_\alpha \mid \alpha <
j(\kappa)\rangle = j(\langle S_\alpha \mid \alpha < \kappa\rangle)\). Since
\(F\) is first-order definable over \(\mathcal H_{\epsilon+3}\), \(T_\alpha\) is
\(F\)-positive for all \(\alpha\). In particular, \(T_\kappa\) is
\(F\)-positive, or in other words, \(T_\kappa\) has nonempty intersection with
every set in \(F\). The set of fixed points of \(j\) below \(\delta\) belongs to
\(F\), so \(T_\kappa\) contains an ordinal \(\eta\) that is fixed by \(j\). Now
\(\eta\in S_\alpha\) for some \(\alpha < \kappa\), and therefore \( \eta =
j(\eta)\in j(S_\alpha) = T_{j(\alpha)}\). It follows that \(T_{j(\alpha)}\cap
T_\kappa\neq \emptyset\), so since the sets \(\langle T_\alpha \mid \alpha <
\kappa\rangle\) are pairwise disjoint, \(j(\alpha) = \kappa\). This contradicts
that \(\kappa\) is the critical point of \(j\).

Using the L\"owenheim-Skolem cardinals, it is easy to show that \(F\) is
\(V_{\lambda}\)-complete. Therefore by \cref{LSUlam}, \(F\) is atomic. 

We now show that \(F\subseteq U\). Suppose \(k\in \mathcal E\). As we noted in
the first paragraph, \(k(U) = U\). Therefore \(k\) is a countably complete
Lipschitz homomorphism with \(k^{-1}[U] = U\). If \(U\) contains the set of
ordinals that are not fixed by \(k\), then \(k\) witnesses that \(U\) is
strictly below \(U\) in the Ketonen order, which is impossible. Since \(U\) is
an ultrafilter, \(U\) must instead concentrate on fixed points of \(k\). Since
\(U\) is \(V_{\lambda}\)-complete, it follows that \(U\) contains the basis
generating \(F\) as in \cref{FixedPointFilter}, so \(F\subseteq U\). 

Since \(F\) is atomic, there is some atom \(A\) of \(F\) such that \(U =
F\restriction A\), and this completes the proof.
\end{proof}
\end{thm}

A number of interesting questions remain. We state them in the context of
\(I_0\), which is arguably the simplest special case, but obviously the same
questions are relevant in the choiceless large cardinal context.
\begin{qst}Assume \(I_0\). In \(L(V_{\lambda+1})\), is there a surjection from
\(V_{\lambda+1}\) onto the set of \(\lambda^+\)-complete ultrafilters over
\(\lambda^+\)?
\end{qst}
The question is at least somewhat subtle, since one can show that in
\(L(V_{\lambda+1})\), there is a \(\delta < \theta_{\lambda+2}\) such that there
is no surjection from \(V_{\lambda+1}\) onto the set of \(\lambda^+\)-complete
ultrafilters over \(\delta\). In fact, one can take \(\delta =
(\delta^2_1)^{L(V_{\lambda+1})}\). This is exactly parallel to the situation in
\(L(\mathbb R)\). An even more basic question is whether the ultrapower of
\(\lambda^+\) by the unique normal ultrafilter over \(\lambda^+\) concentrating
on ordinals of cofinality \(\omega\) is smaller than \(\lambda^{+\lambda}\).

Another question, directly related to \cref{AlmostUA}, concerns the size of
antichains in the Ketonen order:
\begin{qst}
Assume \(I_0\). In \(L(V_{\lambda+1})\), if \(\langle U_\alpha \mid \alpha <
\lambda\rangle\) is a sequence of \(\lambda^+\)-complete ultrafilters over
ordinals, must there be \(\alpha \leq \beta < \lambda\) such that \(U_\alpha \E
U_\beta\)?
\end{qst}
A positive answer would bring us even closer to a ``proof of the Ultrapower
Axiom" from choiceless cardinals. Actually one can prove a weak version of this
for \(\lambda^+\)-sequences of ultrafilters, whose statement and proof are
omitted.
\subsection{The filter extension property}
We now turn to a different application of the Ketonen order: extending filters
to ultrafilters. For this, we need the Ketonen order on countably complete
filters, which was introduced in \cite{UA}:
\begin{defn}\label{KFilterDef}
Suppose \(\delta\) is an ordinal. The {\it Ketonen order on filters} is defined
on countably complete filters \(F\) and \(G\) on \(\delta\) as follows:
\begin{itemize}
\item \(F\sE G\) if \(F\subseteq G\text{-lim}_{\alpha < \delta} F_\alpha\) where
for \(G\)-almost all \(\alpha < \delta\), \(F_\alpha\) is a countably complete
filter over \(\delta\) with \(\alpha\in F_\alpha\).
\item \(F\E G\) if \(F\subseteq G\text{-lim}_{\alpha < \delta} F_\alpha\) where
for \(G\)-almost all \(\alpha < \delta\), \(F_\alpha\) is a countably complete
filter over \(\delta\) with \(\alpha+1\in F_\alpha\).
\end{itemize}
\end{defn}
The notation is a bit unfortunate since the Ketonen order on ultrafilters
(\cref{KetonenDef}) need not be equal to the restriction of the Ketonen order on
filters to the class of ultrafilters (although the latter order is an extension
of the former one). For example, assuming \(I_0\), in \(L(V_{\lambda+1})\), any
\(\omega\)-club ultrafilter over \(\lambda^+\) lies below any \(\omega_1\)-club
ultrafilter over \(\lambda^+\) in the Ketonen order on filters, but not in the
Ketonen order on ultrafilters. (We do not know whether the two Ketonen orders
can diverge assuming ZFC, though it seems very likely that this can be forced.)
The Ultrapower Axiom obviously implies that the Ketonen order on filters
coincides with the Ketonen order on ultrafilters. We will not make substantial
use of the Ketonen order on ultrafilters for the rest of the paper, so this
ambiguity causes no real problem.

A distinctive feature of the Ketonen order on filters is that \(\E\) is not
antisymmetric; similarly \(F\E G\) but \(G\not \E F\) does not imply \(F\sE G\).
This makes it hard to generalize arguments like \cref{AlmostUA} from countably
complete ultrafilters to countably complete filters. Still, many of the key
combinatorial properties of the Ketonen order do generalize. For example, it is
easy to see that the Ketonen order on filters is transitive. Most importantly,
we show in the appendix that this order is wellfounded:
\begin{repthm}{KFilterWF}[DC] The Ketonen order on countably complete filters is
wellfounded.\qed
\end{repthm}

In the choiceless context, we say a cardinal \(\kappa\) is {\it strongly
compact} if for every set \(X\), there is a \(\kappa\)-complete fine ultrafilter
over \(P_\kappa(X)\). Suppose \(j : V\to V\) is an elementary embedding and
\(\lambda\)-DC holds where \(\lambda = \kappa_\omega(j)\). It seems possible
that \(\lambda^+\) is then strongly compact. While we do not know how to prove
this, and expect it is not provable, we can establish a consequence of strong
compactness that is equivalent to strong compactness in ZFC. The consequence we
are referring to is the {\it filter extension property}, which is said
to hold at \(\kappa\) if every \(\kappa\)-complete filter over an ordinal
extends to a \(\kappa\)-complete ultrafilter. If \(\kappa\) is strongly compact,
then a standard argument, which does not require the Axiom of Choice, shows that
the filter extension property holds at \(\kappa\). (On the other hand,
the proof that every \(\kappa\)-complete filter extends to a \(\kappa\)-complete
ultrafilter does use the Axiom of Choice, and in fact any cardinal with this
stronger form of the filter extension property must be inaccessible.)
\begin{thm}
Suppose \(j : V\to V\) is an elementary embedding. Assume
\(\lambda\)-\textnormal{DC} holds where \(\lambda = \kappa_\omega(j)\). Then
every \(\lambda^+\)-complete filter over an ordinal extends to a \(\lambda^+\)-complete
ultrafilter.\qed
\end{thm}
This is an immediate consequence of the following more local theorem:
\begin{thm}\label{FilterExtensionThm}
Suppose \(\epsilon\) is an even ordinal and \(\nu \leq \epsilon\) is a limit of
L\"owenheim-Skolem cardinals. Suppose there is an elementary embedding \(j:
V_{\epsilon+3}\to V_{\epsilon+3}\) with \(\kappa_\omega(j) \leq \nu\). Then
every \(V_\nu\)-complete filter over an ordinal less than
\(\theta_{\epsilon+2}\) extends to a \(V_\nu\)-complete ultrafilter. 
\begin{proof}
For any elementary embedding \(k : V_{\epsilon+2}\to V_{\epsilon+2}\), let \(k'
: \mathcal H_{\epsilon+3}\to \mathcal H_{\epsilon+3}\) be defined by \(k' =
(k^{+})^\star\), assuming that \(k^{+} : V_{\epsilon+3}\to V_{\epsilon+3}\) is
\(\Sigma_1\)-elementary, so that \((k^{+})^\star\) is well-defined. 

Let \(\lambda = \kappa_\omega(j)\), where \(j\) is as in the statement of the
theorem. We begin with a basic observation, whose proof is lifted from a claim
in \cite{KWB}: for any \(n < \omega\) and any \(\xi\leq \epsilon\), there is a
\(\Sigma_n\)-elementary embedding \(i : \mathcal H_{\epsilon+3}\to \mathcal
H_{\epsilon+3}\) such that \(\kappa_\omega(i) = \lambda\) and \(i(\xi) = \xi\).
To see this, suppose the claim fails for some \(n\). Consider the least \(\xi\)
for which there is no such embedding. Then \(\xi\) is first-order definable from
\(\lambda\) over \(\mathcal H_{\epsilon+3}\), so \(j'(\xi) = \xi\). But then
\(j'\) itself witnesses that \(\xi\) is not a counterexample to our basic
observation, and this is a contradiction.

Suppose towards a contradiction that the theorem fails. Fix an ordinal \(\eta <
\theta_{\epsilon+2}\) and a filter \(F\) that is minimal in the Ketonen order
among all \(V_{\nu}\)-complete filters over \(\eta\) that do not extend to
\(V_\nu\)-complete ultrafilters. 

Fix natural numbers \(n_0 < n_1 < n_2\) that are sufficiently far apart for the
following proof to work. For concreteness, one can take \(n_0 = 10\), \(n_1 =
15\), and \(n_2 = 20\).

Let \(\mathcal E\) be the set containing every elementary embedding \(k :
V_{\epsilon+2}\to V_{\epsilon+2}\) whose extension \(k' : \mathcal
H_{\epsilon+3}\to \mathcal H_{\epsilon+3}\) is well-defined and
\(\Sigma_{n_0}\)-elementary, fixes \(\nu\) and \(\eta\), and has \(F\) in its
range. Note that \(\mathcal E\) is \(\Sigma_{n_1}\)-definable over
\(V_{\epsilon+3}\). Since \(\nu\) and \(\eta\) can be coded by a single ordinal
\(\xi\leq\epsilon\), we can fix a \(\Sigma_{n_2}\)-elementary embedding \(i :
\mathcal H_{\epsilon+3}\to \mathcal H_{\epsilon+3}\) with \(\kappa_\omega(i) =
\lambda\), \(i(\nu) = \nu\), and \(i(\eta) = \eta\).

Let \(G = \mathcal F^{V_\nu}(\mathcal D,\eta)\) where \(\mathcal D\) is the set
of embeddings \(k' : \mathcal H_{\epsilon+3}\to \mathcal H_{\epsilon+3}\)
induced by embeddings \(k\in \mathcal E\) as in the previous paragraph. The
filter \(G\) is \(\lambda\)-saturated, as a consequence of Woodin's proof of the
Kunen inconsistency theorem. Suppose towards a contradiction that there is a
partition \(\langle S_\alpha : \alpha < \lambda\rangle\) of \(\eta\) into
\(G\)-positive sets. Let \[\langle T_\alpha : \alpha < \lambda\rangle =
i(\langle S_\alpha :\alpha < \lambda\rangle)\] Since \(\mathcal E\) is
\(\Sigma_{n_1}\)-definable over \(V_{\epsilon+3}\) and \(i\) is
\(\Sigma_{n_2}\)-elementary on \(V_{\epsilon+3}\), \(i(G) = \mathcal
F^{V_{\nu}}(\mathcal D,\eta)\), where \(\mathcal D\) is the set of embeddings
\(k' : \mathcal H_{\epsilon+2}\to \mathcal H_{\epsilon+2}\) induced by
embeddings \(k\in i(\mathcal E)\). Notice that \(i\restriction V_{\epsilon+2}\in
i(\mathcal E)\): this follows easily by our choice of \(i\) and the fact that
\(i(F)\in \text{ran}(i)\). Let \(\kappa\) be the critical point of \(i\). Since
\(T_\kappa\) is \(i(G)\)-positive, it follows that \(\{\xi \mid i(\xi) =
\xi\}\cap T_\kappa\) is nonempty. Fix \(\xi\) such that \(i(\xi) = \xi\) and
\(\xi\in T_\kappa\). Note that \(\xi\in S_\alpha\) for some \(\alpha\) since
\(\langle S_\alpha \mid \alpha < \lambda\rangle\) is a partition of \(\eta\).
Therefore since \(i(\xi) = \xi\), \(\xi \in i(S_\alpha) = T_{i(\alpha)}\). Since
\(\kappa\) is the critical point of \(i\), \(i(\alpha)\neq \kappa\). But
\(\xi\in T_\kappa\cap T_{i(\alpha)}\), and this contradicts that \(\langle
T_\alpha \mid \alpha < \lambda\rangle\) is a partition.

Since \(\nu\) is a limit of L\"owenheim-Skolem cardinals, \(G\) is
\(V_\nu\)-complete, and so \cref{LSUlam} implies that there is some \(\rho <
\lambda\) and a partition \(\langle A_\alpha \mid \alpha < \rho\rangle\) of
\(\eta\) into \(G\)-positive sets such that \(G\restriction A_\alpha\) is an
ultrafilter for all \(\alpha < \rho\).

The main claim is that \(G\cup F\) generates a proper filter. Granting the
claim, the proof is completed as follows. Let \(H\) be the filter generated by
\(G\cup F\). Since \(G\) and \(F\) are \(V_{\nu}\)-complete filters, given that
\(H\) is proper, in fact \(H\) is \(V_{\nu}\)-complete. In particular, for some
\(\alpha < \rho\), \(A_\alpha\) is \(H\)-positive. Let \(U = G\restriction
A_\alpha\), which is an ultrafilter by definition. Since \(H\restriction
A_\alpha\) is a proper filter and the ultrafilter \(U = G\restriction A_\alpha\)
is contained in \(H\restriction A_\alpha\), in fact, \(H\restriction A_\alpha =
U\). Since \(F\subseteq H\subseteq U\), \(U\) is a \(V_{\nu}\)-complete
extension of \(F\). This contradicts our choice of \(F\), and completes the
proof modulo the claim.

We finish by showing that \(G\cup F\) generates a proper filter. Suppose it does
not, so there is a set in \(G\) whose complement is in \(F\). Since \(G =
\mathcal F^{V_{\nu}}(\mathcal E,\eta)\), this means that there is some \(\beta <
\nu\) and a sequence \(\langle i_x \mid x\in V_{\beta}\rangle\subseteq \mathcal
E\) such that the set 
\[T = \bigcup \{\alpha < \eta \mid i_x'(\alpha) > \alpha\}\] belongs to \(F\).
Let \(j_x = i_x'\).

Fix \(x\in V_\beta\) for the rest of the paragraph. Since \(i_x\in \mathcal E\),
there is a filter \(F_x\) such that \(j_x(F_x) = F\). Moreover, \(j_x:\mathcal
H_{\epsilon+3}\to \mathcal H_{\epsilon+3}\) is \(\Sigma_{n_0}\)-elementary,
\(j_x(\nu) = \nu\), and \(j_x(\eta) = \eta\). It follows that \(F_x\) does not
extend to a \(V_\nu\)-complete ultrafilter: this is because \(j_x\) is
\(\Sigma_{n_0}\)-elementary and it is a \(\Sigma_{n_0}\)-expressible fact in
\(\mathcal H_{\epsilon+3}\) that \(j_x(F_x) = F\) does not extend to a
\(V_\nu\)-complete ultrafilter. 

Let \(D^x_\alpha\) denote the ultrafilter over \(\eta\) derived from \(j_x\)
using \(\alpha\). For \(\alpha\in T\), let 
\[D_\alpha = \bigcap \{D^x_\alpha \mid j_x(\alpha) > \alpha\}\] Thus for all
\(\alpha\in T\), \(D_\alpha\) is a countably complete filter and \(\alpha\in
D_\alpha\).

Notice that 
\begin{equation}\label{StratAmalgEqn}\bigcap_{x\in V_\beta} F_x\subseteq F\text{-lim}_{\alpha\in T}D_\alpha\end{equation}
The proof is a matter of unwinding the definitions. Fix \(A\in \bigcap_{x\in V_\beta} F_x\). For each \(x\in V_\beta\), let \(S_x = \{\alpha <\eta \mid A\in D^x_\alpha\}\). In other words, \(S_x = j_x(A)\), and so since \(A\in F_x\), \(S_x\in j_x(F_x) = F\). Let \(S = \bigcap_{x\in V_\beta} S_x\). Since \(F\) is \(V_{\nu}\)-complete, \(S\in F\). By definition, for \(\alpha \in S\cap T\), \(A\in \bigcap D^x_\alpha\subseteq D_\alpha\). Since \(S\cap T\in F\), this means that for \(F\)-almost all \(\alpha\), \(A\in D_\alpha\). In other words, \(A\in F\text{-lim}_{\alpha\in T}D_\alpha\), as desired.

Since \(F_x\) is \(V_\nu\)-complete for every \(x\in V_\beta\), \(\bigcap_{x\in V_\beta} F_x\) is a \(V_\nu\)-complete filter. Since \(\alpha\in D_\alpha\) for all \(\alpha\in T\), \cref{StratAmalgEqn} implies that \(\bigcap_{x\in V_\beta} F_x\sE F\). Since \(F\) is a minimal counterexample to the theorem, it follows that there is a \(V_\nu\)-complete ultrafilter \(W\) that extends \(\bigcap_{x\in V_\beta} F_x\). 

Recall that for every \(x\in V_\beta\), \(F_x\) does not extend to a \(V_\nu\)-complete ultrafilter. It follows that there is a set in \(W\) whose complement belongs to \(F_x\). Since \(\nu\) is a limit of L\"owenheim-Skolem cardinals, for some ordinal \(\gamma > \epsilon\), there is an elementary substructure \(X\prec V_\gamma\) with \(V_\beta\subseteq X\), \(\langle F_x \mid x\in V_\beta\rangle\in X\), \(W\in X\), and \(X\preceq^{*} V_{\zeta}\) for some \(\zeta < \nu\). Let \(S\) be the intersection of all \(W\)-large sets that belong to \(X\). Since \(X\preceq^{*} V_{\zeta}\), \(\zeta < \nu\), and \(W\) is \(V_{\nu}\)-complete, \(S\in W\). We claim that the complement of \(S\) belongs to \(\bigcap_{x\in V_\beta} F_x\). To see this, fix an \(x\in V_\beta\). There is a \(W\)-large set \(A\in X\) whose complement belongs to \(F_x\) since \(X\) is an elementary substructure of \(V_\gamma\) that contains \(F_x\) and \(W\). Since \(S\) is the intersection of all \(W\)-large sets in \(X\), \(S \subseteq A\). Hence the complement of \(S\) contains the complement of \(A\), and it follows that the complement of \(S\) belongs to \(F_x\). 

The existence of a set \(S\in W\) whose complement is in \(\bigcap_{x\in V_{\beta}} F_x\) contradicts that \(W\) extends \(\bigcap_{x\in V_\beta} F_x\). This contradiction proves the claim that \(G\cup F\) generates a proper filter, and thereby proves the theorem as explained above.
\end{proof}
\end{thm}
By a similar argument, we also have the following consequence of \(I_0\):
\begin{thm}[ZFC]\label{I0FilterExtensionThm}
Suppose there is an elementary embedding from \(L(V_{\lambda+1})\) to \(L(V_{\lambda+1})\) with critical point below \(\lambda\). Then in \(L(V_{\lambda+1})\), every \(\lambda^+\)-complete filter over an ordinal less than \(\theta_{\lambda+2}\) extends to a \(\lambda^+\)-complete ultrafilter.\qed
\end{thm}
\section{Consistency results}\label{ConsistencySection}
\subsection{Introduction}
In a groundbreaking recent development, Schlutzenberg \cite{SchlutzenbergI0} has proved the consistency of the existence of an elementary embedding from \(V_{\lambda+2}\) to \(V_{\lambda+2}\) relative to \(\text{ZF} + I_0\):
\begin{thm}[Schlutzenberg]\label{SchlutzenI0}
Assume \(\lambda\) is an even ordinal and 
\[j : L(V_{\lambda+1})\to L(V_{\lambda+1})\] is an elementary embedding
with \(\textnormal{crit}(j) < \lambda\). 
Let \(M = L(V_{\lambda+1})[j\restriction V_{\lambda+2}]\).
Then \(V_{\lambda+2}\cap M = V_{\lambda+2}\cap L(V_{\lambda+1})\).
Hence \(M\) satisfies that there is an elementary embedding from \(V_{\lambda+2}\) to \(V_{\lambda+2}\).
\end{thm}
It follows that the existence of an elementary embedding from \(V_{\epsilon+2}\) to \(V_{\epsilon+2}\) is equiconsistent with \(I_0\). Moreover, neither hypothesis implies that \(V_{\epsilon+1}^\#\) exists.

If \(\epsilon\) is even and \(\mathcal H_{\epsilon+2}\) satisfies the Collection Principle, every elementary embedding from \(V_{\epsilon+2}\) to \(V_{\epsilon+2}\) extends to a \(\Sigma_0\)-elementary embedding from \(V_{\epsilon+3}\) to \(V_{\epsilon+3}\). Therefore the existence of a \(\Sigma_1\)-elementary embedding from \(V_{\epsilon+3}\) to \(V_{\epsilon+3}\) is in some sense the first rank-to-rank axiom beyond an elementary embedding from \(V_{\epsilon+2}\) to \(V_{\epsilon+2}\). (Also see \cref{Equicon}.) We can prove the existence of sharps from this principle:
\begin{thm}
Suppose \(\epsilon\) is an even ordinal and there is a \(\Sigma_1\)-elementary embedding from \(V_{\epsilon+3}\) to \(V_{\epsilon+3}\). Then \(A^\#\) exists for every \(A\subseteq V_{\epsilon+1}\).
\end{thm}

Recall the following theorem:
\begin{prp}\label{I2Prp}
Suppose \(\lambda\) is a cardinal and there is an elementary embedding \(j: L(V_{\lambda})\to L(V_\lambda)\) such that \(\kappa_\omega(j) = \lambda\). Then \(V_{\lambda}^\#\) exists and for some \(\alpha < \lambda\), there is an elementary embedding from \(V_{\alpha}\) to \(V_{\alpha}\).
\end{prp}
We will prove the following somewhat unexpected equiconsistency at the \(V_{\lambda+2}\) to \(V_{\lambda+2}\) level, which shows that \cref{I2Prp} does not generalize to the other even levels:
\begin{repthm}{Equicon}
The following statements are equiconsistent over \textnormal{ZF}:
\begin{enumerate}[(1)]
\item For some \(\lambda\), there is a nontrivial elementary embedding from \(V_{\lambda+2}\) to \(V_{\lambda+2}\).
\item For some \(\lambda\), there is an elementary embedding from \(L(V_{\lambda+2})\) to \(L(V_{\lambda+2})\) with critical point below \(\lambda\).
\item There is an elementary embedding \(j\) from \(V\) to an inner model \(M\) that is closed under \(V_{\kappa_\omega(j)+1}\)-sequences.
\end{enumerate}
\end{repthm}
Combined with Schlutzenberg's theorem, all these principles are equiconsistent with
the existence of an elementary embedding from \(L(V_{\lambda+1})\) to \(L(V_{\lambda+1})\) with critical point below \(\lambda\). In particular, {\it the existence of an elementary embedding from \(L(V_{\lambda+1})\) to \(L(V_{\lambda+1})\) with critical point below \(\lambda\) is equiconsistent with the existence of an elementary embedding from \(L(V_{\lambda+2})\) to \(L(V_{\lambda+2})\) with critical point below \(\lambda\).}

We then turn to some long-unpublished work of the author. The following is technically an open question:
\begin{qst}
Does the existence of a nontrivial elementary embedding from \(V\) to \(V\) imply the consistency of ZFC + \(I_0\)? 
\end{qst}

Combining a forcing technique due to Woodin \cite{SEM} and the Laver-Cramer theory of inverse limits \cite{Cramer}, we provide the following partial answer:
\begin{repthm}{I0Con}
Suppose \(\lambda\) is an ordinal and there is a \(\Sigma_1\)-elementary embedding \(j :V_{\lambda+3}\to V_{\lambda+3}\) with \(\lambda = \kappa_\omega(j)\). Assume \(\textnormal{DC}_{V_{\lambda+1}}\). 
Then there is a set generic extension \(N\) of \(V\) such that \((V_\lambda)^N\) satisfies \(\textnormal{ZFC}+ I_0\).
\end{repthm}
In particular, in the presence of DC, the existence of a \(\Sigma_1\)-elementary embedding \(j :V_{\lambda+3}\to V_{\lambda+3}\) with \(\lambda = \kappa_\omega(j)\) implies the consistency of \(\text{ZFC}+I_0\).

We also briefly outline a proof of the following theorem:
\begin{repthm}{I0EquiCon}
The following statements are equiconsistent over \textnormal{ZF + DC}:
\begin{enumerate}[(1)]
\item For some \(\lambda\), \(\mathscr E(V_{\lambda+2})\neq \{\textnormal{id}\}\).
\item For some \(\lambda\), \(\lambda\)\textnormal{-DC} holds and \(\mathscr E(V_{\lambda+2})\neq \{\textnormal{id}\}\).
\item The Axiom of Choice\({}+I_0\).
\end{enumerate}
\end{repthm}
The equivalence of (2) and (3) is Schlutzenberg's Theorem.
\subsection{Equiconsistencies and sharps}
We begin with the equiconsistencies for embeddings of the even levels. Here we need some basic observations about ultrapowers assuming weak choice principles, which we will later apply to inner models of the form \(L(V_{\epsilon+1})[C]\), which satisfy these principles.
\begin{lma}\label{CollectionLos}
Suppose \(\epsilon\) is an even ordinal and \(M\) is an inner model containing \(V_{\epsilon+1}\). Suppose \(j : V_{\epsilon+2}^M\to V_{\epsilon+2}^M\) is a \(\Sigma_1\)-elementary embedding. Assume that for all relations \(R\subseteq V_{\epsilon+1}\times M\) in \(M\), there is some \(S\subseteq R\) in \(M\) such that \(\textnormal{dom}(S) = \textnormal{dom}(R)\) and, in \(M\), \(\text{ran}(S)\preceq^{*} V_{\epsilon+1}\). Let \(\mathcal U\) be the \(M\)-ultrafilter derived from \(j\) using \(j[V_{\epsilon}]\). Then the ultrapower of \(M\) by \(\mathcal U\) satisfies \L o\'s's Theorem. Moreover, if \(\mathcal U\in M\), then in \(M\), \(\textnormal{Ult}(M,\mathcal U)\) is closed under \(V_{\epsilon+1}\)-sequences.
\begin{proof}
To establish \L o\'s's Theorem, it suffices to show that if \(R\subseteq V_{\epsilon+1}\times M\) belongs to \(M\) and \(\text{dom}(R)\in \mathcal U\), has a {\it \(\mathcal U\)-uniformization} in \(M\), which is just a \(f\subseteq R\) in \(M\) such that \(\text{dom}(f)\in \mathcal U\). We can reduce to the case of relations on \(V_{\epsilon+1}\times V_{\epsilon+1}\). Given \(R\subseteq V_{\epsilon+1}\times M\), take \(S\subseteq R\) with \(\text{dom}(S) = \text{dom}(R)\) and \(\text{ran}(S)\preceq^{*} V_{\epsilon+1}\). Fix a surjection \(p : V_{\epsilon+1}\to \text{ran}(S)\). Let 
\[R' = \{(x,y) \mid (x,p(y))\in S\}\]
If \(g\) is a \(\mathcal U\)-uniformization of \(T\), then \(p\circ g\) is a \(\mathcal U\)-uniformization of \(S\), and therefore \(p\circ g\) is a \(\mathcal U\)-uniformization of \(R\).

Therefore fix \(R\subseteq V_{\epsilon+1}\times V_{\epsilon+1}\) in \(M\) with \(\text{dom}(R)\in \mathcal U\). We have that \(j[V_{\epsilon}]\in j(\text{dom}(R))\) by the definition of a derived ultrafilter. Note that \(\text{dom}(j(R)) = j(\text{dom}(R))\) by the \(\Sigma_1\)-elementarity of \(j\). (Here we extend \(j\) to act on \(R\), which is essentially an element of \(V_{\epsilon+2}\).) Therefore \(j[V_{\epsilon}]\in \text{dom}(j(R))\). 

Fix \(y\in V_{\epsilon+1}\) such that \((j[V_\epsilon],y)\in R\). Then the function \(f_y\) given by \cref{RepresentationDef}  has the property that \(y = j(f)(j[V_{\epsilon}])\) (by the proof of \cref{RepresentationThm}), but also \(f_y\in M\) since \(f_y\) is definable over \(V_{\epsilon+1}\) from \(y\). Let \(g = f\cap R\), so \(g\subseteq R\). Note that \(j(f)(j[V_{\epsilon}]) = y\) has the property that \((j[V_{\epsilon}],y)\in j(R)\), and hence \(j(g)(j[V_{\epsilon}])\) is defined and is equal to \(y\). In other words, \(j[V_\epsilon]\in j(\{x \in V_{\epsilon+1} \mid (x,g(x))\in R\})\), again using the \(\Sigma_1\)-elementarity of \(j\) on \(V_{\epsilon+2}\). This means that \(\text{dom}(g)\in \mathcal U\), so \(g\) is a \(\mathcal U\)-uniformization of \(R\) that belongs to \(M\).

We finally show that if \(\mathcal U\in M\), then \(\text{Ult}(M,\mathcal U)\) is closed under \(V_{\epsilon+1}\)-sequences in \(M\). We might as well assume \(V = M\), since what we are trying to prove is first-order over \(M\). Let \(N = \text{Ult}(V,\mathcal U)\). We cannot assume \(N\) is transitive, but we will abuse notation by identifying certain points in \(N\) with their extensions. 

We first show that every set \(X\in [N]^{V_{\epsilon+1}}\) is covered by a set \(Y\in N\) such that \(Y\preceq^{*} V_{\epsilon+1}\) in \(N\). Let \(p : V_{\epsilon+1}\to X\) be a surjection. Let \(R\subseteq V_{\epsilon+1}\times V\) be the relation defined by \(R(x,f)\) if \(p(x) = [f]_\mathcal U\). Take \(S\subseteq R\) such that \(\text{dom}(S) = \text{dom}(R)\) and \(\text{ran}(S)\preceq^{*} V_{\epsilon+1}\). Then 
\[X \subseteq \{j(f)([\text{id}]_{\mathcal U}) \mid f \in \text{ran}(S)\} \subseteq  \{g([\text{id}]_\mathcal U) \mid f\in j(S)\}\]
Let \(Y = \{g([\text{id}]_\mathcal U) \mid f\in j(S)\}\). Then \(X\subseteq Y\), \(Y\in N\), and \(Y\preceq^{*} V_{\epsilon+1}\) in \(N\).

Now we show that \(N\) is closed under \(V_{\epsilon+1}\)-sequences. It suffices to show that \([N]^{V_{\epsilon+1}}\subseteq N\). Fix \(X\in [N]^{V_{\epsilon+1}}\). Take \(Y\in N\) with \(X\subseteq Y\) and \(Y\preceq^{*} V_{\epsilon+1}\) in \(N\). Let \(q : V_{\epsilon+1} \to Y\) be a surjection that belongs to \(N\). Consider the set
\[A = \{x\in V_{\epsilon+1} \mid q(x)\in X\}\]
By \cref{RepresentationThm}, \(A\in N\). Hence \(q[A]= X\) belongs to \(N\). This finishes the proof.
\end{proof}
\end{lma}
To prove the wellfoundedness of the ultrapower seems to require a stronger hypothesis which is related to Schlutzenberg's results on ultrapowers using L\"owenheim-Skolem cardinals.
\begin{lma}\label{LowenheimWF}
Suppose \(\epsilon\) is an even ordinal and \(j : V_{\epsilon+2}\to V_{\epsilon+2}\) is a \(\Sigma_1\)-elementary embedding. Assume that every transitive set \(N\) containing \(V_{\epsilon+1}\) has an elementary substructure \(H\) containing \(V_{\epsilon+1}\) such that \(H\preceq^{*} V_{\epsilon+1}\). Let \(\mathcal U\) be the ultrafilter over \(V_{\epsilon+1}\) derived from \(j\) using \(j[V_{\epsilon}]\). Then \(\textnormal{Ult}(V,\mathcal U)\) is wellfounded.
\begin{proof}
Assume towards a contradiction that the lemma fails. Let \(\alpha\) be an ordinal greater than \(\epsilon\) such that \(V_{\alpha}\) is a \(\Sigma_4\)-elementary substructure of \(V\). Then \(\text{Ult}(V_\alpha,\mathcal U)\) is illfounded. Let \(H\) be an elementary substructure of \(V_\alpha\) containing \(V_{\epsilon+1}\) and \(\mathcal U\) such that \(H\preceq^{*} V_{\epsilon+1}\). (Take \(H = H'\cap V_\alpha\) where \(H'\) is an elementary substructure of a \(N = V_\alpha\cup \{V_{\alpha}\times \mathcal U\}\).) 

Let \(P\) be the Mostowski collapse of \(H\). Let \(\mathcal W = \mathcal U\cap P\). Since \(V_{\epsilon+1}\subseteq H\), \(\mathcal W\) is the image of \(\mathcal U\) under the Mostowski collapse map. Therefore by elementarity, \(\text{Ult}(P,\mathcal W)\) is illfounded. Note that there is a \(\Sigma_2\)-elementary embedding \(k : \text{Ult}(P,\mathcal W)\to j_\mathcal U(P)\) defined by \(k([f]_\mathcal W) = [f]_\mathcal U\). Therefore \(j_\mathcal U(P)\) is illfounded. Let \(E\subseteq V_{\epsilon+1}\times V_{\epsilon+1}\) be a wellfounded relation whose Mostowski collapse is \(P\). Then in \(\text{Ult}(V,\mathcal U)\), \(j_\mathcal U(E)\) has Mostowski collapse \(j_\mathcal U(P)\) since \L o\'s's Theorem holds by \cref{CollectionLos}. (Note that the L\"owenheim-Skolem hypothesis of this lemma is stronger than the collection hypothesis from \cref{CollectionLos}.)  It follows that \(j_\mathcal U(E)\) is illfounded. Since \(E\subseteq V_{\epsilon+1}\), \(j_\mathcal U(E) \cong j(E)\). Therefore \(j(E)\) is illfounded. (Here we must extend \(j\) slightly to act on binary relations.) This contradicts that \(j\) is a \(\Sigma_1\)-elementary embedding from \(V_{\epsilon+2}\) to \(V_{\epsilon+2}\), since such an embedding preserves wellfoundedness.
\end{proof}  
\end{lma}
The following lemma gives an example of a structure satisfying the hypotheses of \cref{CollectionLos} and \cref{LowenheimWF}:
\begin{lma}\label{InnerUltraLemma}
Suppose \(j : V_{\epsilon+2}\to V_{\epsilon+2}\) is a \(\Sigma_1\)-elementary embedding. Let \(\mathcal U\) be the ultrafilter over \(V_{\epsilon+1}\) derived from \(j\) using \(j[V_{\epsilon}]\). Then for any class \(C\), the ultrapower of \(L(V_{\epsilon+1})[C]\) by \(\mathcal U\) using functions in \(L(V_{\epsilon+1})[C]\) is wellfounded and satisfies \L o\'s's Theorem.
\begin{proof}
Let \(M = L(V_{\epsilon+1})[C].\) The L\"owenheim-Skolem hypothesis of \cref{LowenheimWF} holds inside \(M\) as an immediate consequence of the fact that \(M\) satisfies that every set is ordinal definable from parameters in \(V_{\epsilon+1}\cup \{C\cap M\}\). This yields Skolem functions \(\langle f_x \mid x\in V_{\epsilon+1}\rangle\) for any transitive structure \(N\): if \(\varphi(v_0,v_1)\) is a formula, \(f_x(\varphi,p)\) is the least \(a\in \text{OD}_{C\cap M,x}\) in the canonical wellorder of \(\text{OD}_{C\cap M,x}\) such that \(N\vDash \varphi(a,p)\). (Obviously this would work for any structure \(N\) in a countable language.) If \(V_{\epsilon+1}\subseteq N\), then closing under these Skolem functions, one obtains an elementary substructure \(H\prec N\) containing \(V_{\epsilon+1}\) such that \(H\preceq^{*} V_{\epsilon+1}\).
Since this hypothesis implies the collection hypothesis from \cref{CollectionLos}, \L o\'s's Theorem holds for the ultrapower in question.

For the proof of wellfoundedness, we would like to apply \cref{LowenheimWF} inside \(M\), but the problem arises that \(\mathcal U\cap M\) may not belong to \(M\). Note, however, that it suffices to show that \(\text{Ult}(M',\mathcal U\cap M')\) is wellfounded where \(M' = L(V_{\epsilon+1})[C,\mathcal U]\):
if the ultrapower \(i : M'\to \text{Ult}(N,\mathcal U\cap N)\) is wellfounded, then since \(\text{Ult}(M,\mathcal U\cap M)\) elementarily embeds into \(i(M)\) via the canonical factor map, \(\text{Ult}(M,\mathcal U\cap M)\) is wellfounded as well. Since \(M'\) is of the form \(L(V_{\epsilon+1})[C']\) for some class \(C'\) coding \(C\) and \(\mathcal U\), the previous paragraph yields that the hypothesis of \cref{LowenheimWF} holds inside \(M'\). Therefore \cref{LowenheimWF} yields the wellfoundedness of \(\text{Ult}(M',\mathcal U\cap M')\), which completes the proof.
\end{proof}
\end{lma}

\begin{thm}\label{Equicon}
The following theories are equiconsistent:
\begin{enumerate}[(1)]
\item For some \(\lambda\), there is a \(\Sigma_1\)-elementary embedding from \(V_{\lambda+2}\) to \(V_{\lambda+2}\).
\item For some \(\lambda\), there is an elementary embedding from \(V_{\lambda+2}\) to \(V_{\lambda+2}\).
\item For some \(\lambda\), there is an elementary embedding from \(L(V_{\lambda+2})\) to \(L(V_{\lambda+2})\) with critical point below \(\lambda\).
\item There is an elementary embedding \(j : V\to M\) where \(M\) is an inner model that is closed under under \(V_{\lambda+1}\)-sequences for \(\lambda = \kappa_\omega(j)\).
\end{enumerate}
\begin{proof}
Clearly each statement is implied by the next (except for (4)!), so it suffices to show that (1) implies that (4) holds in an inner model. Assume (1). Let \(\lambda\) be the least ordinal such that there is a \(\Sigma_1\)-elementary embedding \(j : V_{\lambda+2}\to V_{\lambda+2}\). Then \(\lambda\) is a limit ordinal. (1) still holds in \(L(V_{\lambda+2})[\mathcal U]\), and so applying the proof of \cref{InnerUltraLemma} and \cref{CollectionLos}, we obtain that \(L(V_{\lambda+2})[\mathcal U]\) satisfies (4). This proves the theorem.
\end{proof}
\end{thm}

The following is a proof, without requiring \(\lambda\)-DC or any choice principles, of \cite[Lemma 28]{SEM2}:
\begin{thm}\label{ChoicelessWdnThm}
Suppose \(\epsilon\) is an even ordinal, \(A, B\subseteq V_{\epsilon+1}\), \(A\in L(V_{\epsilon+1},B)\), and \(j : L(V_{\epsilon+1},B)\to L(V_{\epsilon+1},B)\) is an elementary embedding that fixes \(B\). Assume \[(\theta_{\epsilon+2})^{L(V_{\epsilon+1},A)} < (\theta_{\epsilon+2})^{L(V_{\epsilon+1},B)}\]
Then \(A^\#\) exists, and \(A^\#\in L(V_{\epsilon+1},B)\).
\begin{proof}
For \(D\subseteq V_{\epsilon+1}\), let \(M_D = L(V_{\epsilon+1},D)\) and let \(\theta_D = (\theta_{\epsilon+2})^{L(V_{\epsilon+1},D)} \). If \(D\in M_B\), let \(\mathcal U_D\) be the \(M_D\)-ultrafilter over \(V_{\epsilon+1}\) derived from \(j\) using \(j[V_{\epsilon}]\). Let \(j_D : M_D \to \text{Ult}(M_D,\mathcal U_D)\) be the ultrapower embedding.

Note that \(M_A\cap V_{\epsilon+2}\preceq^{*} V_{\epsilon+1}\) in \(M_B\). Therefore 
\[j\restriction M_A\cap V_{\epsilon+2}\in M_B\]
This yields that \(\mathcal U_A\in M_B\). Thus within \(M_B\), one can compute the ultrapower 
\[j_A : M_A \to \text{Ult}(M_A,\mathcal U_A)\]
In particular, \(j_A\restriction \theta_B\in M_B\).
By a standard argument, \(j_B\restriction \theta_B\notin M_B\). ({\it Sketch:} Assume not. Inside \(M_B\), compute first \(j_B\restriction L_{\theta_B}(V_{\epsilon+1})\), then \(\mathcal U_B\), and finally \(j_B : M_B\to M_B\). Now in \(M_B\), there is a definable embedding from \(V\) to \(V\), contradicting \cref{SuzukiThm}.)

Since \(j_A\restriction \theta_B\in M_B\) and \(j_B\restriction \theta_B\notin M_B\), it must be that \(j_A\restriction \theta_B\neq j_B\restriction \theta_B\). Let \(k : \text{Ult}(M_A,\mathcal U_A)\to j_B(M_A)\) be the factor embedding. Note that \(k\restriction V_{\epsilon+1}\) is the identity. Therefore \(j_A(A) = j_B(A) = j(A)\), and so by elementarity and wellfoundedness, \(\text{Ult}(M_A,\mathcal U_A) = j_B(M_A) = M_{j(A)}\). 
Since \(j_A\restriction \theta_B\neq j_B\restriction \theta_B\), \(k\) has a critical point, and \(\text{crit}(k) < \theta_B\). Clearly \(\text{crit}(k) > \epsilon\) since \(k\restriction V_{\epsilon+1}\) is the identity. 
Thus we have produced an elementary embedding
\[k : L(V_{\epsilon+1},j(A))\to L(V_{\epsilon+1},j(A))\]
with critical point between \(\epsilon\) and \(\theta_B\). This implies that \(j(A)^\#\) exists. Moreover, since \(j_B\restriction Z\in M_B\) for all transitive sets such that \(Z\preceq^{*}V_{\epsilon+1}\) in \(M_B\), the same holds true of \(k\). In particular, the normal \(M_{j(A)}\)-ultrafilter \(W\) on \(\text{crit}(k)\) derived from \(k\) belongs to \(M_B\). Here we use the Coding Lemma (\cref{InnerModelCodingThm}) to see that \[P(\text{crit}(k))\cap M_A\subseteq P(\text{crit}(k))\cap M_B\preceq^{*} V_{\epsilon+1}\] in \(M_B\).

Since \(W\) is \(V_{\epsilon+1}\)-closed (\cref{CompletenessDef}), it is easy to check that the ultrapower of \(M_{j(A)}\) by \(W\) satisfies \L o\'s's Theorem. This ultrapower is wellfounded since it admits a factor embedding into \(k\). The elementary embedding \[j_W : M_{j(A)}\to M_{j(A)}\]
is therefore definable over \(M_B\). Thus \(M_B\) satisfies that \(j(A)^\#\) exists. By elementarity, \(M_B\) satisfies that \(A^\#\) exists. By absoluteness, \(A^\#\) exists, and \(A^\#\in M_B\).
\end{proof}
\end{thm}

\begin{cor}\label{SmallThetaSharpCor}
Suppose \(\epsilon\) is an even ordinal and there is a \(\Sigma_1\)-elementary embedding from \(V_{\epsilon+2}\) to \(V_{\epsilon+2}\). Then \(A^\#\) exists for every \(A\subseteq V_{\epsilon+1}\) such that \[(\theta_{\epsilon+2})^{L(V_{\epsilon+1},A)} < \theta_{\epsilon+1}\]
\begin{proof}
Fix \(A\subseteq V_{\epsilon+1}\). By \cref{ProperFix}, there is a set \(B\subseteq V_{\epsilon+1}\) such that \(A\in L(V_{\epsilon+1},B)\), \(j(B) = B\), and \[(\theta_{\epsilon+2})^{L(V_{\epsilon+1},A)} < (\theta_{\epsilon+2})^{L(V_{\epsilon+1},B)}\]
Taking the ultrapower of \(L(V_{\epsilon+1},B)\) by the ultrafilter derived from \(j\), \cref{InnerUltraLemma} shows that one obtains an elementary embedding \(i : L(V_{\epsilon+1},B)\to L(V_{\epsilon+1},B)\) such that \(i(B) = B\). By \cref{ChoicelessWdnThm}, this implies that \(A^\#\) exists.
\end{proof}
\end{cor}

\begin{cor}\label{3SharpCor}
Suppose \(\epsilon\) is an even ordinal and there is a \(\Sigma_1\)-elementary embedding from \(V_{\epsilon+3}\) to \(V_{\epsilon+3}\). Then \(A^\#\) exists for every \(A\subseteq V_{\epsilon+1}\).
\begin{proof}
We claim that for all \(A\subseteq V_{\epsilon+1}\), \((\theta_{\epsilon+2})^{L(V_{\epsilon+1},A)} < \theta_{\epsilon+2}\). The corollary then follows by applying \cref{SmallThetaSharpCor}. To prove the claim, note that \(L(V_{\epsilon+1},A)\) satisfies that there is a sequence \(\langle f_\alpha \mid \alpha < \theta_{\epsilon+2}\rangle\) such that for all \(\alpha < \theta_{\epsilon+2}\), \(f_\alpha: V_{\epsilon+1}\to \alpha\) is a surjection; this is immediate from the fact that \(L(V_{\epsilon+1})\) satisfies that every set is ordinal definable from parameters in \(V_{\epsilon+1}\cup \{A\}\). By \cref{NotODThm}, this does not hold in \(V\), and therefore \((\theta_{\epsilon+2})^{L(V_{\epsilon+1},A)} < \theta_{\epsilon+2}\).
\end{proof}
\end{cor}
By a similar proof, we obtain the following consistency strength separation:
\begin{thm}\label{ConSepThm}
The existence of a \(\Sigma_1\)-elementary embedding from \(V_{\lambda+3}\) to \(V_{\lambda+3}\) implies the consistency of \textnormal{ZF} plus the existence of an elementary embedding from \(V_{\lambda+2}\) to \(V_{\lambda+2}\).
\end{thm}
This follows immediately from a more semantic fact:
\begin{prp}\label{ConSepPrp}
Suppose \(\epsilon\) is an even ordinal and there is a \(\Sigma_1\)-elementary embedding from \(V_{\epsilon+3}\) to \(V_{\epsilon+3}\). Then there is a set \(E\subseteq V_{\epsilon+1}\) and an inner model \(M\subseteq L(V_{\epsilon+1},E)\) containing \(V_{\epsilon+1}\) such that \(M\) satisfies that there is an elementary embedding from \(V_{\epsilon+2}\) to \(V_{\epsilon+2}\).
\end{prp}
\begin{proof}
Fix an elementary embedding \(j : V_{\epsilon+2}\to V_{\epsilon+2}\). Then the model \(M = L(V_{\epsilon+1})[j]\) satisfies that there is an elementary embedding from \(V_{\epsilon+2}^M\) to \(V_{\epsilon+2}^M\), namely \(j\restriction V_{\epsilon+2}^M\). This model also satisfies that there is a sequence \(\langle f_\alpha \mid \alpha < \theta_{\epsilon+2}\rangle\) such that for all \(\alpha < \theta_{\epsilon+2}\), \(f_\alpha: V_{\epsilon+1}\to \alpha\) is a surjection. Thus \((\theta_{\epsilon+2})^M < \theta_{\epsilon+2}\) by the proof of \cref{NotODThm}.  By condensation, this implies that \(V_{\epsilon+2}\cap M\preceq^{*} V_{\epsilon+1}\). Therefore 
\[V_{\epsilon+2}^M\cup \{j\restriction V_{\epsilon+2}^M\}\in \mathcal H_{\epsilon+2}\]
It follows that there is a wellfounded relation \(E\) on \(V_{\epsilon+1}\) whose Mostowski collapse is \(V_{\epsilon+2}^M\cup \{j\restriction V_{\epsilon+2}^M\}\). Hence \(M\subseteq L(V_{\epsilon+1},E)\), as desired.
\end{proof}
\begin{proof}[Proof of \cref{ConSepThm}]
By minimizing, we may assume \(\lambda\) is a limit ordinal. By \cref{ConSepPrp}, there is a set \(E\subseteq V_{\lambda+1}\) and an inner model \(M\subseteq L(V_{\lambda+1},E)\) containing \(V_{\lambda+1}\) such that \(M\) satisfies that there is an elementary embedding from \(V_{\lambda+2}\) to \(V_{\lambda+2}\). By \cref{3SharpCor}, \(E^\#\) exists. Therefore \(L(V_{\lambda+1},E)\) has a proper class of inaccessible cardinals. Fix an inaccessible \(\delta\) of \(L(V_{\lambda+1},E)\) such that \(\delta > \lambda\). Then \(M\cap V_\delta\) is a model satisfying \text{ZF} plus the existence of an elementary embedding from \(V_{\lambda+2}\) to \(V_{\lambda+2}\).
\end{proof}
\subsection{Forcing choice}
It is natural to wonder whether choiceless large cardinal axioms really are stronger than the traditional large cardinals in terms of the consistency hierarchy. Perhaps the situation is analogous to the status of the full Axiom of Determinacy in that traditional large cardinal axioms imply the existence of an inner model of ZF containing choiceless large cardinals. Could fairly weak traditional large cardinal axioms imply the consistency of the axioms we have been considering in this paper? Using the techniques of inner model theory, one can show that the choiceless cardinals imply the existence of inner models with many Woodin cardinals. But what about large cardinal axioms currently out of reach of inner model theory?

In fact, Woodin showed that one can prove that certain very large cardinals are equiconsistent with their choiceless analogs. For example:
\begin{thm}[Woodin]
The following theories are equiconsistent:
\begin{itemize}
	\item \textnormal{ZF + there is a proper class of supercompact cardinals}.
	\item \textnormal{ZFC + there is a proper class of supercompact cardinals}.
\end{itemize}
\end{thm}
In this context, we are using the following definition of a supercompact cardinal:
\begin{defn}
A cardinal \(\kappa\) is {\it supercompact} if for all \(\alpha\geq \kappa\), 
for some \(\beta \geq \alpha\) and some transitive set \(N\)
with \([N]^{V_\alpha}\subseteq N\),
there is an elementary embedding \(j : V_\beta\to N\)
such that \(\text{crit}(j) = \kappa\) and \(j(\kappa) > \alpha\).
\end{defn}
The proof shows that if there is a proper class of supercompact cardinals, there is a class forcing extension preserving all supercompact cardinals in which the Axiom of Choice holds. (Not every countable model of ZF is an inner model of a model of ZFC, since for example every inner model of a model of ZFC has a proper class of regular cardinals. More recently, Usuba showed that the existence of a proper class of L\"owenheim-Skolem cardinals suffices to carry out Woodin's forcing construction.)

In particular, this theorem implies that the existence of an elementary embedding from \(V_{\lambda}\) to \(V_{\lambda}\), in ZF alone, implies the consistency of the existence of a proper class of supercompact cardinals in ZFC. Indeed, the same ideas produce models of ZFC with many \(n\)-huge cardinals from the same hypothesis. But the question arises whether the weakest of the choiceless large cardinal axioms in fact implies the consistency (with ZFC) of all the traditional large cardinal axioms.

In this section, we combine Woodin's method of forcing choice and a reflection theorem due to Scott Cramer to prove the following theorem:
\begin{thm}
Over \textnormal{ZF + DC}, the existence of a \(\Sigma_1\)-elementary embedding from \(V_{\lambda+3}\) to \(V_{\lambda+3}\) implies \(\textnormal{Con}(\textnormal{ZFC} + I_0)\).\qed
\end{thm}
This will follow as an immediate consequence of \cref{I0Con} below.

We appeal to the following result due to Scott Cramer:
\begin{thm}[Cramer, \cite{Cramer}]\label{InverseLimitRefl}
Suppose \(\lambda\) is a cardinal, \(V_{\lambda+1}^\#\) exists, and there is a \(\Sigma_1\)-elementary embedding from \((V_{\lambda+1},V_{\lambda+1}^\#)\) to \((V_{\lambda+1},V_{\lambda+1}^\#)\). Assume \(\textnormal{DC}_{V_{\lambda+1}}\). Then there is a cardinal \(\bar \lambda < \lambda\) such that \(V_{\bar \lambda}\prec V_\lambda\) and there is an elementary embedding from \(L(V_{\bar \lambda+1})\) to \(L(V_{\bar \lambda+1})\) with critical point less than \(\bar \lambda\).\qed
\end{thm}
This uses the method of {\it inverse limit reflection}, which is the technique used to prove
reflection results at the level of \(I_0\). For smaller large cardinals,
reflection results are typically not very deep, and tend to require
no use of the Axiom of Choice. It is not clear, however, whether inverse limit reflection
can be carried out without the use of DC. This is the underlying reason
that DC is required as a hypothesis in \cref{I0Con}.

We also appeal to the following theorem of Woodin:
\begin{thm}[{Woodin, \cite[Theorem 226]{SEM}}]\label{WoodinForcing}
Suppose \(\delta\) is supercompact, \(\bar \lambda < \delta\) is such that \(V_{\bar \lambda} \prec V_\delta\), and \(j : V_{\bar \lambda+1}\to V_{\bar \lambda+1}\) is an elementary embedding. Then there is a weakly homogeneous partial order \(\mathbb P\subseteq V_{\bar \lambda+1}\) definable over \(V_{\bar \lambda + 1}\) without parameters, a condition \(p\in \mathbb P\), and a \(\mathbb P\)-name \(\dot{\mathbb Q}\) such that for any \(V\)-generic filter \(G\subseteq \mathbb P\) with \(p\in G\), 
the following hold:
\begin{itemize}
\item \(V_{\bar \lambda+1}[G] = V[G]_{\bar \lambda+1}\) and 
\(j[G]\subseteq G\).
\item \(V[G]\) satisfies \(\bar \lambda\)\textnormal{-DC}.
\item \(\mathbb Q = (\dot{\mathbb Q})_G\) is a \(\bar \lambda^+\)-closed partial order in 
\(V[G]_\delta\).
\item For any \(V[G]\)-generic filter \(H\subseteq \mathbb Q\), \(V[G][H]_\delta\) satisfies \textnormal{ZFC}.\qed
\end{itemize}
\end{thm}
Combining these two theorems, we show:
\begin{thm}\label{I0Con}
Suppose \(\lambda\) is an ordinal and there is a \(\Sigma_1\)-elementary embedding \(j :V_{\lambda+3}\to V_{\lambda+3}\) with \(\lambda = \kappa_\omega(j)\). Assume \(\textnormal{DC}_{V_{\lambda+1}}\). 
Then there is a set generic extension \(N\) such that for some \(\delta < \lambda\),
\((V_{\delta})^N\) satisfies \(\textnormal{ZFC}+ I_0\).
\end{thm}
\begin{proof}
By \cref{3SharpCor}, \(V_{\lambda+1}^\#\) exists. Since \(V_{\lambda+1}^\#\) is definable without parameters in \(V_{\lambda+2}\), any elementary embedding from \(V_{\lambda+2}\) to \(V_{\lambda+2}\) restricts to an elementary embedding from \((V_{\lambda+1},V_{\lambda+1}^\#)\) to \((V_{\lambda+1},V_{\lambda+1}^\#)\). Therefore the hypotheses of \cref{InverseLimitRefl} are satisfied. It follows that there is a cardinal \(\bar \lambda < \lambda\) such that \(V_{\bar \lambda}\prec V_\lambda\) and there is an elementary embedding from \(j : L(V_{\bar \lambda+1})\to L(V_{\bar \lambda+1})\) with critical point less than \(\bar \lambda\). We can assume (by the ultrapower analysis) that \(j\) is definable over \(V_\lambda\) from parameters in \(V_{\bar \lambda+2}\).

Now let \(\delta < \lambda\) be a supercompact cardinal of \(V_\lambda\) such that \(\delta > \bar \lambda\) and \(V_\delta \prec V_\lambda\). (If \(k : V_{\lambda}\to V_\lambda\) is elementary, then any point above \(\bar \lambda\) on the critical sequence of \(k\) will do.)
The embedding \(j : L(V_{\bar \lambda+1})\to L(V_{\bar \lambda+1})\) is 
definable from parameters in \(V_{\bar\lambda+2}\), so since \(V_\delta \prec V_\lambda\),
it then follows that \(j\) restricts to an elementary embedding from 
\(L_{\delta}(V_{\bar \lambda+1})\) to \(L_{\delta}(V_{\bar \lambda+1})\)
that is definable over \(L_{\delta}(V_{\bar \lambda+1})\).

The hypotheses of \cref{WoodinForcing} hold in \(V_\lambda\) (taking \(j\) equal to \(j\restriction V_{\bar \lambda+1}\)).
Let \(\mathbb P\), \(p\), and \(\dot {\mathbb Q}\) be as in \cref{WoodinForcing}
applied in \(V_\lambda\).
Let \(G\subseteq \mathbb P\) be \(V\)-generic with \(p\in G\) and \(H\subseteq (\dot {\mathbb Q})_G\) be \(V[G]\)-generic. 
We claim that \(V[G][H]_\delta\) satisfies ZFC + \(I_0\). The fact that \(V[G][H]_\delta\) satisfies ZFC is immediate from \cref{WoodinForcing} applied in \(V_\lambda\). (Here we use that
\(V[G][H]_\delta = V_\lambda[G][H]_\delta\), which follows from the fact that
\(\mathbb P*\dot{\mathbb Q}\in V_\lambda\).) 
Moreover \(j[G]\subseteq G\), \(\mathbb P\in L_\delta(V_{\bar \lambda+1})\), and \(j(\mathbb P) = \mathbb P\), so by standard forcing theory,
\(j\) extends to an elementary embedding from \(L_\delta(V_{\bar \lambda+1})[G]\) to \(L_\delta(V_{\bar \lambda+1})[G]\). Since \(V_{\bar \lambda+1}[G]= V[G]_{\bar \lambda+1} = V[G][H]_{\bar \lambda+1}\), it follows that \(j\) extends to an elementary embedding from \(L_\delta(V[G][H]_{\bar \lambda+1})\) to \(L_\delta(V[G][H]_{\bar \lambda+1})\). 
Therefore \(V[G][H]_\delta\) is a model of ZFC + \(I_0\), completing the proof.
\end{proof}

We finish by very briefly sketching the following equiconsistency:
\begin{thm}\label{I0EquiCon}
The following statements are equiconsistent over \textnormal{ZF + DC}:
	\begin{enumerate}[(1)]
		\item For some \(\lambda\), \(\mathscr E(V_{\lambda+2})\neq \{\textnormal{id}\}\).
		\item For some \(\lambda\), \(\lambda\)\textnormal{-DC} holds and \(\mathscr E(V_{\lambda+2})\neq \{\textnormal{id}\}\).
		\item The Axiom of Choice \textnormal{+ \(I_0\)}.
	\end{enumerate}
\end{thm}
The equiconsistency of (2) and (3) is due to Schlutzenberg.

The equiconsistency uses Schlutzenberg's Theorem (\cref{SchlutzenI0}) to reduce to the situation where inverse limit reflection \cite{Cramer} can be applied.
\begin{thm}\label{InverseDCLma}
Assume there is an embedding \(j \in \mathscr E(L(V_{\lambda+1}))\) with \(\lambda = \kappa_\omega(j)\). Assume \(\textnormal{DC}\) holds in \(L(V_{\lambda+1})\). Then for any infinite cardinal \(\gamma < \lambda\), if \(\gamma\textnormal{-DC}\) holds in \(V_\lambda\), then \(\gamma\textnormal{-DC}\) holds in \(L(V_{\lambda+1})\).
\begin{proof}
Assume \(\gamma\)-DC holds in \(V_\lambda\). By a standard argument, it suffices to show that \(\gamma\textnormal{-DC}_{V_{\lambda+1}}\) holds in \(L(V_{\lambda+1})\). Suppose \(T\) is a \(\gamma\)-closed tree on \(V_{\lambda+1}\) with no maximal branches. We must find a cofinal branch of \(T\). Fix \(\alpha < (\theta_{\lambda+2})^{L(V_{\lambda+1}}\) such that \(T\in L_\alpha(V_{\lambda+1})\). By inverse limit reflection \cite{Cramer}, there exist \(\gamma < \bar \lambda < \bar \alpha < \lambda\) and an elementary embedding \(J : L_{\bar \alpha}(V_{\bar \lambda+1})\to L_\alpha(V_{\lambda+1})\) with \(T\in \text{ran}(J)\). Let \(\bar T = J^{-1}(T)\). Working in \(V_\lambda\), \(\gamma\)-DC yields a cofinal branch \(\bar b\subseteq \bar T\). Since \(\bar b\) is a \(\gamma\)-sequence of elements of \(V_{\bar \lambda+1}\), \(\bar b \in L_1(V_{\bar \lambda+1})\). Therefore \(\bar \in L_{\bar \alpha}(V_{\bar \lambda+1})\). (We may assume without loss of generality that \(\bar \alpha \geq 1\).) Now \(J(\bar b)\) is a cofinal branch of \(T\), as desired.
\end{proof}
\end{thm}
\begin{cor}\label{DCCor}
Assume there is an embedding \(j \in \mathscr E(L(V_{\lambda+1}))\) with \(\lambda = \kappa_\omega(j)\). Assume \(\textnormal{DC}\) holds in \(L(V_{\lambda+1})\). Then for any infinite cardinal \(\gamma < \lambda\), if \(\gamma\textnormal{-DC}\) holds in \(V_\lambda\), then \(\gamma\textnormal{-DC}\) holds in \(L(V_{\lambda+1})[j\restriction V_{\lambda+2}]\).
\begin{proof}
Let \(M = L(V_{\lambda+1})[j\restriction V_{\lambda+2}]\). Again, it suffices to show \(\gamma\textnormal{-DC}_{V_{\lambda+1}}\) holds in \(M\). But by Schlutzenberg's Theorem, 
\(V_{\lambda+2}\cap M = V_{\lambda+2}\cap L(V_{\lambda+1})\), so
\(M\) satisfies \(\gamma\textnormal{-DC}_{V_{\lambda+1}}\) if and only if \(L(V_{\lambda+1})\) does. Applying \cref{InverseDCLma} then yields the corollary.
\end{proof}
\end{cor}

\begin{proof}[Proof of \cref{I0EquiCon}]
Assume (1).
We may assume \(V = L(V_{\lambda+1})[j]\) for a nontrivial embedding 
\(j\in \mathscr E(V_{\lambda+2})\) with \(\kappa_\omega(j) = \lambda\). 
We build a forcing extension satisfying (2).  
Let \(\langle \mathbb Q_\alpha \mid \alpha < \lambda\rangle\) 
be Woodin's class Easton iteration for forcing AC, as computed in \(V_\lambda\). 
(See \cite[Theorem 226]{SEM}.) Let \(\mathbb P\) be the inverse limit of the sequence 
\(\langle\mathbb Q_\alpha\rangle_{\alpha < \lambda}\), and let 
\(\dot{\mathbb P}_{\alpha,\lambda}\) be the factor forcing, so 
\(\mathbb Q_\alpha * \dot{\mathbb P}_{\alpha,\lambda} \cong \mathbb P\). 
By construction, there is an increasing sequence 
\(\langle \kappa_\alpha\rangle_{\alpha < \lambda}\) 
such that \(\mathbb Q_\alpha\) forces \(\kappa_\alpha\)-DC 
over \(V_{\lambda}\) and \(\mathbb P_{\alpha,\lambda}\) 
is \(\kappa_\alpha^+\)-closed in \(V\). But by \cref{DCCor}, 
\(\mathbb Q_\alpha\) forces \(\kappa_\alpha\)-DC over \(V\). 
Therefore using that \(\mathbb P_{\alpha,\lambda}\) is \(\kappa_\alpha^+\)-closed in \(V\), 
if \(G\subseteq \mathbb P\) is \(V\)-generic,
\(V[G\cap \mathbb Q_\alpha]\) is closed under \(\kappa_\alpha\)-sequences 
in \(V[G]\) and \(V[G]\) satisfies \(\kappa_\alpha\)-DC. 
Since the cardinals \(\kappa_\alpha\) increase to \(\lambda\), 
this shows that \(V[G]\) satisfies \(\lambda\)-DC.
Moreover the standard master condition argument for \(I_1\)-embeddings, 
given for example in \cite[Lemma 5.2]{HamkinsFragile}, 
shows that \(G\) can be chosen so that the embedding 
\(j\) lifts to an elementary embedding \(j^* : V[G]_{\lambda+2}\to V[G]_{\lambda+2}\).

The equiconsistency of (2) and (3) is Schlutzenberg's Theorem \cite{SchlutzenbergI0}.
\end{proof}
\section{Appendix}
In this appendix, we collect together the wellfoundedness proofs for the various Ketonen orders we have used throughout the paper. 
We take a more general approach by considering a Ketonen order on countably complete filters on complete Boolean algebras.
The orders we have considered so far belong to the special case where the Boolean algebras involved are atomic. 
In our view, the more abstract approach significantly clarifies the wellfoundedness proofs. 
For a more concrete approach, see the treatment in the author's thesis \cite{UA}.

\begin{defn}
Suppose \(\mathbb B_0\) and \(\mathbb B_1\) are complete Boolean algebras. A {\it \(\sigma\)-map} from \(\mathbb B_0\) to \(\mathbb B_1\) is a function that preserves \(0\), \(1\), and countable meets.
\end{defn}
We work with \(\sigma\)-maps rather than countably complete homomorphisms so that our results apply to the Ketonen order on filters in addition to the Ketonen order on ultrafilters: if \(\langle F_y \mid y \in Y\rangle\) is a sequence of countably complete filters over \(X\), then the function \(h : P(X)\to P(Y)\) defined by \(h(A) = \{y\in Y \mid A\in F_y\}\) is a \(\sigma\)-map, but is a countably complete homomorphism only if \(F_y\) is an ultrafilter for all \(y\in Y\). 

Given a \(\sigma\)-map from \(\mathbb B_0\) to \(\mathbb B_1\), one can define a \(\mathbb B_1\)-valued relation on names for ordinals in \(V^{\mathbb B_0}\) and \(V^{\mathbb B_1}\):
\begin{defn}\label{BooleanRelationDef}
Suppose \(\mathbb B_0\) and \(\mathbb B_1\) are complete Boolean algebras, \(\dot \alpha_0\in V^{\mathbb B_0}\) and \(\dot \alpha_1\in V^{\mathbb B_1}\) are names for ordinals, and \(h : \mathbb B_0\to\mathbb B_1\) is a \(\sigma\)-map. Then \[\llbracket\dot \alpha_0 < \dot \alpha_1\rrbracket_h = \bigvee_{\beta\in \text{Ord}} h\left(\llbracket \dot \alpha_0 < \beta\rrbracket_{\mathbb B_0}\right)\cdot \llbracket \dot \alpha_1 = \beta\rrbracket_{\mathbb B_1}\]
\end{defn}
Note that ``\(\dot \alpha_0 < \dot \alpha_1\)'' is not a formula in the forcing language associated to either \(\mathbb B_0\) or \(\mathbb B_1\). The notation should be regarded as purely formal.

This notation is motivated by the following considerations. Suppose \(h :\mathbb B_0\to \mathbb B_1\) is a complete homomorphism. Then there is an embedding \(i :  V^{\mathbb B_0}\to V^{\mathbb B_1}\) defined by \(i (\dot x) = h\circ \dot x\). In this case, \(\llbracket\dot \alpha_0 < \dot \alpha_1\rrbracket_h = \llbracket i(\dot \alpha_0) < \dot \alpha_1\rrbracket_{\mathbb B_1}\).
More generally, \[\llbracket\dot \alpha_0 < \dot \alpha_1\rrbracket_h = \llbracket h(\llbracket \dot \alpha_0 < \dot \alpha_1\rrbracket_{\mathbb B_0})\in \dot G_{\mathbb B_1}\rrbracket_{\mathbb B_1}\]
Recall that \(\dot G_\mathbb B\) denotes the canonical name for a generic ultrafilter in the forcing language associated with the complete Boolean algebra \(\mathbb B\).
Given our assertion above that ``\(\dot \alpha_0 < \dot \alpha_1\)'' is not a formula in the forcing language associated to \(\mathbb B_0\), the reader may want to take some time interpreting the right-hand side of the formula above.



The following lemma asserts a form of wellfoundedness for the relation given by \cref{BooleanRelationDef}:
\begin{lma}\label{BooleanWFLma}
Suppose \(\langle \mathbb B_n,h_{n,m}: \mathbb B_n \to \mathbb B_m \mid m \leq n < \omega\rangle\) is an inverse system of complete Boolean algebras and \(\sigma\)-maps. Suppose for each \(n < \omega\), \(\dot \alpha_n\) is a \(\mathbb B_n\)-name for an ordinal. Then 
\(\bigwedge_{n < \omega}h_{n,0}(\llbracket\dot \alpha_{n+1} < \dot \alpha_n\rrbracket_{h_{n+1,n}}) = 0\).
\begin{proof}
Assume towards a contradiction that the lemma is false. Let \(\beta\) be the least ordinal such that for some \(\langle \mathbb B_n,h_{n,m}, \dot \alpha_n: m \leq n < \omega\rangle\) witnessing the failure of the lemma,
 \(\llbracket \dot \alpha_0 =\beta\rrbracket_{\mathbb B_0}\cdot \bigwedge_{n < \omega}h_{n,0}(\llbracket\dot \alpha_{n+1} < \dot \alpha_n\rrbracket_{h_{n+1,n}})\neq 0\). 

The definition of \(\llbracket\dot \alpha_{1} < \dot \alpha_0\rrbracket_{h_{1,0}}\) yields:
\begin{align}\llbracket \dot \alpha_0 =\beta\rrbracket_{\mathbb B_0}\cdot \llbracket\dot \alpha_{1} < \dot \alpha_0\rrbracket_{h_{1,0}}
&= \llbracket \dot \alpha_0 =\beta\rrbracket_{\mathbb B_0}\cdot h_{1,0}(\llbracket \dot \alpha_1 < \beta\rrbracket_{\mathbb B_1})\nonumber\\
&\leq h_{1,0}(\llbracket \dot \alpha_1 < \beta\rrbracket_{\mathbb B_1})\label{KeyKetIneq}\end{align}
For each \(m < \omega\), let 
\(a_m = \bigwedge_{m\leq n < \omega}h_{n,m}(\llbracket\dot \alpha_{n+1} < \dot \alpha_n\rrbracket_{h_{n+1,n}})\).
Since \(h_{1,0}\) is a \(\sigma\)-map, 
\[a_0 = \llbracket \dot \alpha_{1} < \dot \alpha_0\rrbracket_{h_{1,0}}\cdot h_{1,0}(a_1)\]
As a consequence of this and \cref{KeyKetIneq}: 
\begin{align*}
\llbracket \dot \alpha_0 =\beta\rrbracket_{\mathbb B_0}\cdot a_0 
&= \llbracket \dot \alpha_0 =\beta\rrbracket_{\mathbb B_0}\cdot \llbracket \dot \alpha_1 < \dot \alpha_0\rrbracket_{h_{1,0}} \cdot h_{1,0}(a_1)\\
&\leq h_{1,0}(\llbracket \dot \alpha_1 < \beta\rrbracket_{\mathbb B_1})\cdot h_{1,0}(a_1)\\
&= h_{1,0}(\llbracket \dot \alpha_1 < \beta\rrbracket_{\mathbb B_1}\cdot a_1)
\end{align*}
By our choice of \(\beta\), \(\llbracket \dot \alpha_0 =\beta\rrbracket_{\mathbb B_0}\cdot a_0 \neq 0\), 
so we can conclude that \(\llbracket \dot \alpha_1 < \beta\rrbracket_{\mathbb B_1}\cdot a_1\neq 0\).
Therefore there is some \(\xi < \beta\) such that \(\llbracket \dot \alpha_1 = \xi\rrbracket_{\mathbb B_1}\cdot a_1\neq 0\). This contradicts the minimality of \(\beta\).
\end{proof}
\end{lma}
\begin{defn}\label{SigmaReductions}
Suppose \(F_0\) and \(F_1\) are countably complete filters on the complete Boolean algebras \(\mathbb B_0\) and \(\mathbb B_1\). A {\it \(\sigma\)-reduction} \(h : F_0\to F_1\) is a \(\sigma\)-map \(h : \mathbb B_0\to \mathbb B_1\) such that \(F_0 \subseteq h^{-1}[F_1]\). Suppose \(\dot \alpha_0\in V^{\mathbb B_0}\) and \(\dot \alpha_1\in V^{\mathbb B_1}\) are names for ordinals. A {\it \(\sigma\)-comparison} \(h : (F_0,\dot \alpha_0)\to (F_1,\dot \alpha_1)\) is a \(\sigma\)-reduction \(h : F_0\to F_1\) such that \(\llbracket \dot \alpha_0 < \dot \alpha_1\rrbracket_h\in F_1\). 
\end{defn}
\begin{thm}\label{SigmaDescentThm}
There is no infinite sequence of \(\sigma\)-comparisons and countably complete filters of the form
\(\cdots \stackrel{h_{3,2}}\longrightarrow(F_2,\dot \alpha_2)\stackrel{h_{2,1}}\longrightarrow(F_1,\dot \alpha_1)\stackrel{h_{1,0}}\longrightarrow(F_0,\dot \alpha_0)\).
\begin{proof}
Assume towards a contradiction that there is such a sequence. Fix \(n < \omega\). Since \(h_{n+1,n} : (F_{n+1},\dot \alpha_{n+1})\to (F_{n},\dot \alpha_{n})\) is a \(\sigma\)-comparison, \[\llbracket \dot \alpha_{n+1} < \dot \alpha_n\rrbracket_{h_{n+1,n}}\in F_n\]
Let \(h_{n,0} = h_{1,0} \circ \cdots \circ h_{n,n-1}\). Clearly \(h_{n,0} : F_n\to F_0\) is a \(\sigma\)-reduction, and therefore 
\[h_{n,0}(\llbracket \dot \alpha_{n+1} < \dot \alpha_n\rrbracket_{h_{n+1,n}})\in F_0\]
Since \(F_0\) is a countably complete filter, 
\(\bigwedge_{m < \omega } h_{m,0}(\llbracket \dot \alpha_{m+1} < \dot \alpha_m\rrbracket_{h_{m+1,m}})\in F_0\).
In particular, this infinite meet is not \(0\), contrary to \cref{BooleanWFLma}.
\end{proof}
\end{thm}
We now use this to prove the wellfoundedness of the Ketonen order and the irreflexivity of the internal relation. This is a matter of specializing the theorems we have proved to the case of atomic Boolean algebras.
\begin{defn}\label{NormedKetonenDef}
A {\it pointed filter} on a set \(X\) is a pair \((F,f)\) where \(F\) is a countably complete filter over \(X\) and \(f : X\to \textnormal{Ord}\) is a function. 
\end{defn}

Every function \(f: X\to \text{Ord}\) can be associated to the \(P(X)\)-name \(\tau_f\) for the ordinal 
\(f(x_G)\) where \(x_G\in X\) is the point in \(X\) selected by the generic (i.e., principal) ultrafilter \(G\subseteq P(X)\).
More concretely, the name \(\tau_f\) is defined by setting \(\text{dom}(\tau_f) = \bigcup_{x\in X} f(x)\) and \[\tau_f(\alpha) = \{x\in X \mid \alpha < f(x)\}\]
for all \(\alpha\in \text{dom}(\tau_f)\). Identifying \((F,f)\) and \((F,\tau_f)\), \cref{SigmaReductions} is transformed as follows:

\begin{defn}Suppose \((F,f)\) and \((G,g)\) are pointed filters over \(X\) and \(Y\) and \(Z = \langle Z_y \mid y\in Y\rangle\) is a sequence of countably complete filters over \(X\). 
\begin{itemize}
\item \(Z\) is a {\it filter reduction from \(F\) to \(G\)} if \(F = G\text{-lim}_{y\in Y}Z_y\).
\item \(Z\) is a {\it filter comparison from \((F,f)\) to \((G,g)\)} if \(Z\) is a filter reduction from \(F\) to \(G\) and for \(G\)-almost all \(y\in Y\), for \(Z_y\)-almost all \(x\in X\), \(f(x) < g(y)\). 
\end{itemize}
We write \(Z : F\to G\) to indicate that \(Z\) is a filter reduction from \(F\) to \(G\). We write \(Z: (F,f)\to (G,g)\) to indicate that \(Z\) is a filter comparison from \((F,f)\) to \((G,g)\).
\end{defn}
Clearly, \(\sigma\)-reductions and \(\sigma\)-comparisons generalize filter reductions and filter comparisons. Let us state this more precisely.
\begin{defn}
Let \(\Phi\) be the function sending a \(\sigma\)-map \(h : P(X)\to P(Y)\) to the sequence \(\langle Z_y \mid y\in Y\rangle\) where \(Z_y = \{A\subseteq X \mid y\in h(A)\}\) is the filter over \(X\) derived from \(h\) using \(y\).
\end{defn}

\begin{lma}
Suppose \((F,f)\) and \((G,g)\) are pointed filters over \(X\) and \(Y\). Suppose \(h : P(X)\to P(Y)\) is a \(\sigma\)-map. Then \(\Phi(h)\) is a filter reduction from \(F\) to \(G\) if and only if \(h\) is a \(\sigma\)-reduction from \(F\) to \(G\). Moreover \(\Phi(h)\) is a filter comparison from \((F,f) \) to \((G,g)\) if and only if \(h\) is a \(\sigma\)-comparison from \((F,\tau_f)\) to \((G,\tau_g)\).\qed
\end{lma}

As an immediate corollary of these lemmas and \cref{SigmaDescentThm}, we have the following theorems:
\begin{thm}\label{KetonenWellfounded}
There is no descending sequence of pointed filters and filter comparisons of the form \(\cdots \stackrel{Z_3}{\longrightarrow} (F_2,f_2)\stackrel{Z_2}{\longrightarrow} (F_1,f_1)\stackrel{Z_1}{\longrightarrow}(F_0,f_0)\).\qed
\end{thm}

Of course, the Ketonen order on filters can be characterized in terms of the notion of a filter reduction:
\begin{lma}
Suppose \(F_0\) and \(F_1\) are countably complete filters over ordinals. Then \(F_0 \sE F_1\) in the Ketonen order on filters if and only if there is a \(\sigma\)-comparison from \((F_0,\textnormal{id})\) to \((F_1,\textnormal{id})\).\qed
\end{lma}

\begin{thm}[DC]\label{KFilterWF}
The Ketonen order on filters is wellfounded.\qed
\end{thm}

Whenever \(U\sE W\) in the Ketonen order on ultrafilters, \(U\sE W\) in the Ketonen order on filters. Therefore \cref{KFilterWF} implies:
\begin{thm}[DC]\label{KetonenWF}
The Ketonen order on ultrafilters is wellfounded.\qed
\end{thm}

We finally prove the irreflexivity of the internal relation.
\begin{lma}\label{EasyInternal}
Suppose \(U\) and \(W\) are countably complete ultrafilters over sets \(X\) and \(Y\). 
Suppose \(\langle Z_y \mid y\in Y\rangle\) is a sequence of countably complete ultrafilters witnessing \(U\I W\).
Suppose \(\kappa\) is an ordinal and \(g : Y\to \kappa\) is a function such that for any \(\alpha < \kappa\),
\(g(y) > \alpha\) for \(W\)-almost all \(y\in Y\).
Then for any function \(f : X\to \kappa\),  \(Z: (U,f)\to(W,g)\).
\begin{proof}
Let \(\langle U_y \mid y\in Y\rangle\) witness that \(U\I W\). Then easily
\[U = W\text{-}\lim_{y\in Y}Z_y\]
So \(\langle Z_y \mid y\in Y\rangle\) is an ultrafilter reduction from \(U\) to \(W\).
Fix a function \(f : X\to \kappa\).

We must now verify that for \(W\)-almost all \(y\in Y\), for \(Z_y\)-almost all \(x\in X\), \(f(x) < g(y)\).
Since \(\langle Z_y \mid y\in Y\rangle\) witnesses \(U\I W\), it suffices to show that
for \(U\)-almost all \(x\in X\), for \(W\)-almost all \(y\in Y\), \(f(x) < g(y)\).
But this is a trivial consequence of our assumption on \(g\) and \(W\).
\end{proof}
\end{lma}

\begin{cor}\label{InternalIrrefl}
Suppose \(U\) is a countably complete ultrafilter such that \(j_U\) has a critical point.
Then \(U\not \I U\).
\begin{proof}
Let \(X\) be the underlying set of \(U\). The fact that \(j_U\) has a critical point \(\kappa\) implies that there is a function \(g : X\to \kappa\) such that for any \(\alpha < \kappa\),
\(g(x) > \alpha\) for \(U\)-almost all \(x\in X\). Assume \(U\I U\). Then by \cref{EasyInternal}, there is
an ultrafilter comparison from \((U,g)\) to \((U,g)\). This contradicts \cref{KetonenWellfounded}.
\end{proof}
\end{cor}
\bibliography{Bibliography.bib}
\bibliographystyle{unsrt}
\end{document}